\documentclass{article}

\usepackage[utf8]{inputenc}
\usepackage{amsmath, amssymb, amsthm, bm, mathrsfs}
\usepackage[top=1.5cm, bottom=1.5cm, left=2cm, right=2cm]{geometry}
\usepackage{csquotes, authblk, cases, enumitem}
\usepackage{graphicx, float, caption, subcaption, adjustbox}
\usepackage[hidelinks]{hyperref}
\usepackage{cleveref}
\usepackage[numbers,sort]{natbib}
\usepackage[section]{algorithm}
\usepackage{algpseudocode}
\usepackage{array, tabularx, xcolor, colortbl}
\usepackage{color, xcolor}
\usepackage[para]{footmisc}
\usepackage[normalem]{ulem}
\counterwithin{figure}{section}
\counterwithin{table}{section}
\numberwithin{equation}{section}
\graphicspath{{Figures/}}
\definecolor{darkred}{rgb}{0.8,0,0}

\DeclareMathOperator{\Span}{\operatorname{span}}

\DeclareMathOperator{\sinc}{\operatorname{sinc}}
\DeclareMathOperator{\diag}{\operatorname{diag}}

\DeclareMathOperator{\supp}{\operatorname{supp}}
\DeclareMathOperator{\sparsity}{\operatorname{sp}}

\newenvironment{frcseries}{\fontfamily{frc}\selectfont}{}

\theoremstyle{definition}
\newtheorem{definition}{Definition}[section]
\newtheorem{remark}[definition]{Remark}
\newtheorem{example}[definition]{Example}
\newtheorem{theorem}[definition]{Theorem}
\newtheorem{lemma}[definition]{Lemma}
\newtheorem{corollary}[definition]{Corollary}

\title{Mass lumping and stabilization for immersogeometric analysis}
\author[1]{Yannis Voet \thanks{yannis.voet@epfl.ch}}
\author[1]{Espen Sande \thanks{espen.sande@epfl.ch}}
\author[1]{Annalisa Buffa \thanks{annalisa.buffa@epfl.ch}}
\affil[1]{\small MNS, Institute of Mathematics, École polytechnique fédérale de Lausanne, Station 8, CH-1015 Lausanne, Switzerland}
\date{\today}

\begin{document}

\maketitle

\begin{abstract}
Trimmed (multi-patch) geometries are the state-of-the-art technology in computer-aided design for industrial applications such as automobile crashworthiness. In this context, fast solution techniques extensively rely on explicit time integration schemes in conjunction with mass lumping techniques that substitute the consistent mass with a (usually diagonal) approximation. For smooth isogeometric discretizations, Leidinger \cite{leidinger2020explicit} first showed that mass lumping removed the dependency of the critical time-step on the size of trimmed elements. This finding has attracted considerable attention but has unfortunately overshadowed another more subtle effect: mass lumping may disastrously impact the accuracy of low frequencies and modes, potentially inducing spurious oscillations in the solution. In this article, we provide compelling evidence for this phenomenon and later propose a stabilization technique based on polynomial extensions that restores a level of accuracy comparable to boundary-fitted discretizations.

\end{abstract}

\section{Introduction and background}
Isogeometric analysis (IGA) is a spline-based discretization technique for solving partial differential equations (PDEs) \cite{hughes2005isogeometric,cottrell2009isogeometric}. Originally conceived as an extension of classical finite element methods (FEM), it offers exclusive advantages such as greater smoothness yielding superior approximation properties \cite{bazilevs2006isogeometric,bressan2019approximation,sande2020explicit} and exact representation of common geometries by employing spline functions from computer-aided design (CAD), such as B-splines and Non-Uniform Rational B-splines (NURBS). The success of IGA has been overwhelming, ranging from fluid dynamics \cite{tagliabue2014isogeometric,hsu2015dynamic} and biomedical applications \cite{morganti2015patient,lorenzo2019computer} to structural \cite{cottrell2006isogeometric,cottrell2007studies,hughes2014finite} and fracture mechanics \cite{borden2014higher}. In particular, it is widely praised in structural analysis for removing the so-called ``optical branches'' from the discrete spectrum \cite{cottrell2006isogeometric,cottrell2007studies}. Nevertheless, a few inaccurate ``outlier'' eigenvalues persist and actually diverge with polynomial and mesh refinement \cite{gallistl2017stability}, thereby severely constraining the step size of explicit time integration methods (referred to as the Courant–Friedrichs–Lewy (CFL) condition). One classic example is the critical time step of the (undamped) central difference method
\begin{equation}
\label{eq: CFL_central_difference}
    \Delta t_c= \frac{2}{\omega_n}
\end{equation}
where $\omega_n$ is the largest discrete frequency of the system \cite{hughes2012finite,bathe2006finite}. Explicit time integration is by far the most popular method in structural analysis, especially for wave propagation problems such as blasts and impacts where short term effects are prominent. Contrary to implicit time integration methods, explicit ones only require solving linear systems with the mass matrix and the latter is commonly substituted with a diagonal lumped mass matrix, thereby drastically speeding up the solution process \cite{hughes2012finite,bathe2006finite}.

Mass lumping is nearly as old as the finite element method itself and accurate techniques based on nodal quadrature have been known ever since the 1970s \cite{fried1975finite,cohen1994higher} and sometimes connect to the algebraic row-sum \cite{hughes2012finite} and diagonal scaling (or HRZ) \cite{hinton1976note} methods for the spectral element method \cite{duczek2019mass}. Additionally, mass lumping often increases the critical time step, a property formally proven for the row-sum technique and nonnegative matrices \cite{voet2023mathematical}. However, the row-sum technique is only second order accurate for IGA, independently of the spline order \cite{cottrell2006isogeometric}, even for generalizations thereof \cite{voet2023mathematical,voet2024mass} and unfortunately, extending the nodal quadrature method to IGA is not immediate due to the non-interpolatory nature of the basis functions. While restoring interpolation recovers some accuracy, it comes at the unaffordable price of globally supported functions, unsuited for large scale simulations \cite{li2022significance,li2024interpolatory}. Instead, several authors have suggested combining (approximate) dual functions with the row-sum technique in a Petrov-Galerkin framework \cite{anitescu2019isogeometric,nguyen2023towards,hiemstra2023higher}. Although the method delivers greater accuracy, several technical difficulties still hold it back. Most notably, the support of (approximate) dual functions is larger than the standard B-splines and (approximate) duality only holds in the parametric domain. Rescaling the dual functions by the Jacobian determinant, as suggested by the authors, complicates the assembly of the stiffness matrix and the overall computational savings remain unclear. Moreover, the method is currently limited to rather simple academic examples and the authors do not offer any clear strategy for extending it to more realistic settings such as trimmed geometries.

Due to the aforementioned technical issues, practitioners still heavily rely on the simpler row-sum or diagonal scaling methods, which are more general but significantly less accurate for IGA \cite{cottrell2006isogeometric,adam2015stable}. As a matter of fact, despite its limited accuracy, the row-sum technique still represents the state-of-the-art lumping technique in commercial software packages, such as LS-DYNA \cite{hallquist2005ls}. Those computer programs support highly complex industrial CAD models, sometimes made up of several hundred or even thousands of trimmed NURBS patches \cite{leidinger2019explicit}. Trimming is a ubiquitous boolean operation whereby parts of a geometry are merged or discarded. While trimming changes the visualization of the model, it does not change its mathematical description. Within individual patches, trimming curves may create arbitrarily small trimmed elements, causing severe stability and conditioning issues for the underlying discretization \cite{marussig2017stable,de2017condition,de2019preconditioning,de2023stability}. As a matter of fact, the largest generalized eigenvalue for a consistent mass approximation may even become unbounded, preventing explicit time integration unless some form of stabilization \cite{stoter2023critical,eisentrager2024eigenvalue} or hybrid implicit-explicit (IMEX) schemes \cite{fassbender2024implicit} are used.

For isogeometric discretizations, Leidinger \cite{leidinger2020explicit} apparently found another remarkably simple way of preventing this effect: by lumping the mass matrix, the largest eigenvalue suddenly becomes bounded for sufficiently smooth discretizations. This fact was later independently verified in several numerical studies \cite{messmer2021isogeometric,messmer2022efficient,coradello2021accurate,stoter2023critical,radtke2024analysis} and only recently investigated theoretically in \cite{bioli2024theoretical}. While the analysis therein confirmed the numerical observations, it also revealed a much more subtle effect: although the largest eigenvalue may indeed become bounded, the smallest eigenvalue now converges to zero as the trimmed elements get smaller and the polynomial degree increases. To the best of our knowledge, this effect has largely been overlooked and its impact on the accuracy of the solution still remains unclear.

Our article will provide greater insight and is structured as follows: After recalling some standard notations and summarizing early findings in \Cref{se: model_problem}, we provide compelling evidence in \Cref{se: motivation} showing the disastrous effect mass lumping may have on the solution. Explanations, some rigorous, others more intuitive, are later laid out in \Cref{se: analysis} and a solution based on polynomial extension techniques is described in \Cref{se: stabilization} for improving the accuracy. This solution, covering both smooth isogeometric discretizations and standard finite elements, is applied in \Cref{se: experiments} to cure the examples of \Cref{se: motivation} and restore a level of accuracy and CFL condition comparable to boundary-fitted discretizations. Finally, conclusions are drawn in \Cref{se: conclusion}.


\section{Model problem and discretization}
\label{se: model_problem}
In a nutshell, structural dynamics studies how objects or structures deform over time in response to a dynamic loading. Typical examples include wave propagation, plate and beam bending or buckling and solid deformations \cite{hughes2012finite,bathe2006finite,zienkiewicz2005finite,belytschko2014nonlinear}. From a mathematical perspective, these processes are often described by hyperbolic PDEs featuring a second order time derivative and second (or higher order) spatial derivatives. The prototypical example is the wave equation, which will serve as model problem. Let $\Omega \subset \mathbb{R}^d$ be an open connected physical domain with Lipschitz boundary $\partial \Omega = \overline{\partial \Omega_D \cup \partial \Omega_N}$ where $\partial \Omega_D \cap \partial \Omega_N = \emptyset$. The physical domain $\Omega \subset \widehat{\Omega}$ is embedded in a fictitious (or ambient) domain $\widehat{\Omega}$. Let $[0,T]$ be the time domain with $T>0$ denoting the final time. We look for $u \colon \Omega \times [0,T] \to \mathbb{R}$ such that 
\begin{align}
 \partial_{tt} u(\mathbf{x},t)-\Delta u(\mathbf{x},t) &=f(\mathbf{x},t) & &\text{ in } \Omega \times (0,T], \label{eq: wave_equation} \\
 u(\mathbf{x},t)&=0 & &\text{ on } \partial \Omega_D \times (0,T], \label{eq: dirichlet_bc} \\
 \partial_n u(\mathbf{x},t)&=h(\mathbf{x},t) & &\text{ on } \partial \Omega_N \times (0,T], \label{eq: neumann_bc} \\
 u(\mathbf{x},0)&=u_0(\mathbf{x}) & &\text{ in } \Omega,  \label{eq: initial_disp} \\
 \partial_t u(\mathbf{x},0)&=v_0(\mathbf{x}) & &\text{ in } \Omega. \label{eq: initial_vel}
\end{align}
The PDE described by \cref{eq: wave_equation} is supplemented with Dirichlet (\cref{eq: dirichlet_bc}) and Neumann (\cref{eq: neumann_bc}) boundary conditions and initial conditions (\cref{eq: initial_disp,eq: initial_vel}). For the sake of the presentation, we only consider homogeneous Dirichlet boundary conditions and assume that $\partial \Omega_D \subset \partial \widehat{\Omega}$. Letting $V \subseteq H^1(\Omega)$ and defining the Bochner space
\begin{equation*}
    V_T = \{v \in L^2((0,T); V) \colon \ \partial_t v \in L^2((0,T); L^2(\Omega)), \ \partial_{tt} v \in L^2((0,T); V^*) \},
\end{equation*}
the weak form of the problem reads: find $u \in V_T$ such that for almost every time $t \in (0,T)$
\begin{align}
 \langle\partial_{tt} u(t), v\rangle + (\nabla u(t), \nabla v)_{L^2(\Omega)^d} &= (f(t),v)_{L^2(\Omega)} + (h(t),v)_{L^2(\partial \Omega_N)} & \forall v \in V, \label{eq: continuous_weak_form} \\
 (u(0),v)_{L^2(\Omega)} &= (u_0,v)_{L^2(\Omega)} & \forall v \in V, \nonumber \\
 \langle \partial_t u(0),v \rangle &= \langle v_0,v \rangle & \forall v \in V, \nonumber
\end{align}
where we assume that $f \in L^2((0,T); L^2(\Omega))$, $h \in L^2((0,T); L^2(\partial \Omega_N))$, $u_0 \in L^2(\Omega)$ and $v_0 \in V^*$ and $\langle .,.\rangle$ denotes the duality pairing between $V^*$ and $V$ (see e.g. \cite{evans2022partial}). For a standard Galerkin discretization of the spatial variables, one seeks an approximate solution $u^h(.,t)$ in a finite dimensional subspace $V^h \subset V$. In isogeometric analysis, $V^h = \mathcal{S}_{h,p,k}$ is a spline space parameterized by a mesh size $h$, spline order $p$ and continuity $0 \leq k \leq p-1$, all assumed uniform in each direction for simplicity. One conveniently computes a basis for $\mathcal{S}_{h,p,k}$ by first constructing the standard B-spline basis $\widehat{\mathcal{B}}$ over the fictitious domain $\widehat{\Omega}$ and later retaining all basis functions whose support intersects the physical domain; i.e.
\begin{equation*}
    \mathcal{B} = \{B_i \in \widehat{\mathcal{B}} \colon \mathrm{supp}(B_i) \cap \Omega \neq \emptyset \} \subseteq \widehat{\mathcal{B}}.
\end{equation*}
The tensor product basis $\widehat{\mathcal{B}}$ follows a standard construction thoroughly explained in textbooks and early literature on isogeometric analysis; see e.g. \cite{hughes2005isogeometric,cottrell2009isogeometric}. Denoting $I$ the index set of $\mathcal{B}$, $n=|\mathcal{B}|=\dim(\mathcal{S}_{h,p,k})$ the space dimension, and 
\begin{equation*}
    V^h_T = \{v \in L^2((0,T); V^h) \colon \ \partial_t v \in L^2((0,T); V^h), \ \partial_{tt} v \in L^2((0,T); V^h) \},
\end{equation*}
the Galerkin method seeks an approximate solution $u^h \in V^h_T$ such that for almost every time $t \in (0,T)$
\begin{align}
 b(\partial_{tt} u^h(t), v^h) + a(u^h(t),v^h) &= F(v^h) & \forall v^h \in V^h, \label{eq: discrete_weak_form} \\
 (u^h(0),v^h)_{L^2(\Omega)} &= (u_0,v^h)_{L^2(\Omega)} & \forall v^h \in V^h, \nonumber \\
 (\partial_t u^h(0),v^h)_{L^2(\Omega)} &= (v_0,v^h)_{L^2(\Omega)} & \forall v^h \in V^h, \nonumber
\end{align}
where $a,b \colon V^h \times V^h \to \mathbb{R}$ are symmetric bilinear forms defined as
\begin{equation}
\label{eq: global_contribution}
    a(u^h,v^h) = (\nabla u^h, \nabla v^h)_{L^2(\Omega)^d}, \quad \text{and} \quad b(u^h,v^h) = (u^h,v^h)_{L^2(\Omega)}
\end{equation}
and $F \colon V^h \to \mathbb{R}$ is the linear functional 
\begin{equation*}
    F(v^h) = (f(t),v^h)_{L^2(\Omega)} + (h(t),v^h)_{L^2(\partial \Omega_N)}.
\end{equation*}
For the time being, $b(u^h,v^h)$ is the standard $L^2$ inner product but will later take on different definitions for lumped mass approximations. Expanding the approximate solution in the B-spline basis and testing against every function in $\mathcal{B}$, \eqref{eq: discrete_weak_form} is equivalent to a second order system of ordinary differential equations (ODEs) \citep{hughes2012finite,quarteroni2009numerical}
\begin{align}
\label{eq: semi_discrete_pb}
\begin{split}
M\ddot{\mathbf{u}}(t) + K\mathbf{u}(t) &= \mathbf{f}(t) \qquad \text{for } t \in [0,T], \\
\mathbf{u}(0) &= \mathbf{u}_0,\\
\dot{\mathbf{u}}(0) &= \mathbf{v}_0.
\end{split}
\end{align}
where $\mathbf{u}(t)$ is the vector of coefficients of $u^h(\mathbf{x},t)$ in the basis $\mathcal{B}$, 
\begin{equation}
    \label{eq: stiffness_mass_matrix}
    K_{ij} = a(B_i,B_j) \quad \text{and} \quad M_{ij} = b(B_i,B_j) \qquad i,j \in I
\end{equation}
are the stiffness and (consistent) mass matrices (i.e. the Gram matrices of $a(u,v)$ and $b(u,v)$ in the basis $\mathcal{B}$) and
\begin{equation*}
    f_i(t) = F(B_i) \qquad i \in I.
\end{equation*}
If $\partial \Omega_D \neq \emptyset$, the stiffness and mass matrices are both symmetric positive definite and the mass matrix is additionally nonnegative, owing to the nonnegativity of the B-spline basis. In immersogeometric analysis, the basis construction is completely independent of the physical domain $\Omega$ and while it certainly simplifies the initial setup, it might also create basis functions $B_i$ whose active support $\mathrm{supp}(B_i) \cap \Omega$ is arbitrarily small, potentially causing stability and conditioning issues. For a consistent mass approximation, the former leads to extremely large generalized eigenvalues and practically prevents solving \eqref{eq: semi_discrete_pb} with explicit time integration methods. In practice though, the consistent mass is oftentimes substituted with an ad hoc (diagonal) lumped mass matrix. The row-sum technique \cite{hughes2012finite} is a natural choice for nonnegative matrices and is defined as
\begin{equation}
\label{eq: row_sum}
\mathcal{L}(M)=\diag(d_1,\dots,d_n)
\end{equation}
where $d_i=\sum_{j=1}^n m_{ij}$ for $i=1,\dots,n$. The row-sum technique was generalized in \cite{voet2023mathematical} to form sequences of (block) banded matrices that monotonically converge to the consistent mass in the Loewner ordering. This construction was later further extended to arbitrary single-patch, multi-patch and trimmed geometries. See \cite{voet2024mass} for the details.

In his Ph.D. thesis, Leidinger \cite{leidinger2020explicit} first showed that the largest generalized eigenvalues for the row-sum lumped mass did not depend on the trimming configuration if the continuity was sufficiently large. This astonishing finding, sometimes considered a ``blessing of smoothness'', was later confirmed in several independent numerical studies \cite{messmer2021isogeometric,messmer2022efficient,coradello2021accurate,stoter2023critical,radtke2024analysis} and recently analyzed in \cite{bioli2024theoretical}. The apparently beneficial effect of trimming was widely promoted in early literature \cite{messmer2021isogeometric,messmer2022efficient}, where the authors even suggested artificially embedding the physical domain in a larger fictitious domain to trim off open knot vectors. 

However, the tide is now turning as increasing evidence indicates small trimmed elements may instead bring in spurious modes in the low-frequency spectrum \cite{coradello2021accurate,radtke2024analysis,bioli2024theoretical} and their impact on the discrete solution is becoming the object of intense scrutiny. Intuitively, one would not expect any contribution at all from spurious low-frequency modes given they are highly localized on small trimmed elements. However, numerical evidence from existing literature is not as clear-cut. Already in \cite{messmer2021isogeometric}, the authors lacked proper stabilization and did not report any results of transient simulations with small trimmed elements despite conducting a step size analysis. This deficiency was later reiterated in \cite{messmer2022efficient}. Around the same time, Coradello \cite{coradello2021accurate} noticed severe inaccuracies in the low-frequency spectrum and modes and stated that ``the impact of these spurious eigenvalues on the accuracy of a dynamic simulation is not yet fully understood''. Similar issues have lately resurfaced in \cite{radtke2024analysis} and their impact on the solution is still being debated. In the next section, we provide compelling evidence highlighting the disastrous effect of spurious low-frequency modes.

\section{Motivation}
\label{se: motivation}
Although mass lumping often completely removes the dependency of the CFL on small cut elements, it may come at a very high price. Namely, high-frequency spurious eigenmodes for the consistent mass may become low-frequency eigenmodes when it is lumped. If activated, these spurious low-frequency modes may completely ruin the solution. The following examples will illustrate it. In all our examples, the mesh size is chosen fine enough such that the accuracy for a consistent mass formulation is satisfactory and all our experiments are done using GeoPDEs \cite{vazquez2016new}, an open source Matlab/Octave software package for isogeometric analysis.

\begin{example}[1D counter-example]
\label{ex: counter_example_1D}
For demonstrating how spurious modes appear in the low-frequency spectrum, we first solve the 1D Laplace eigenvalue problem on the trimmed domain $\Omega = (0, 0.75+\epsilon) \subset (0,1) = \widehat{\Omega}$, with $\epsilon=10^{-6}$. Dirichlet boundary conditions are prescribed on the left side and Neumann boundary conditions on the right side (see \Cref{fig: 1D_trimming}). The exact eigenvalues are
\begin{equation}
    \lambda_j = \left(\frac{2j-1}{2}\frac{\pi}{|\Omega|}\right)^2 \quad j=1,2,\dots, \infty.
\end{equation}

\begin{figure}[H]
    \centering
    \includegraphics[width=0.6\linewidth]{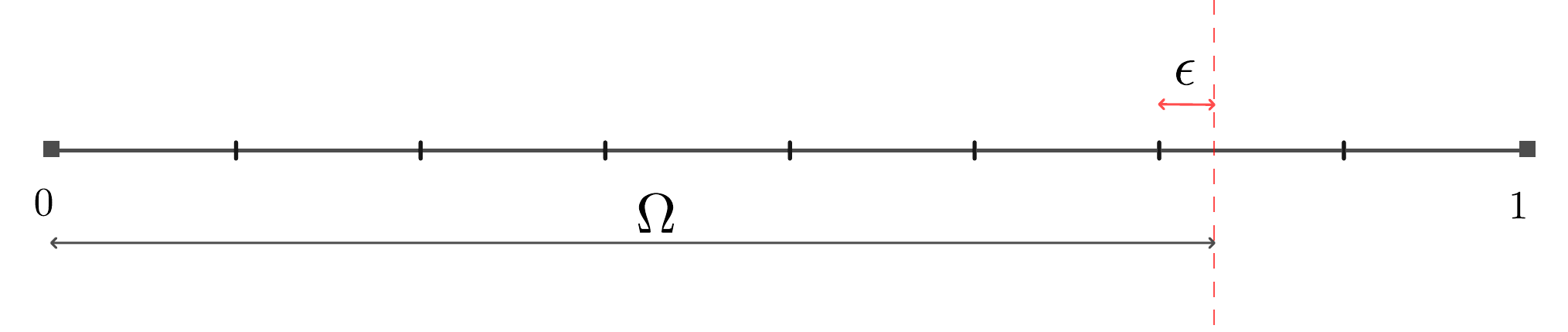}
    \caption{Trimmed line segment}
    \label{fig: 1D_trimming}
\end{figure}

For discretizing the problem, the standard B-spline basis is constructed over $\widehat{\Omega}$ using a fine mesh with $256$ subdivisions. As usual, the basis over the physical domain is obtained by retaining the basis functions whose support intersects $\Omega$. The normalized spectrum of $(K,M)$ and $(K,\mathcal{L}(M))$ for maximally smooth discretizations of degree $1$ to $4$ is shown in \Cref{fig: 1D_Laplace_trimming_normalized_spectrum_eps_1e-6_default} and reproduces some results of Coradello's Ph.D. thesis \cite{coradello2021accurate}. As Leidinger \cite{leidinger2020explicit} and others found out, the largest eigenvalue for the lumped mass is unaffected by the small trimmed element of size $\epsilon$ if the smoothness is at least $C^1$. However, the rest of the figure is much less pleasing: while the largest eigenvalues of $(K,M)$ blow-up (but are not shown for visualization purposes), the smallest eigenvalues of $(K,\mathcal{L}(M))$ converge to zero as the spline order increases, once again confirming the findings in \cite{bioli2024theoretical}. A closer examination actually reveals an apparent shifting whereby inaccurate eigenvalues lie in between accurate ones, as shown in \cite{radtke2024analysis}. To illustrate it, \Cref{fig: 1D_Laplace_trimming_spectrum_p3_eps_1e-6_zoom,fig: 1D_Laplace_trimming_spectrum_p4_eps_1e-6_zoom} zoom in on the first $14$ eigenvalues for degrees $p=3$ and $p=4$, respectively, and pair the eigenvalues of $(K,\mathcal{L}(M))$ to the closest exact eigenvalue. \Cref{fig: 1D_Laplace_trimming_spectrum_p3_eps_1e-6_zoom} reveals that in between two well approximated eigenvalues lies a fictitious eigenvalue, which does not approximate anything. The situation worsens as $\epsilon$ decreases or the degree increases: for $p=4$, this same eigenvalue has now decreased by several orders of magnitude and will eventually converge to zero as the degree further increases \cite{bioli2024theoretical}. However, the remaining eigenvalues are still relatively accurate. Thus, the picture sketched in \Cref{fig: 1D_Laplace_trimming_normalized_spectrum_eps_1e-6_default} is mostly an artifact of the ordering and a priori, the situation is far less dramatic than it originally suggested. Indeed, eliminating the spurious eigenvalues and relabeling the remaining eigenvalues in \Cref{fig: 1D_Laplace_trimming_normalized_spectrum_eps_1e-6_relabeled} reveals how accurate some of the lower frequencies are. However, in more realistic settings where the exact eigenvalues are unknown, distinguishing physical eigenvalues from spurious ones is impossible without information on the eigenfunctions. Yet, spurious eigenvalues in the low-frequency spectrum should not be ignored as they bring in spurious eigenfunctions, shown in \Cref{fig: 1D_Laplace_trimming_dynamics_p3_eps_1e-6_eigenfunctions} for $p=3$. These spurious eigenfunctions closely resemble outlier eigenfunctions for the consistent mass and if they are unluckily activated during a simulation, they might contribute to the solution given that they are associated to small frequencies. Unfortunately, this is a real possibility as the next part of our example shows.


\begin{figure}[H]
     \centering
     \begin{subfigure}[t]{0.48\textwidth}
    \centering
    \includegraphics[width=\linewidth]{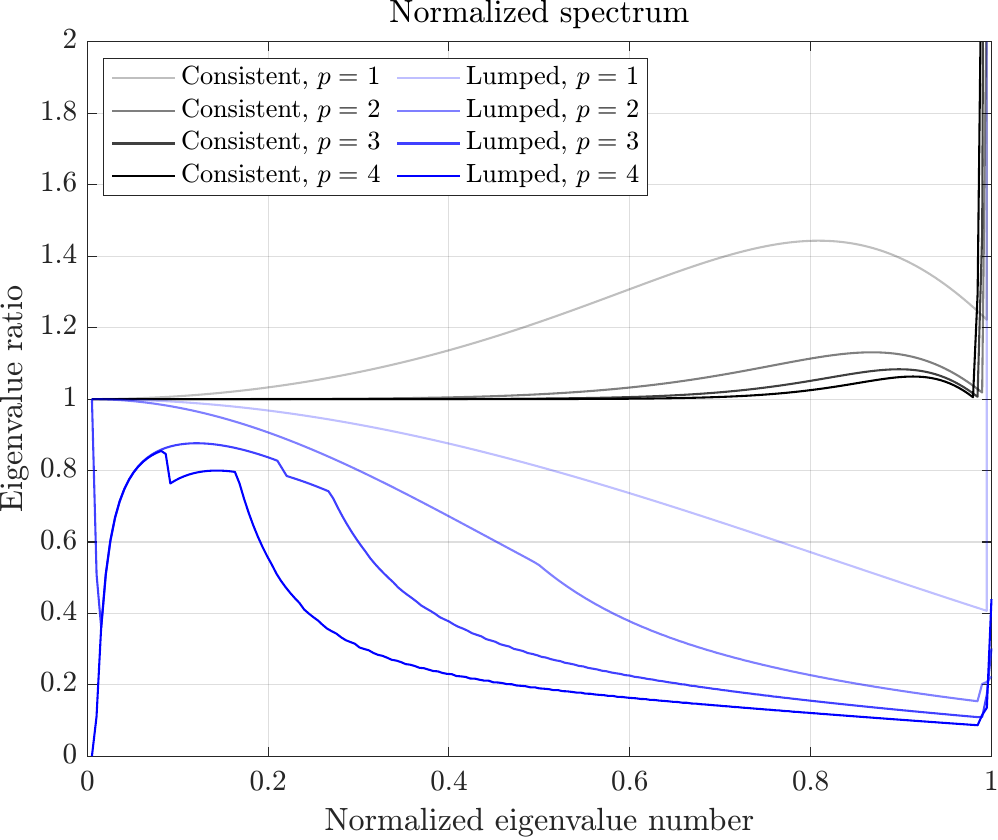}
    \caption{Eigenvalues labeled in increasing algebraic order}
    \label{fig: 1D_Laplace_trimming_normalized_spectrum_eps_1e-6_default}
     \end{subfigure}
     \hfill
     \begin{subfigure}[t]{0.48\textwidth}
    \centering
    \includegraphics[width=\linewidth]{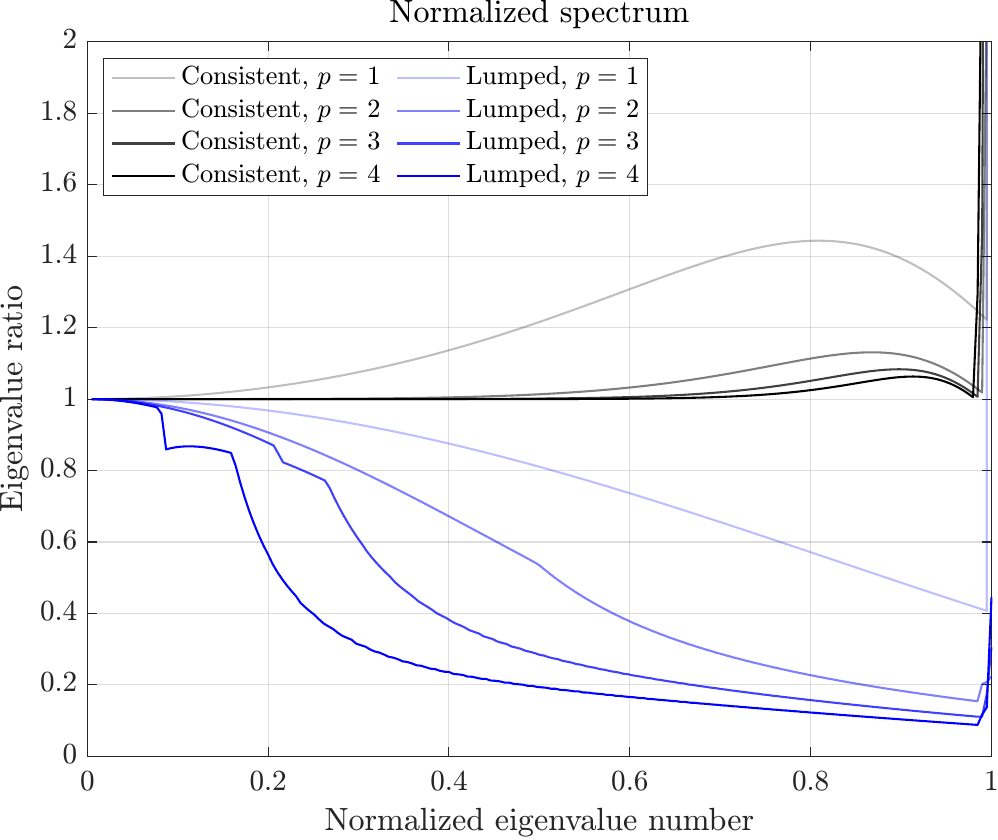}
    \caption{With spurious eigenvalues removed}
    \label{fig: 1D_Laplace_trimming_normalized_spectrum_eps_1e-6_relabeled}
     \end{subfigure}
     \hfill
    \caption{Ratio of approximate over exact eigenvalues for consistent and lumped mass approximations}
    \label{fig: 1D_Laplace_trimming_normalized_spectrum_eps_1e-6}
\end{figure}

\begin{figure}[H]
     \centering
     \begin{subfigure}[t]{0.48\textwidth}
    \centering
    \includegraphics[width=\textwidth]{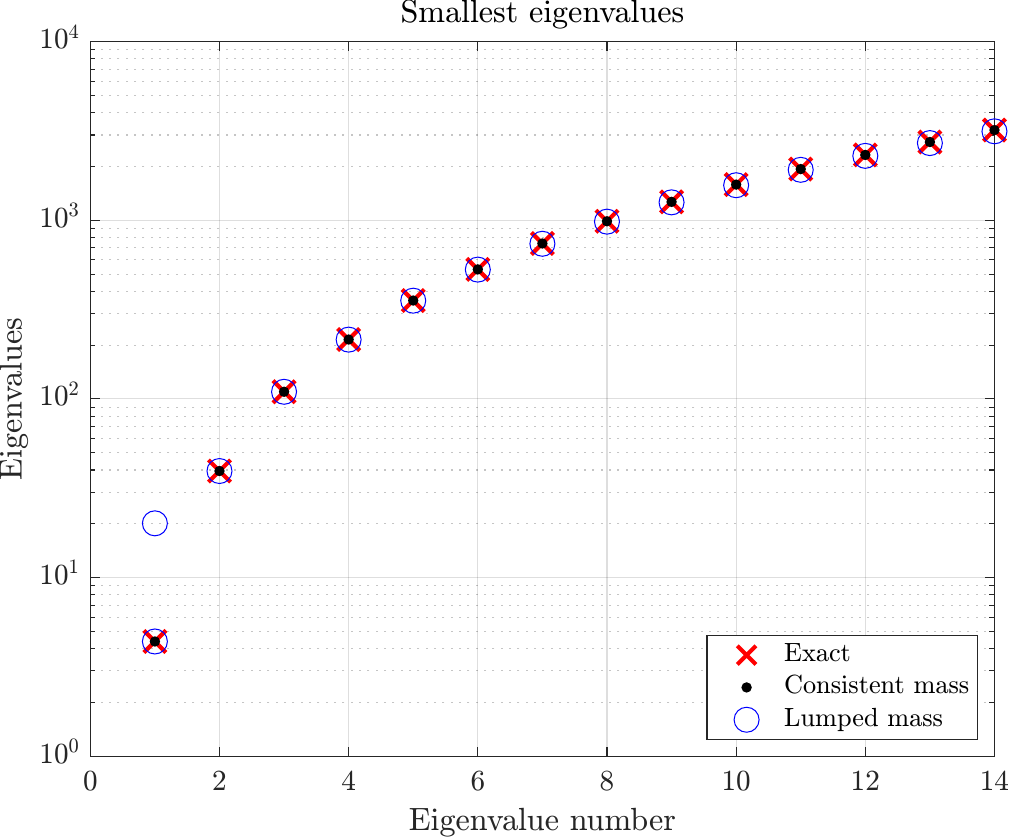}
    \caption{$p=3$}
    \label{fig: 1D_Laplace_trimming_spectrum_p3_eps_1e-6_zoom}
     \end{subfigure}
     \hfill
     \begin{subfigure}[t]{0.48\textwidth}
    \centering
    \includegraphics[width=\textwidth]{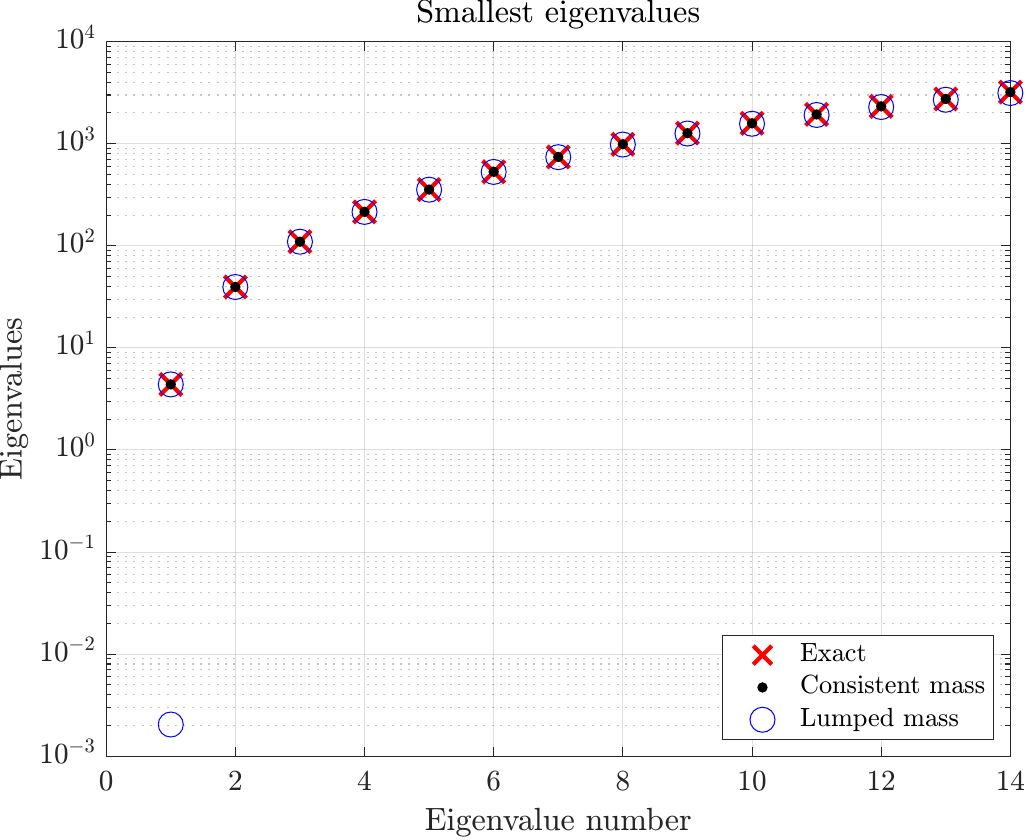}
    \caption{$p=4$}
    \label{fig: 1D_Laplace_trimming_spectrum_p4_eps_1e-6_zoom}
     \end{subfigure}
     \hfill
    \caption{Close-up on the first $14$ eigenvalues. The eigenvalues of $(K,\mathcal{L}(M))$ are relabeled based on the closest exact eigenvalue.}
    \label{fig: 1D_Laplace_trimming_spectrum_zoom}
\end{figure}

\begin{figure}[H]
    \centering
    \includegraphics[width=\linewidth]{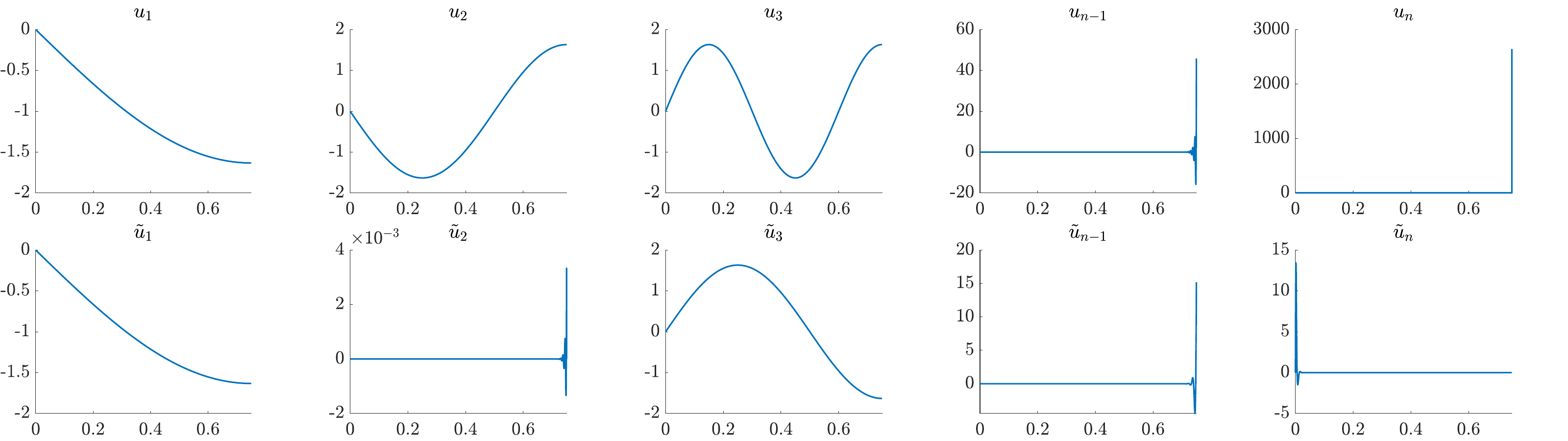}
    \caption{Three first and two last eigenfunctions of $(K,M)$ (top row) and $(K,\mathcal{L}(M))$ (bottom row) for $p=3$. The labeling of eigenfunctions matches the labeling of eigenvalues.}
    \label{fig: 1D_Laplace_trimming_dynamics_p3_eps_1e-6_eigenfunctions}
\end{figure}

We now solve the wave equation on the same domain $\Omega$ over the time interval $[0,3]$ with the manufactured solution $u(x,t)=q(x)\sin(n\pi t)$ where
\begin{equation}
\label{eq: function_q(x)}
    q(x) = C^{(\frac{x}{x_r})^a}x\sin\left(\frac{\pi}{x_l-x}\right)
\end{equation}
with parameters $x_r=0.75+\epsilon$, $x_l=\frac{1}{w}+x_r$, $C=8$, $a=8$, $w=15$ and $n=3$. The function $q(x)$ is shown in \Cref{fig: 1D_Laplace_trimming_dynamics_eps_1e-6_function} and becomes increasingly oscillatory near the trimmed boundary. The right-hand side, initial and boundary conditions are enforced such that $u(x,t)$ is the exact solution of the wave equation. In particular, it satisfies homogeneous Dirichlet boundary conditions on the non-trimmed (left) boundary and Neumann boundary conditions on the trimmed (right) boundary. The solution for the lumped mass is computed with the central difference method using the step size given by \cref{eq: CFL_central_difference} (and multiplied by a safeguarding factor of $0.85$). For the consistent mass, the outliers prevent using an explicit method and we instead rely on an implicit unconditionally stable Newmark method with the same step size as for the lumped mass. Snapshots of the solutions are shown in \Cref{fig: 1D_Laplace_trimming_dynamics_solution_snapshots_p3_eps_1e-6_Cp_1_explicit} for a cubic $C^2$ discretization. While the solution for the consistent mass remains consistently accurate, the solution for the lumped mass is wholly inaccurate and periodically features oscillations near the trimmed boundary. As shown in \Cref{fig: 1D_Laplace_trimming_dynamics_solution_snapshots_p3_eps_1e-6_C0_explicit}, the issue persists when reducing to $C^0$ continuity, although the oscillations are now highly localized at the trimmed boundary and are actually undetectable from the global $L^2$ error. However, the beneficial effect smoothness had on the CFL condition is lost and we were forced to employ an unconditionally stable Newmark method both for the consistent and lumped mass. Nevertheless, we used the same step size as in the explicit case with maximal smoothness to ensure the time discretization error was similar. 

These observations underline the subtle and mixed role smoothness plays for lumped mass approximations. While greater smoothness ensures the step size remains bounded independently of small trimmed elements, comparing \Cref{fig: 1D_Laplace_trimming_dynamics_solution_snapshots_p3_eps_1e-6_Cp_1_explicit,fig: 1D_Laplace_trimming_dynamics_solution_snapshots_p3_eps_1e-6_C0_explicit} also highlights its detrimental effect on the accuracy of the lumped mass approximation. Indeed, for $C^0$ discretizations, the oscillations are confined near the trimmed boundary while they tend to propagate inside the domain when increasing the smoothness, thereby leading to larger errors. As evidenced in \Cref{fig: 1D_Laplace_trimming_dynamics_L2_error}, increasing the degree or reducing $\epsilon$ further amplifies the error. In particular, \Cref{fig: 1D_Laplace_trimming_dynamics_L2_error_LM_p3_explicit} also reveals that the error oscillates over time and \Cref{se: analysis} will later shed some light on this behavior. We must stress that these artifacts are not a consequence of time integration. As a matter of fact, as shown in \Cref{se: analysis}, the fully discrete solution closely matches the exact solution of the semi-discrete problem \eqref{eq: semi_discrete_pb}, which admits a closed form solution in this example. Thus, the issue is solely tied to mass lumping. More specifically, the culprit is the spatial part of the exact solution $q(x)$, which is unfortunately exceedingly well described by one of the spurious eigenmodes for the lumped mass ($\tilde{u}_2$ in \Cref{fig: 1D_Laplace_trimming_dynamics_p3_eps_1e-6_eigenfunctions}).

\begin{figure}[H]
    \centering
    \includegraphics[width=0.5\linewidth]{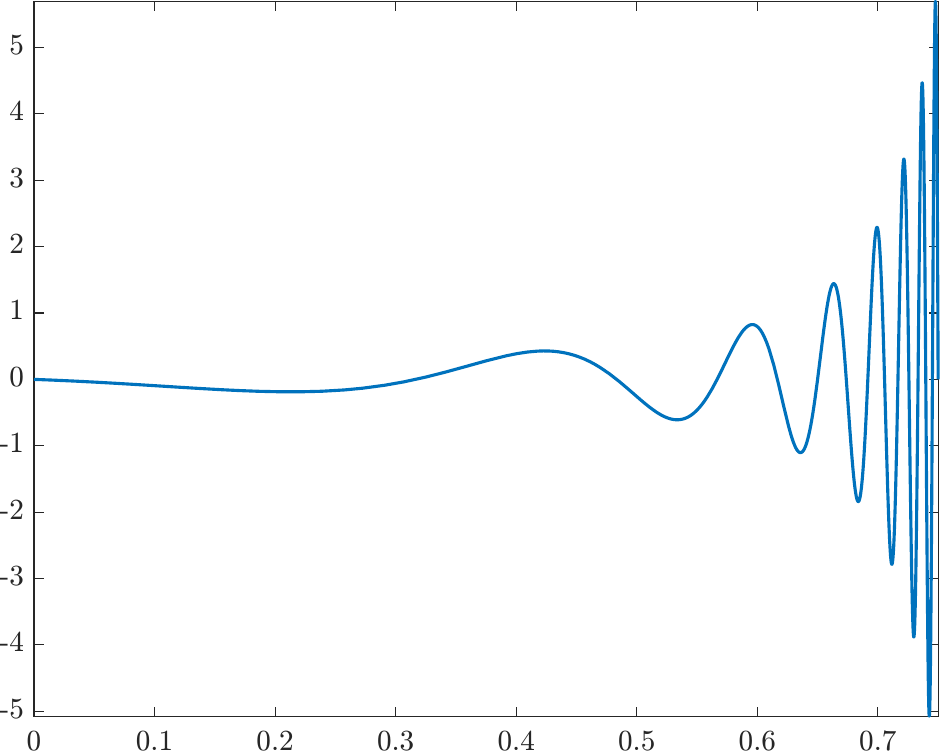}
    \caption{Function $q(x)$ defined in \cref{eq: function_q(x)}}
    \label{fig: 1D_Laplace_trimming_dynamics_eps_1e-6_function}
\end{figure}

\begin{figure}[H]
    \centering
    \includegraphics[width=\linewidth]{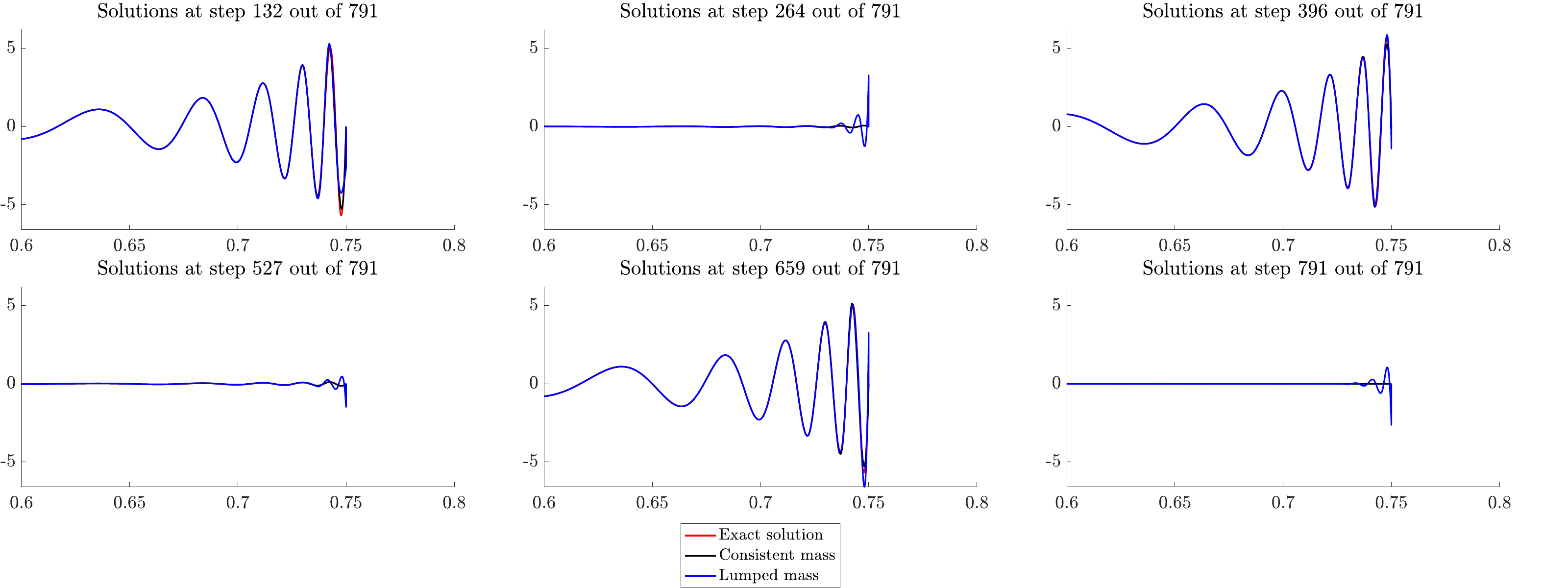}
    \caption{Exact solution and discrete solutions for the consistent and lumped mass for cubic $C^2$ B-splines}
    \label{fig: 1D_Laplace_trimming_dynamics_solution_snapshots_p3_eps_1e-6_Cp_1_explicit}
\end{figure}

\begin{figure}[H]
    \centering
    \includegraphics[width=\linewidth]{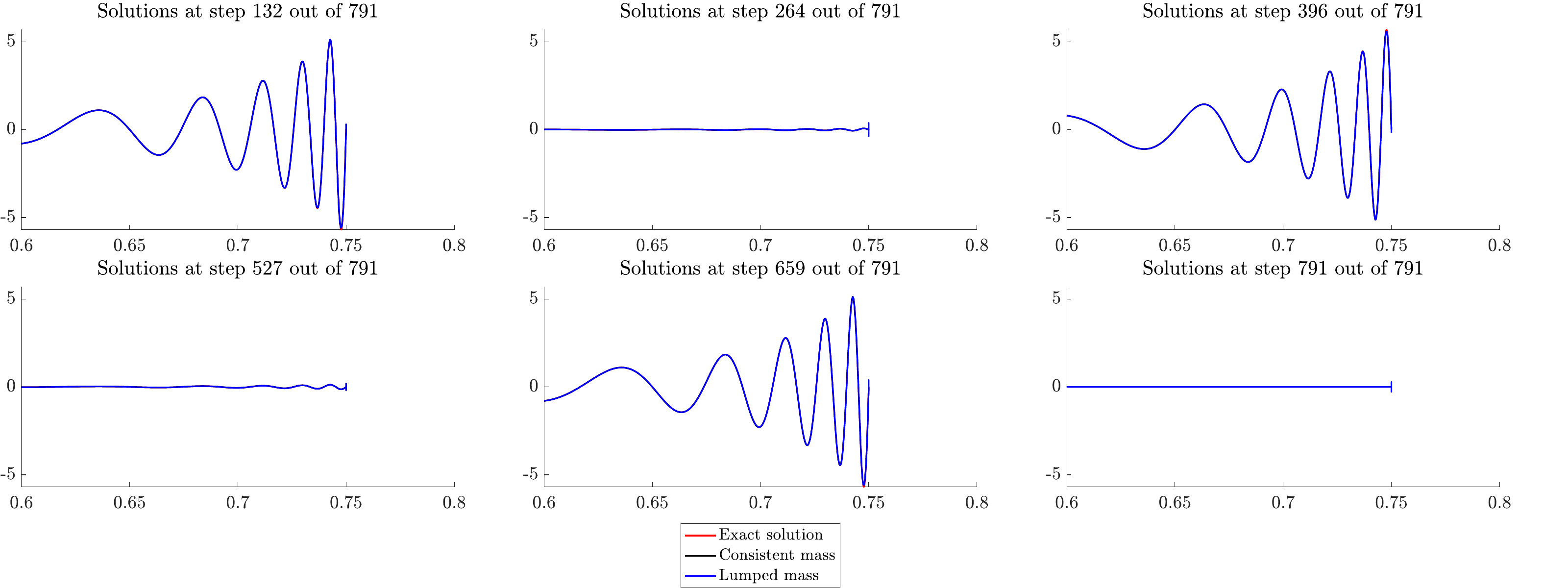}
    \caption{Exact solution and discrete solutions for the consistent and lumped mass for cubic $C^0$ B-splines}
    \label{fig: 1D_Laplace_trimming_dynamics_solution_snapshots_p3_eps_1e-6_C0_explicit}
\end{figure}

\begin{figure}[H]
     \centering
     \begin{subfigure}[t]{0.48\textwidth}
    \centering
    \includegraphics[width=\textwidth]{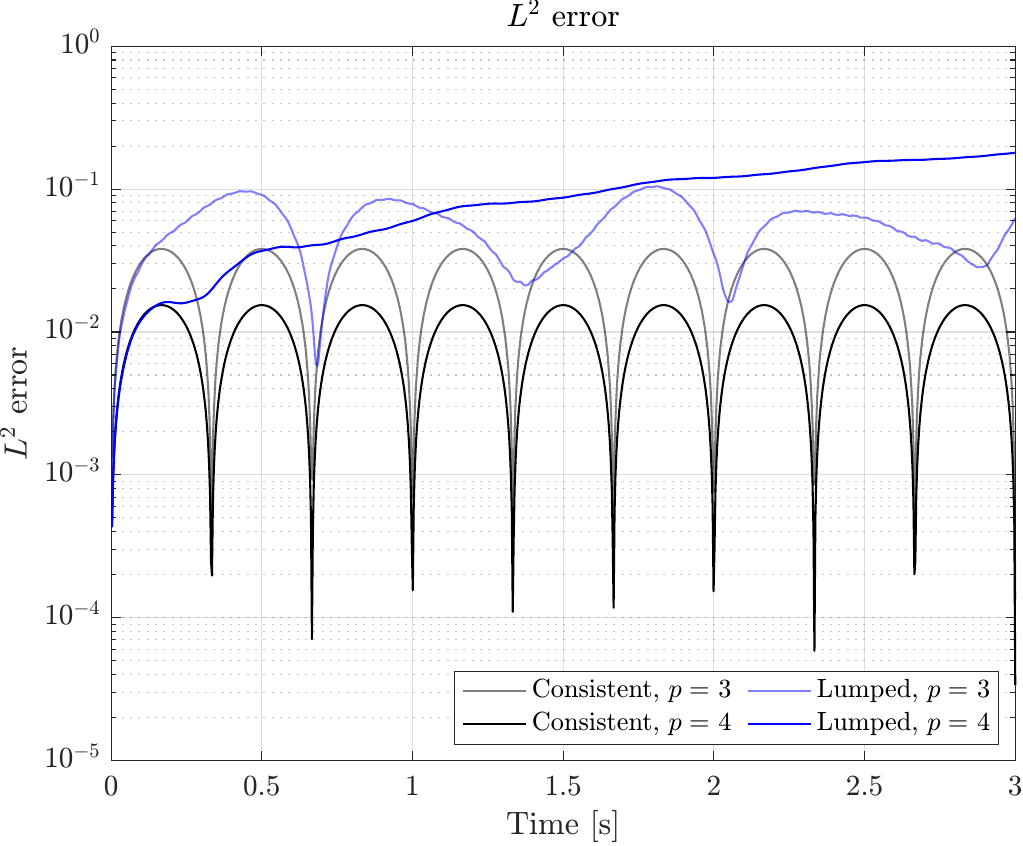}
    \caption{$L^2$ error over the time for $C^{p-1}$ discretizations of degree $p=3,4$}
    \label{fig: 1D_Laplace_trimming_dynamics_L2_error_eps_1e-6_explicit}
     \end{subfigure}
     \hfill
     \begin{subfigure}[t]{0.48\textwidth}
    \centering
    \includegraphics[width=\textwidth]{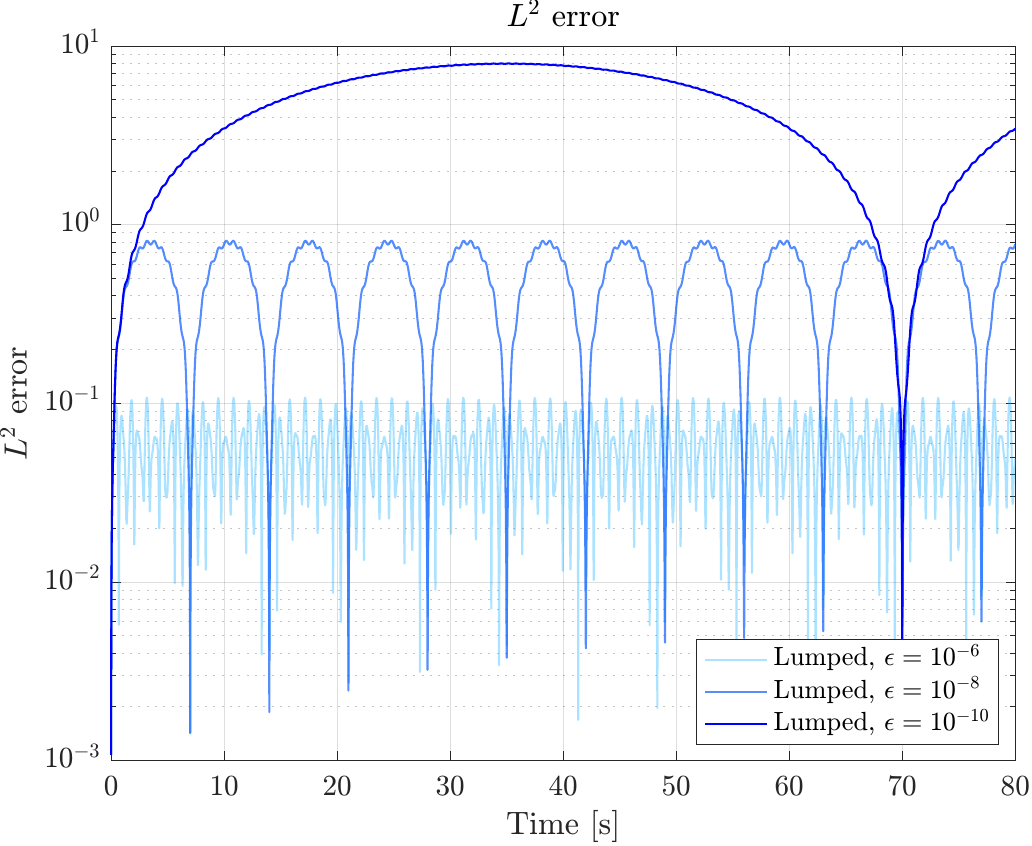}
    \caption{$L^2$ error over the time interval $[0,80]$ for a cubic $C^2$ discretization and decreasing values of $\epsilon$}
    \label{fig: 1D_Laplace_trimming_dynamics_L2_error_LM_p3_explicit}
     \end{subfigure}
     \hfill
    \caption{Evolution of the $L^2$ error over time for the consistent and lumped mass solutions}
    \label{fig: 1D_Laplace_trimming_dynamics_L2_error}
\end{figure}
\end{example}

\begin{remark}
Our experiments always employ a mesh-independent trimming parameter $\epsilon$ for defining the physical domain. This ensures it remains independent of the discretization. However, the quality of the results depends on how small $\eta=\epsilon/h$ is, where $h$ is the mesh size. Refining the mesh while keeping $\epsilon$ constant (i.e. increasing $\eta$) improves the results for the lumped mass. Moreover, our experiments convey another message: even though the solution for the consistent mass is satisfactory for a given mesh size, it might become unusable when the mass matrix is lumped.
\end{remark}

\begin{example}[Rotated square]
\label{ex: counter_example_rotated_square}
Similar problems may arise on more complicated geometries. For demonstrating it, we solve the wave equation on a shifted and rotated square, oftentimes used as a benchmark for stability and conditioning analyses \cite{de2017condition,leidinger2020explicit}. A square of side length $2s=2(0.25+\epsilon)$, centered at $(0,0)$, is translated by $\bm{\tau}=(0.5,0.5)$ and rotated $0.85$ radians to rest within the unit square. A tensor product mesh is then laid in the ambient fictitious square, as shown in \Cref{fig: 2D_Laplace_trimmed_rotated_square_domain_n32} for $N=32$ subdivisions in each direction. We consider the manufactured solution on the physical square's initial configuration (before translation and rotation) $\hat{u}(\hat{\mathbf{x}},t)=\hat{w}(\hat{\mathbf{x}})\sin(n \pi t)$, where
\begin{equation}
\label{eq: function_w(x)}
    \hat{w}(\hat{\mathbf{x}}) = \left(q(\hat{x})+q(-\hat{x})\right)\left(q(\hat{y})+q(-\hat{y})\right) \quad \text{on} \quad \hat{\Omega}=[-s, s] \times [-s, s]
\end{equation}
where $q(x)$ is defined in \cref{eq: function_q(x)} and we set its parameters to $x_r=s=0.25+\epsilon$, $x_l=\frac{1}{w}+x_r$, $C=8$, $a=8$, $w=10$, $n=3$ and $\epsilon=10^{-6}$. The function $\hat{w}(\hat{\mathbf{x}})$ is shown in \Cref{fig: 2D_Laplace_trimmed_rotated_square_function_2D}. The manufactured solution on the final shifted and rotated configuration is defined by $u(\mathbf{x},t)=\hat{u}(F^{-1}(\mathbf{x}),t)$, where $F \colon \hat{\Omega} \to \Omega$ is the map $F(\hat{\mathbf{x}})=R\hat{\mathbf{x}}+\bm{\tau}$, with $R$ the rotation matrix (for a rotation angle of 0.85 radians) and $\tau$ the translation vector. 

A cubic $C^2$ spline basis is constructed on a fine mesh with 128 subdivisions in each direction. Once again, the numerical solution for the lumped mass is computed with the central difference method using the step size given by the CFL condition \eqref{eq: CFL_central_difference} (and multiplied by a safeguarding factor of $0.85$) while the solution for the consistent mass is computed with an implicit unconditionally stable Newmark method with the same step size as the lumped mass. Snapshots of the solutions at the top corner of the domain shortly after the beginning of the simulation are shown in \Cref{fig: 2D_Laplace_trimmed_rotated_square_dynamics_solution_snapshots_p3_rot_0_85_s20} and almost halfway through the simulation in \Cref{fig: 2D_Laplace_trimmed_rotated_square_dynamics_solution_snapshots_p3_rot_0_85_s135}. Spurious oscillations for the lumped mass are barely visible in \Cref{fig: 2D_Laplace_trimmed_rotated_square_dynamics_solution_snapshots_p3_rot_0_85_s20} but dramatically amplify in \Cref{fig: 2D_Laplace_trimmed_rotated_square_dynamics_solution_snapshots_p3_rot_0_85_s135}. The growth of the $L^2$ error over time in \Cref{fig: 2D_Laplace_trimmed_rotated_square_L2_error_explicit_p3_rot_0_85} further highlights the oscillations' amplification. In order to highlight how local the phenomenon is, the extreme values $c_{\min}$ and $c_{\max}$ of the colormap in \Cref{fig: 2D_Laplace_trimmed_rotated_square_dynamics_solution_snapshots_p3_rot_0_85} are calibrated based on the exact and consistent mass solutions only. Thus, any value smaller (resp. larger) than $c_{\min}$ (resp. $c_{\max}$) for the lumped mass is mapped to $c_{\min}$ (resp. $c_{\max}$). This allows to qualitatively distinguish accurate values from inaccurate ones and avoids being misled by stretching the colormap's range of values. This coloring strategy is consistently applied in all subsequent examples.

\begin{figure}[H]
     \centering
     \begin{subfigure}[t]{0.40\textwidth}
    \centering
    \includegraphics[width=\textwidth]{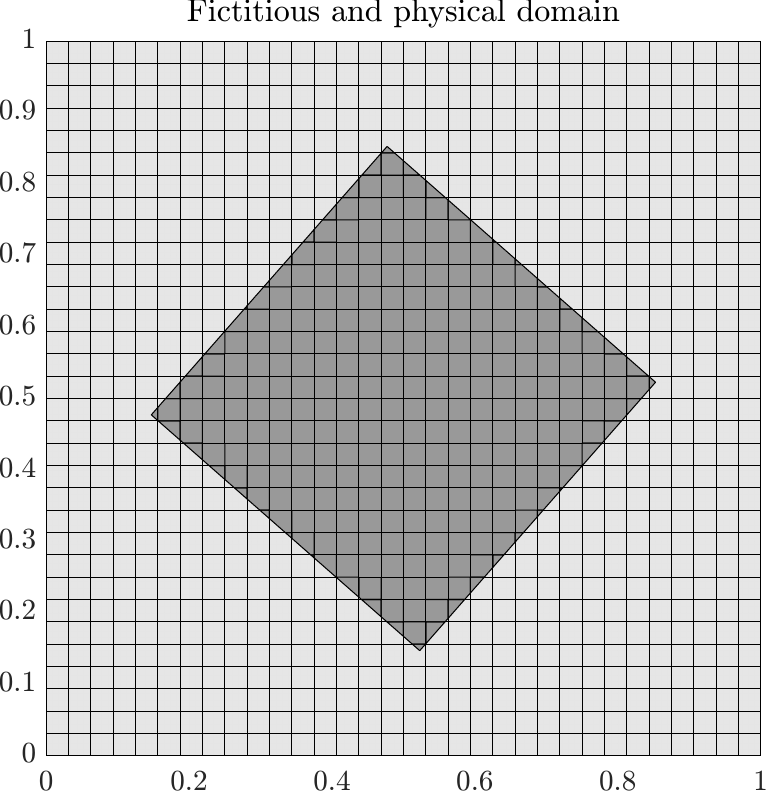}
    \caption{Rotated square embedded in a larger square}
    \label{fig: 2D_Laplace_trimmed_rotated_square_domain_n32}
     \end{subfigure}
     \hfill
     \begin{subfigure}[t]{0.58\textwidth}
    \centering
    \includegraphics[width=\textwidth]{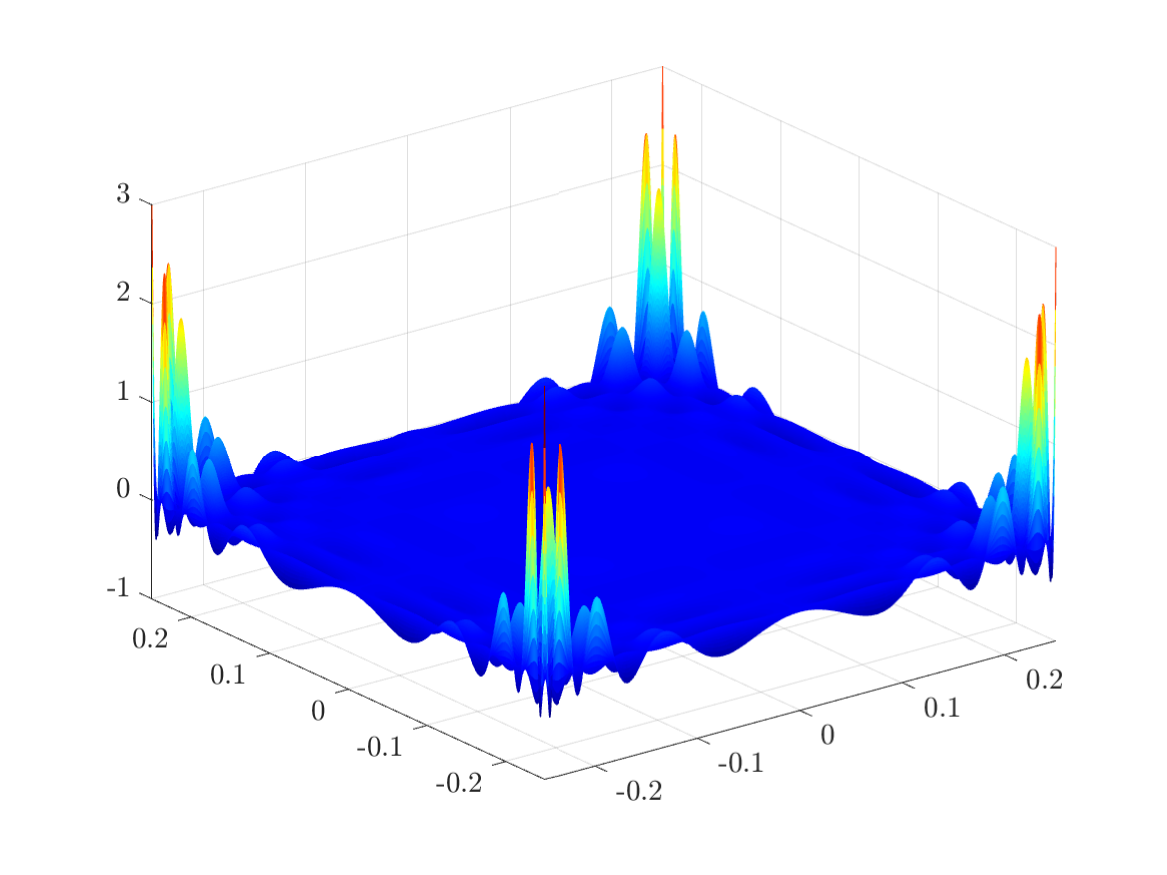}
    \caption{Function $\hat{w}(\hat{\mathbf{x}})$ defined in \cref{eq: function_w(x)}}
    \label{fig: 2D_Laplace_trimmed_rotated_square_function_2D}
     \end{subfigure}
     \hfill
    \caption{Domain and spatial part of the solution}
    \label{fig: 2D_Laplace_trimmed_rotated_square_domain_function}
\end{figure}

\begin{figure}[H]
     \centering
     \begin{subfigure}[t]{1.0\textwidth}
    \centering
    \includegraphics[width=\textwidth]{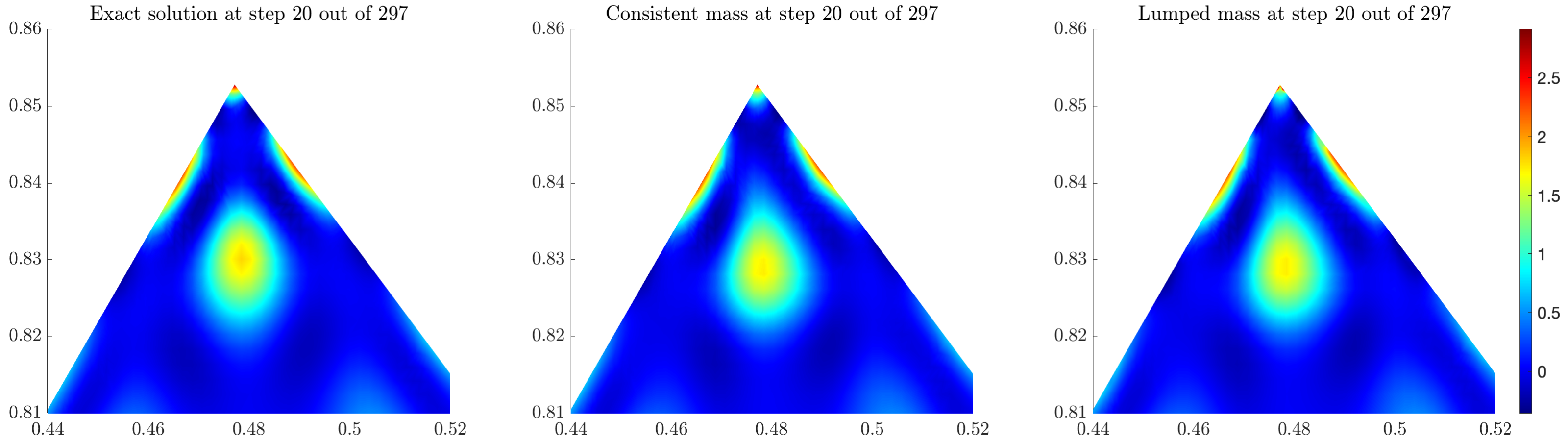}
    \caption{Solution snapshots at time $t=0.1926$}
    \label{fig: 2D_Laplace_trimmed_rotated_square_dynamics_solution_snapshots_p3_rot_0_85_s20}
     \end{subfigure}
     \hfill
     \begin{subfigure}[t]{1.0\textwidth}
    \centering
    \includegraphics[width=\textwidth]{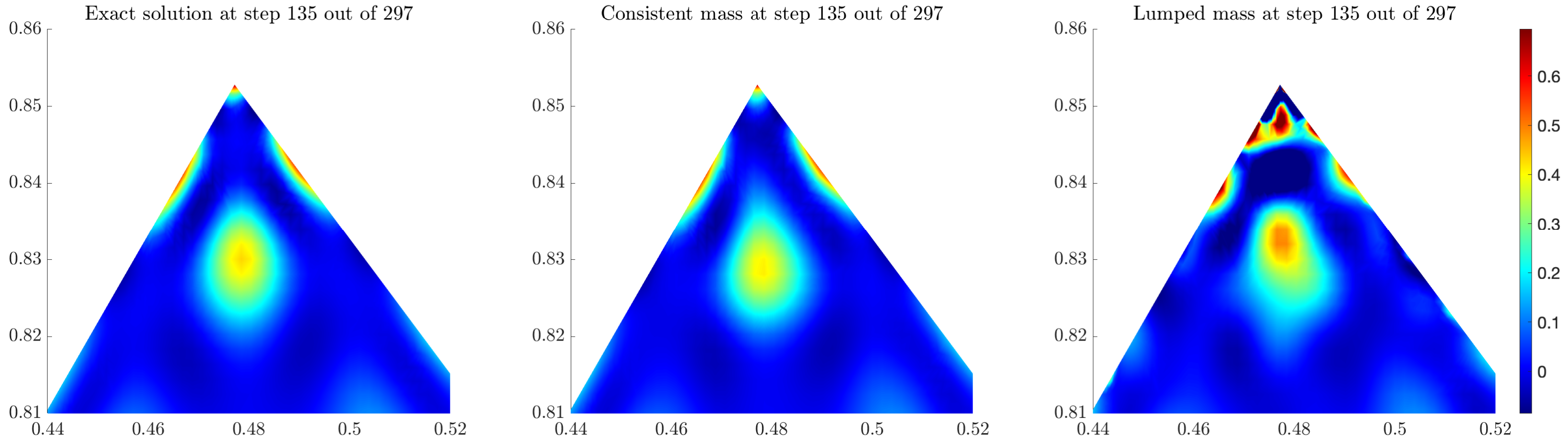}
    \caption{Solution snapshots at time $t=1.3581$}
    \label{fig: 2D_Laplace_trimmed_rotated_square_dynamics_solution_snapshots_p3_rot_0_85_s135}
     \end{subfigure}
     \hfill
    \caption{Exact solution and discrete solutions for the consistent and lumped mass for $p=3$ (zoomed in on the top corner)}
    \label{fig: 2D_Laplace_trimmed_rotated_square_dynamics_solution_snapshots_p3_rot_0_85}
\end{figure}

\begin{figure}[H]
    \centering
    \includegraphics[width=0.5\linewidth]{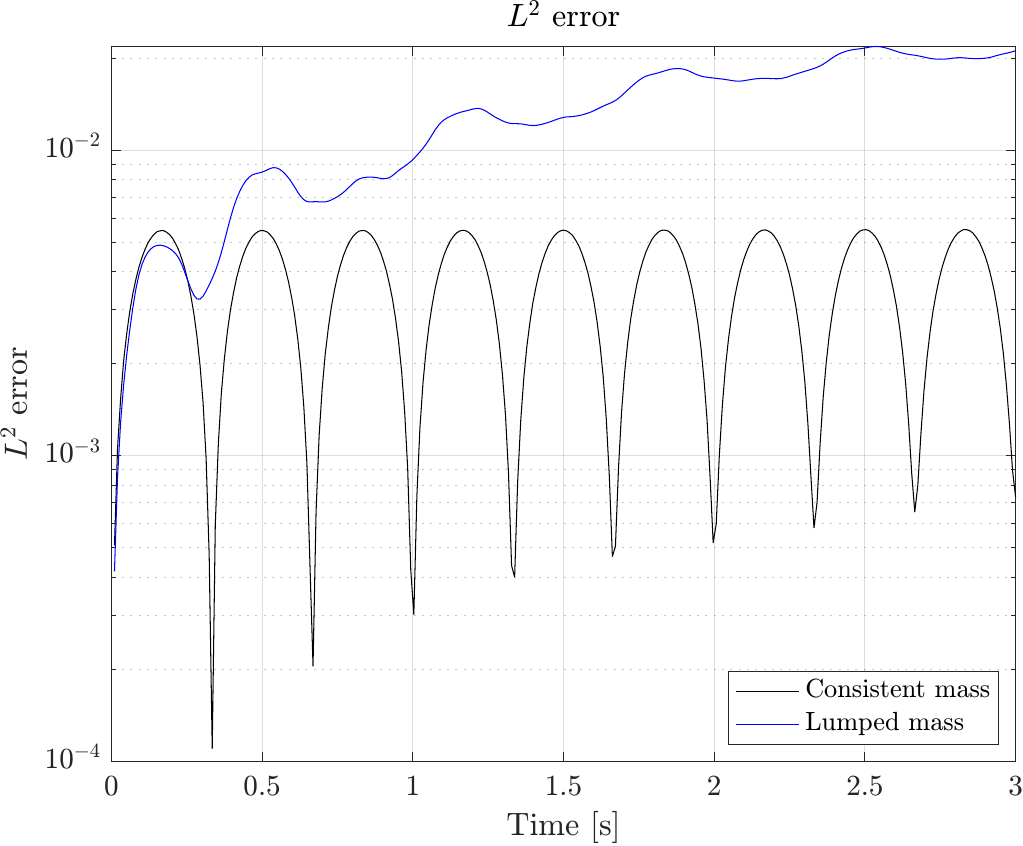}
    \caption{Evolution of the $L^2$ error over time for the consistent and lumped mass solutions}
    \label{fig: 2D_Laplace_trimmed_rotated_square_L2_error_explicit_p3_rot_0_85}
\end{figure}

\end{example}

\begin{example}[Plate with extrusion]
\label{ex: counter_example_plate_with_extrusion}
As a third example, we solve the wave equation on an extruded plate. Two circular arcs of radius $r=0.125-\epsilon$ and center $\mathbf{c}_1=(0.5,0.25)$ and $\mathbf{c}_2=(0.5,0.75)$ are connected vertically to extrude the unit square $\widehat{\Omega}=(0,1)^2$, as shown in \Cref{fig: 2D_Laplace_trimmed_plate_with_extrusion_domain_n32} for $N=32$ subdivisions in each direction. We consider the manufactured solution $u(\mathbf{x},t)=w(\mathbf{x})\sin(n \pi t)$ with 
\begin{equation}
\label{eq: function_w(x)_plate}
    w(\mathbf{x}) = x(x-1)\mathrm{e}^{-\left(\frac{|x-0.5|-r}{\sigma}\right)^2}\sin(m|x-0.5|)
\end{equation}
together with parameters $n=3$, $m=100$, $\sigma=0.05$ and $\epsilon=10^{-7}$. The function $w(\mathbf{x})$, which actually only depends on $x$, is shown in \Cref{fig: 2D_Laplace_trimmed_plate_with_extrusion_function}. As usual, the right-hand side, boundary and initial conditions are prescribed such that $u(\mathbf{x},t)$ is the exact solution of the PDE, which is discretized with quadratic $C^1$ B-splines and $N=48$ subdivisions in each direction. Similarly to the previous examples, the fully discrete solution is computed over the time interval $[0, 3]$ with the central difference method and the step size of the CFL condition (multiplied by $0.85$). Snapshots of the solution are shown in \Cref{fig: 2D_Laplace_trimmed_plate_with_extrusion_dynamics_solution_snapshots_p2_eps_1e-7_s20} for a small time and in \Cref{fig: 2D_Laplace_trimmed_plate_with_extrusion_dynamics_solution_snapshots_p2_eps_1e-7_s135} for a larger time. Once again, while the solution for the consistent mass remains accurate throughout the simulation, the one for the lumped mass rapidly deteriorates and apparently triggers high frequencies, as shown in \Cref{fig: 2D_Laplace_trimmed_plate_with_extrusion_L2_error_explicit_p2_eps_1e-7}.

\begin{figure}[H]
     \centering
     \begin{subfigure}[t]{0.44\textwidth}
    \centering
    \includegraphics[width=\textwidth]{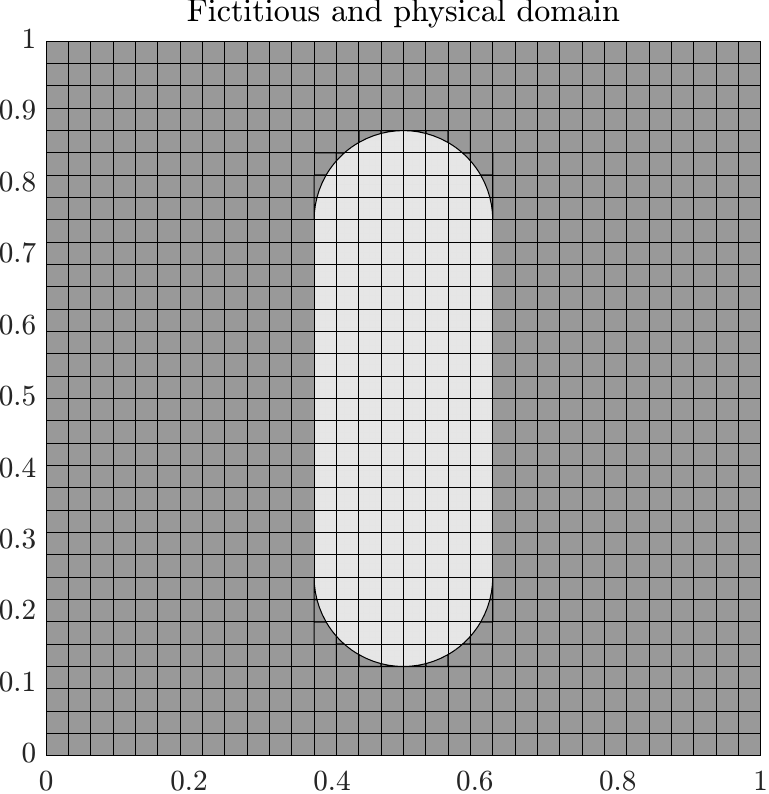}
    \caption{Square plate with an extrusion}
    \label{fig: 2D_Laplace_trimmed_plate_with_extrusion_domain_n32}
     \end{subfigure}
     \hfill
     \begin{subfigure}[t]{0.54\textwidth}
    \centering
    \includegraphics[width=\textwidth]{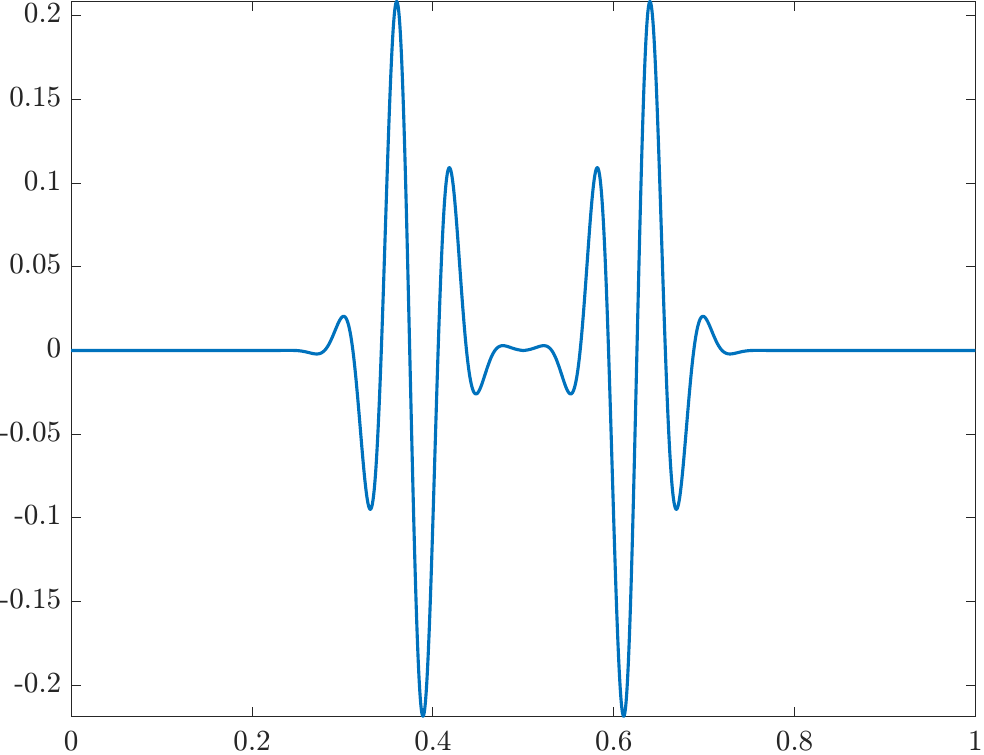}
    \caption{Function $w(\mathbf{x})$ defined in \cref{eq: function_w(x)_plate}}
    \label{fig: 2D_Laplace_trimmed_plate_with_extrusion_function}
     \end{subfigure}
     \hfill
    \caption{Domain and spatial part of the solution}
    \label{fig: 2D_Laplace_plate_with_extrusion_domain_function}
\end{figure}

\begin{figure}[H]
     \centering
     \begin{subfigure}[t]{1.0\textwidth}
    \centering
    \includegraphics[width=\textwidth]{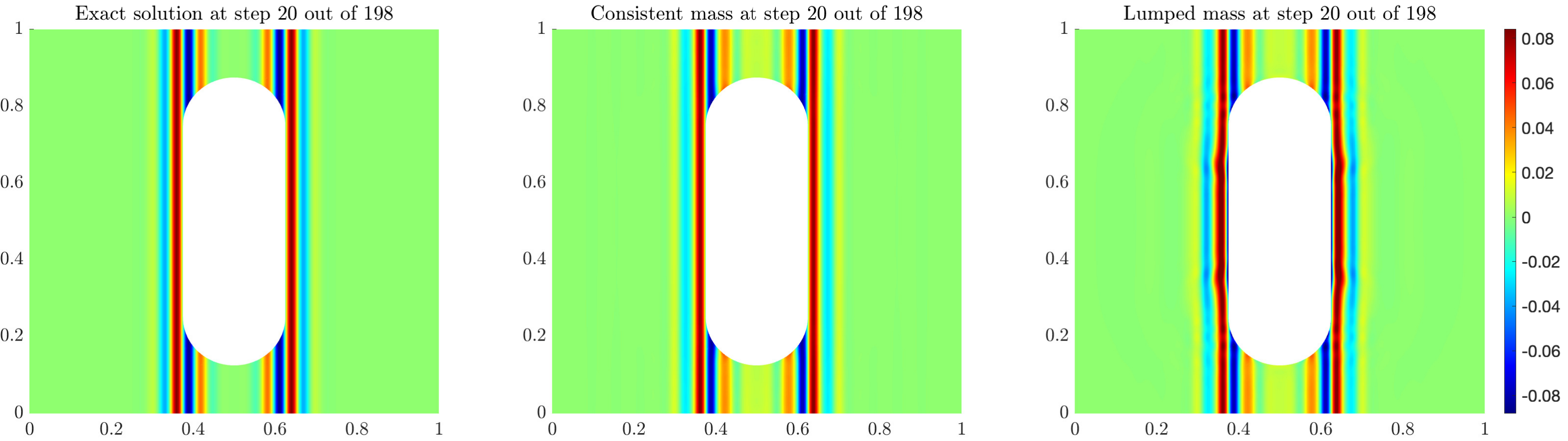}
    \caption{Solution snapshots at time $t=0.2893$}
    \label{fig: 2D_Laplace_trimmed_plate_with_extrusion_dynamics_solution_snapshots_p2_eps_1e-7_s20}
     \end{subfigure}
     \hfill
     \begin{subfigure}[t]{1.0\textwidth}
    \centering
    \includegraphics[width=\textwidth]{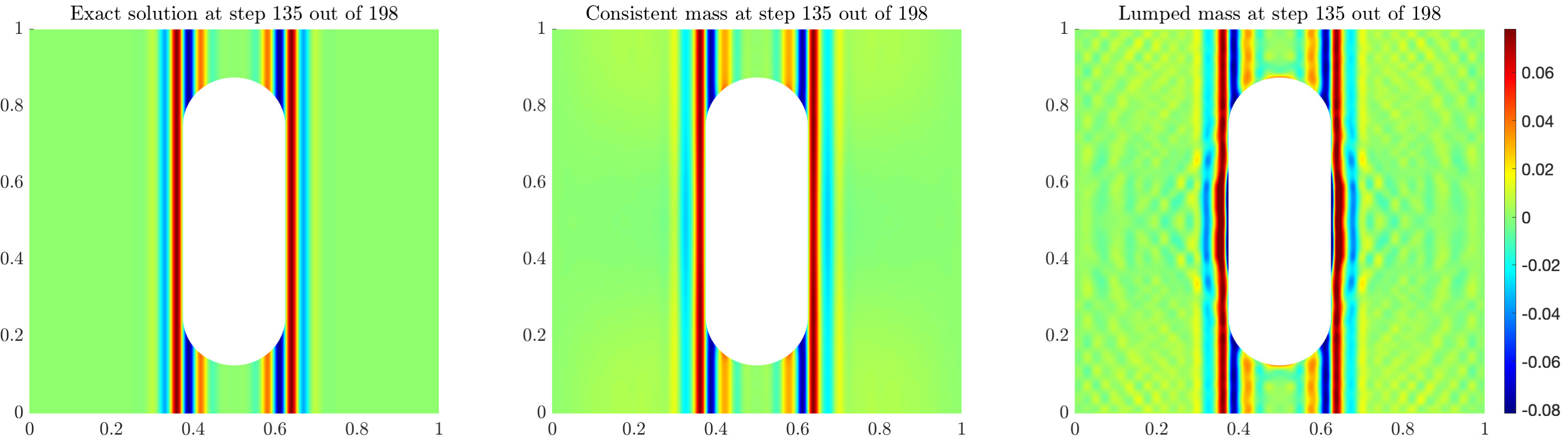}
    \caption{Solution snapshots at time $t=2.0406$}
    \label{fig: 2D_Laplace_trimmed_plate_with_extrusion_dynamics_solution_snapshots_p2_eps_1e-7_s135}
     \end{subfigure}
     \hfill
    \caption{Exact solution and discrete solutions for the consistent and lumped mass for quadratic $C^1$ B-splines}
    \label{fig: 2D_Laplace_trimmed_plate_with_extrusion_dynamics_solution_snapshots_p2_eps_1e-7}
\end{figure}

\begin{figure}[H]
    \centering
    \includegraphics[width=0.5\linewidth]{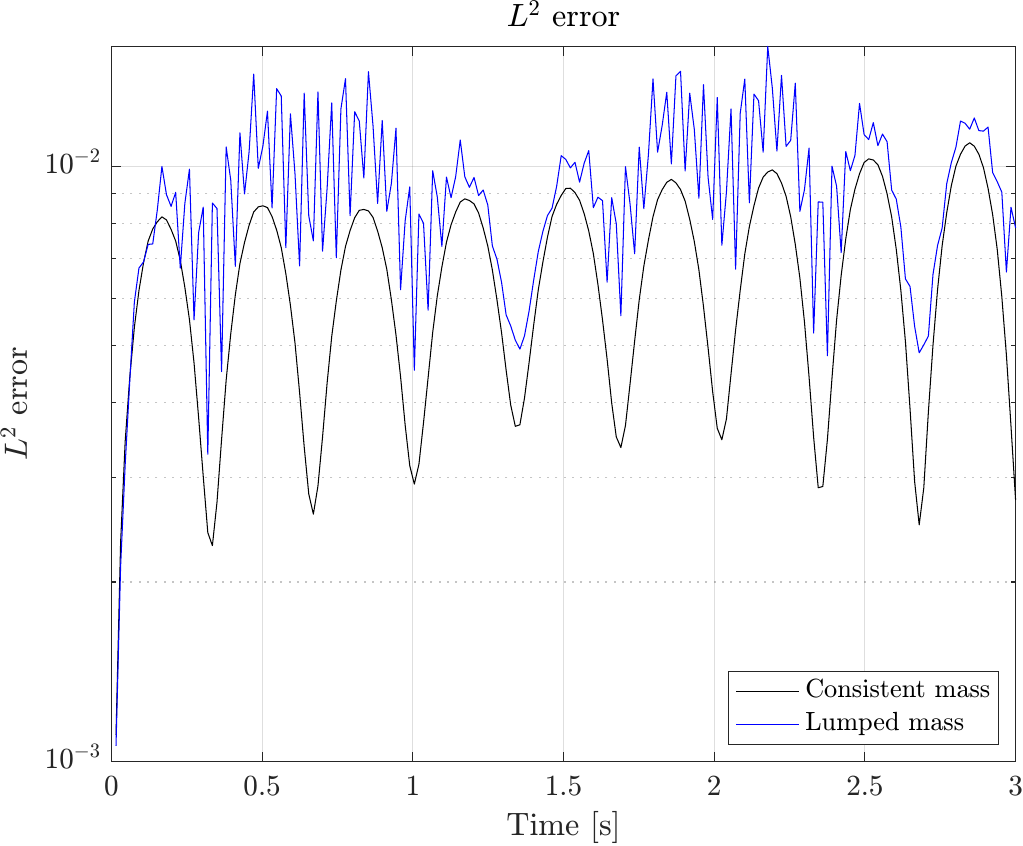}
    \caption{Evolution of the $L^2$ error over time for the consistent and lumped mass solutions}
    \label{fig: 2D_Laplace_trimmed_plate_with_extrusion_L2_error_explicit_p2_eps_1e-7}
\end{figure}
\end{example}

\begin{example}[Perforated plate]
\label{ex: counter_example_plate_with_hole}
As final example, we solve the wave equation on a perforated plate. This geometry is a small variant of the previous one and consists in extruding from the unit square a circle of radius $r=0.125 \sqrt{2}+\epsilon$ centered at $\mathbf{c}=(0.5, 0.5)$. The setup is exemplarily depicted in \Cref{fig: 2D_Laplace_trimmed_plate_with_hole_domain_n32} for a coarse background mesh with $N=32$ subdivisions. Our manufactured solution in this case is $u(\mathbf{x},t)=w(\mathbf{x})\sin(n \pi t)$ with 
\begin{equation}
\label{eq: function_w(x)_plate_with_hole}
    w(\mathbf{x}) = x(x-1)\mathrm{e}^{-\left(\frac{\|\mathbf{x}-\mathbf{c}\|}{\sigma}\right)^2}\sin(g(\mathbf{x}))
\end{equation}
and
\begin{equation*}
    g(\mathbf{x}) = k \mathrm{e}^{-\left(\frac{\|\mathbf{x}-\mathbf{c}\|-0.9r}{\eta}\right)^2}
\end{equation*}
with $n=3$, $\sigma=0.5$, $k=10$, $\eta^2=0.005$ and $\epsilon=10^{-6}$. The function $w(\mathbf{x})$ is shown in \Cref{fig: 2D_Laplace_trimmed_plate_with_hole_function} and satisfies homogeneous Dirichlet boundary conditions on the left and right boundaries. Neumann boundary conditions are prescribed on the remaining boundaries and the solution is approximated with cubic $C^2$ B-splines on a background mesh with $N=56$ subdivisions in each direction. The time interval and the temporal discretization parameters are taken from the previous example. Detailed solution snapshots are shown in \Cref{fig: 2D_Laplace_trimmed_plate_with_hole_dynamics_solution_snapshots_p3_eps_1e-6} at two different times. Some artifacts in the solution for the lumped mass are already visible on two small trimmed elements shortly after the beginning of the simulation and tend to grow over time, as testified in \Cref{fig: 2D_Laplace_trimmed_plate_with_hole_L2_error_explicit_p3_eps_1e-6}.

\begin{figure}[H]
     \centering
     \begin{subfigure}[t]{0.42\textwidth}
    \centering
    \includegraphics[width=\textwidth]{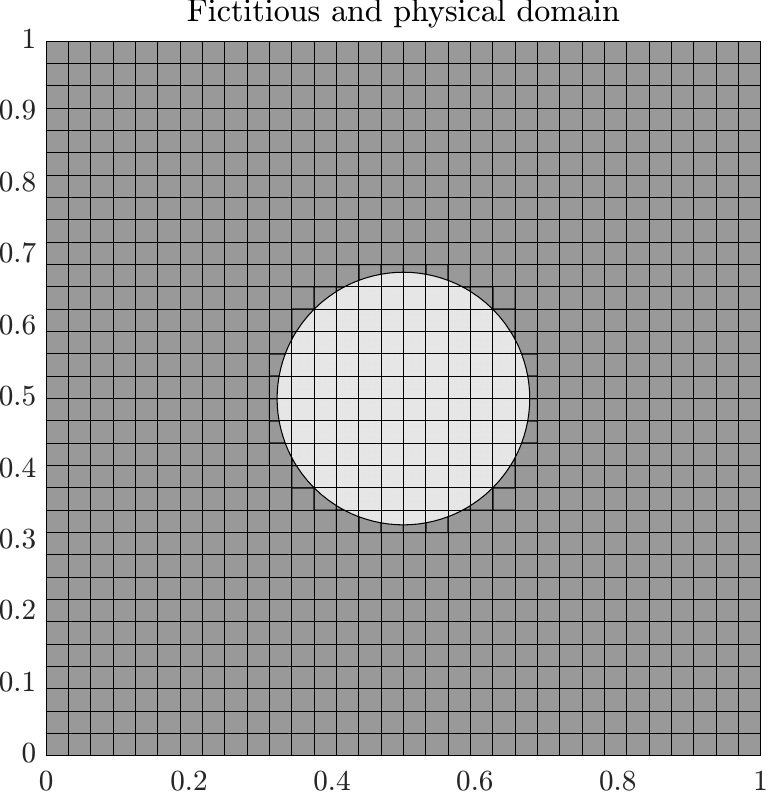}
    \caption{Perforated plate}
    \label{fig: 2D_Laplace_trimmed_plate_with_hole_domain_n32}
     \end{subfigure}
     \hfill
     \begin{subfigure}[t]{0.56\textwidth}
    \centering
    \includegraphics[width=\textwidth]{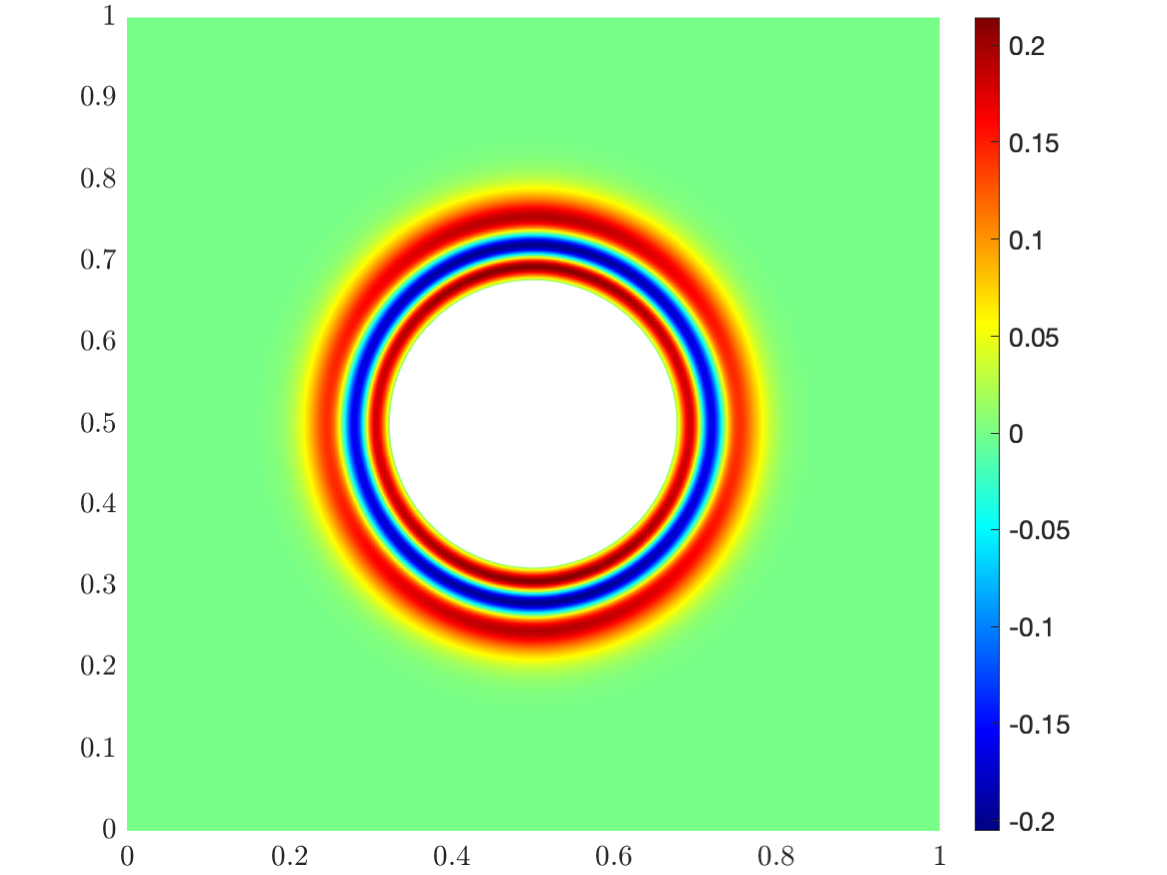}
    \caption{Function $w(\mathbf{x})$ defined in \cref{eq: function_w(x)_plate_with_hole}}
    \label{fig: 2D_Laplace_trimmed_plate_with_hole_function}
     \end{subfigure}
     \hfill
    \caption{Domain and spatial part of the solution}
    \label{fig: 2D_Laplace_perforated_plate_domain_function}
\end{figure}

\begin{figure}[H]
     \centering
     \begin{subfigure}[t]{1.0\textwidth}
    \centering
    \includegraphics[width=\textwidth]{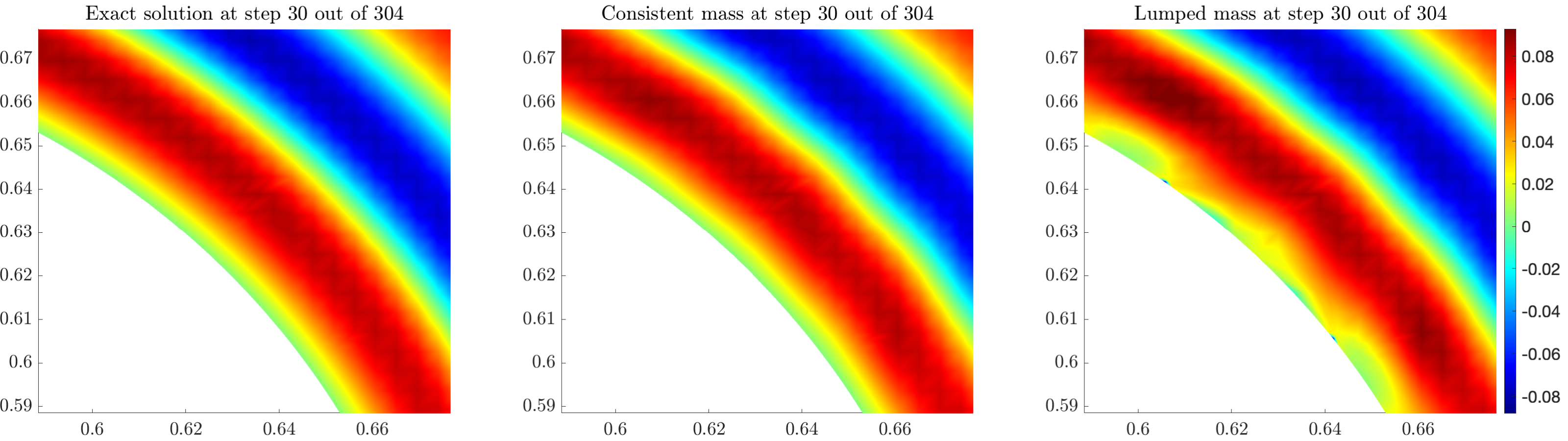}
    \caption{Solution snapshots at time $t=0.2871$}
    \label{fig: 2D_Laplace_trimmed_plate_with_hole_dynamics_solution_snapshots_p3_eps_1e-6_s30}
     \end{subfigure}
     \hfill
     \begin{subfigure}[t]{1.0\textwidth}
    \centering
    \includegraphics[width=\textwidth]{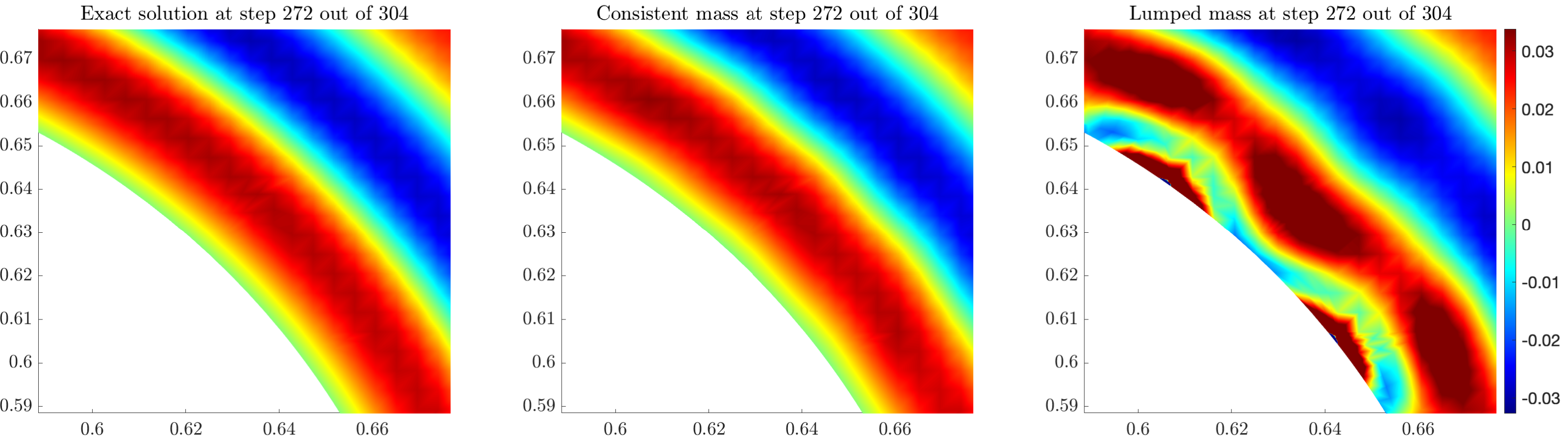}
    \caption{Solution snapshots at time $t=2.6832$}
    \label{fig: 2D_Laplace_trimmed_plate_with_hole_dynamics_solution_snapshots_p3_eps_1e-6_s272}
     \end{subfigure}
     \hfill
    \caption{Exact solution and discrete solutions for the consistent and lumped mass for cubic $C^2$ B-splines}
    \label{fig: 2D_Laplace_trimmed_plate_with_hole_dynamics_solution_snapshots_p3_eps_1e-6}
\end{figure}

\begin{figure}[H]
    \centering
    \includegraphics[width=0.5\linewidth]{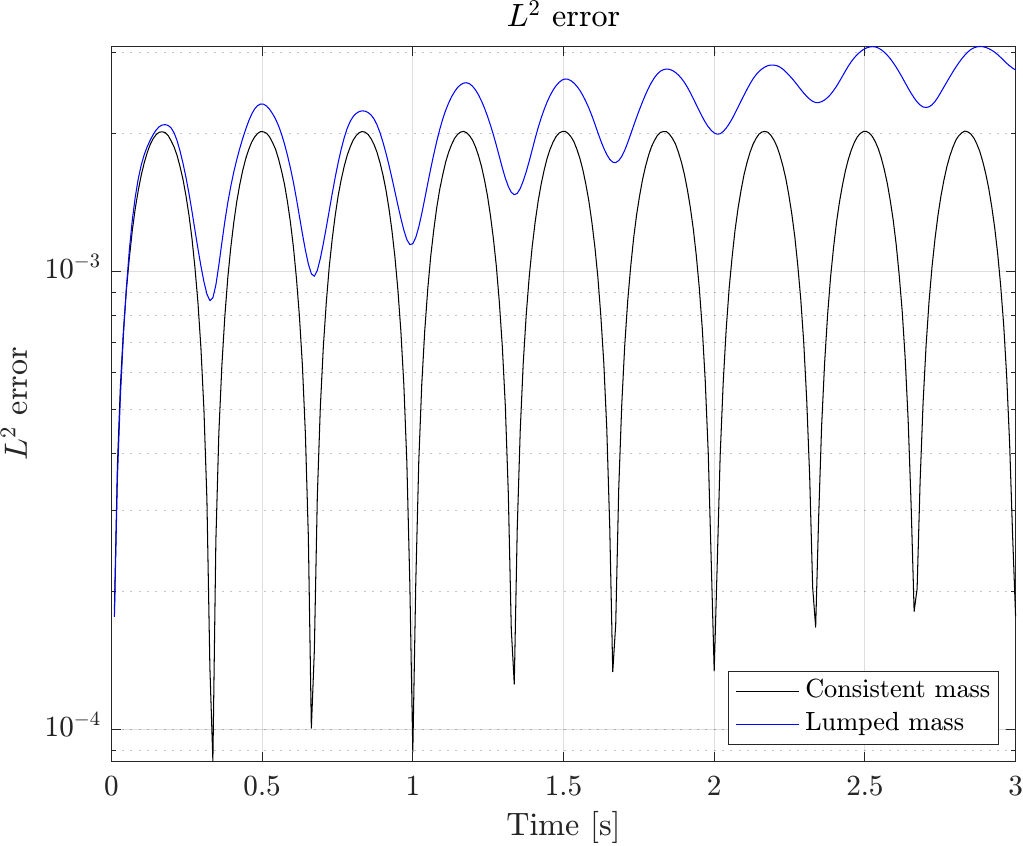}
    \caption{Evolution of the $L^2$ error over time for the consistent and lumped mass solutions}
    \label{fig: 2D_Laplace_trimmed_plate_with_hole_L2_error_explicit_p3_eps_1e-6}
\end{figure}
\end{example}

\section{Analysis}
\label{se: analysis}
This section will shed some light on the oscillations observed in the previous examples. Firstly, we must stress that those oscillations are not an artifact of time integration, but of mass lumping alone. For demonstrating it, we compare the fully discrete solution to the exact solution of the semi-discrete problem; i.e. the coupled second order system of ODEs
\begin{equation}
\label{eq: system_odes}
\begin{cases}
    M\ddot{\mathbf{u}}(t) + K\mathbf{u}(t) = \mathbf{f}(t) \qquad \text{for } t \in [0,T],\\
    \mathbf{u}(0)=\mathbf{u}_0, \\
    \dot{\mathbf{u}}(0) = \mathbf{v}_0,
\end{cases}
\end{equation}
where $T>0$ denotes the final time. The exact solution of \cref{eq: system_odes}, expressed in terms of matrix functions, is structurally similar to its scalar counterpart considered in the next lemma.

\begin{lemma}
\label{lem: scalar_ode}
The general solution of the ODE
\begin{equation}
\label{eq: scalar_ode}
\begin{cases}
    \ddot{u}(t) + \lambda u(t) = f(t) \qquad \text{for } t \in [0,T],\\
    u(0) = u_0, \\
    \dot{u}(0) = v_0,
\end{cases}
\end{equation}
with $f \in C^0([0,T])$, $\lambda \in \mathbb{R}^*_+$ and $u_0,v_0 \in \mathbb{R}$ is given by
\begin{equation*}
    u(t) = u_0 \cos(\sqrt{\lambda} t) + t v_0 \sinc(\sqrt{\lambda} t) + \int_0^t (t-\tau)\sinc(\sqrt{\lambda}(t-\tau))f(\tau) \, d \tau.
\end{equation*}
\end{lemma}
\begin{proof}
The result follows from elementary calculus.
\end{proof}

The solution of \cref{eq: system_odes} is now readily expressed in terms of matrix functions.

\begin{corollary}
\label{cor: semi_discrete_solution}
The exact solution of \cref{eq: system_odes} for $\mathbf{f} \in C^0([0,T])$ and $M,K$ symmetric positive definite is given by
\begin{equation}
\label{eq: semi_discrete_solution}
    \mathbf{u}(t)=\cos(\sqrt{A}t)\mathbf{u}_0+t\sinc(\sqrt{A}t)\mathbf{v}_0+\int_0^t (t-\tau)\sinc(\sqrt{A}(t-\tau))M^{-1}\mathbf{f}(\tau) \, d \tau
\end{equation}
where $A=M^{-1}K$.
\end{corollary}
\begin{proof}
There exist multiple equivalent ways of deducing the result, e.g. via the spectral decomposition of $M^{-1}K$, or in a more symmetric way using the eigenbasis of $(K,M)$ or the Cholesky decomposition of $M$. We will employ the eigenbasis approach. It is well-known that if $(K,M)$ is a symmetric matrix pair and $K$ and $M$ are positive definite, then there exists an invertible matrix of eigenvectors $U$ such that
\begin{equation}
\label{eq: joint_diagonalization}
    U^TKU = D, \qquad U^TMU = I,
\end{equation}
where $D=\diag(\lambda_1,\dots,\lambda_n)$ is the diagonal matrix of positive eigenvalues (see e.g. \citep[][Theorem VI.1.15]{stewart1990matrix}) and $KU=MUD$ is the associated eigendecomposition. In particular, the matrix $U$ forms an $M$-orthonormal basis of $\mathbb{R}^n$. We change basis and set $\mathbf{u}(t)=U\mathbf{x}(t)$ (i.e. $\mathbf{x}(t)=U^{-1}\mathbf{u}(t)$). Substituting this expression in \cref{eq: system_odes}, pre-multiplying by $U^T$ and using \cref{eq: joint_diagonalization}, we obtain the set of uncoupled equations
\begin{equation*}
\begin{cases}
    \ddot{\mathbf{x}}(t) + D\mathbf{x}(t) = U^T\mathbf{f}(t) \\
    \mathbf{x}(0)=U^{-1}\mathbf{u}_0 \\
    \dot{\mathbf{x}}(0) = U^{-1}\mathbf{v}_0
\end{cases}
\end{equation*}
whose exact solution, thanks to \Cref{lem: scalar_ode}, is given by
\begin{equation}
\label{eq: eigenbasis_comp}
    \mathbf{x}(t)=\cos(\sqrt{D}t)U^{-1}\mathbf{u}_0+t\sinc(\sqrt{D}t)U^{-1}\mathbf{v}_0+\int_0^t (t-\tau)\sinc(\sqrt{D}(t-\tau))U^T\mathbf{f}(\tau) \, d \tau.
\end{equation}
The result now follows after back-transforming to the original variables by pre-multiplying by $U$
\begin{equation*}
    \mathbf{u}(t)=U\mathbf{x}(t)=U\cos(\sqrt{D}t)U^{-1}\mathbf{u}_0+tU\sinc(\sqrt{D}t)U^{-1}\mathbf{v}_0+\int_0^t (t-\tau)U\sinc(\sqrt{D}(t-\tau))U^{-1}UU^T\mathbf{f}(\tau) \, d \tau
\end{equation*}
and noting that $M^{-1}K=UDU^{-1}$ is the spectral decomposition of $M^{-1}K$ and $M^{-1}=UU^T$. 
\end{proof}

Practically speaking, \Cref{cor: semi_discrete_solution} is rather inconvenient but serves as the starting point for exponential integrators \cite{hochbruck1999gautschi,grimm2005error,grimm2006error}, which are another class of numerical methods for solving \cref{eq: system_odes}. Nevertheless, the result nicely simplifies for special choices of right-hand sides:
\begin{enumerate}
    \item If $\mathbf{f}(t)=\mathbf{b}$ is a constant vector, then, similarly to the scalar case, the solution reduces to (see e.g. \cite{voet2020computation})
    \begin{equation}
    \label{eq: exact_sol_case_1}
    \mathbf{u}(t)=\cos(\sqrt{A}t)\mathbf{u}_0+t\sinc(\sqrt{A}t)\mathbf{v}_0+ \frac{t^2}{2}\sinc^2\left(\frac{\sqrt{A}t}{2}\right)M^{-1}\mathbf{b}.
    \end{equation}
    \item If $\mathbf{f}(t)=\sin(\tilde{\omega} t)\mathbf{b}$ and $\tilde{\omega}^2 \notin \Lambda(K,M)$, then the solution reduces to
    \begin{equation}
    \label{eq: exact_sol_case_2}
        \mathbf{u}(t)=\cos(\sqrt{A}t)\mathbf{u}_0+t\sinc(\sqrt{A}t)\mathbf{v}_0+ (\tilde{\omega}^2I-A)^{-1}(\tilde{\omega} t \sinc(\sqrt{A}t) - \sin(\tilde{\omega}t)I)M^{-1}\mathbf{b}.
    \end{equation}
\end{enumerate}

Due to the separability of the manufactured solution in \Cref{ex: counter_example_1D}, the right-hand side vector takes the form described in the second case above and therefore the exact solution is known explicitly. We now reproduce \Cref{fig: 1D_Laplace_trimming_dynamics_solution_snapshots_p3_eps_1e-6_Cp_1_explicit} by comparing the fully discrete solution to the exact semi-discrete one \eqref{eq: exact_sol_case_2} for the lumped mass. The results, shown in \Cref{fig: 1D_Laplace_trimming_dynamics_solution_snapshots_p3_eps_1e-6_exact}, indicate that the exact semi-discrete solution completely overlaps with the fully discrete one and exclude time integration as a reason for the discrepancy.

\begin{figure}[H]
    \centering
    \includegraphics[width=\linewidth]{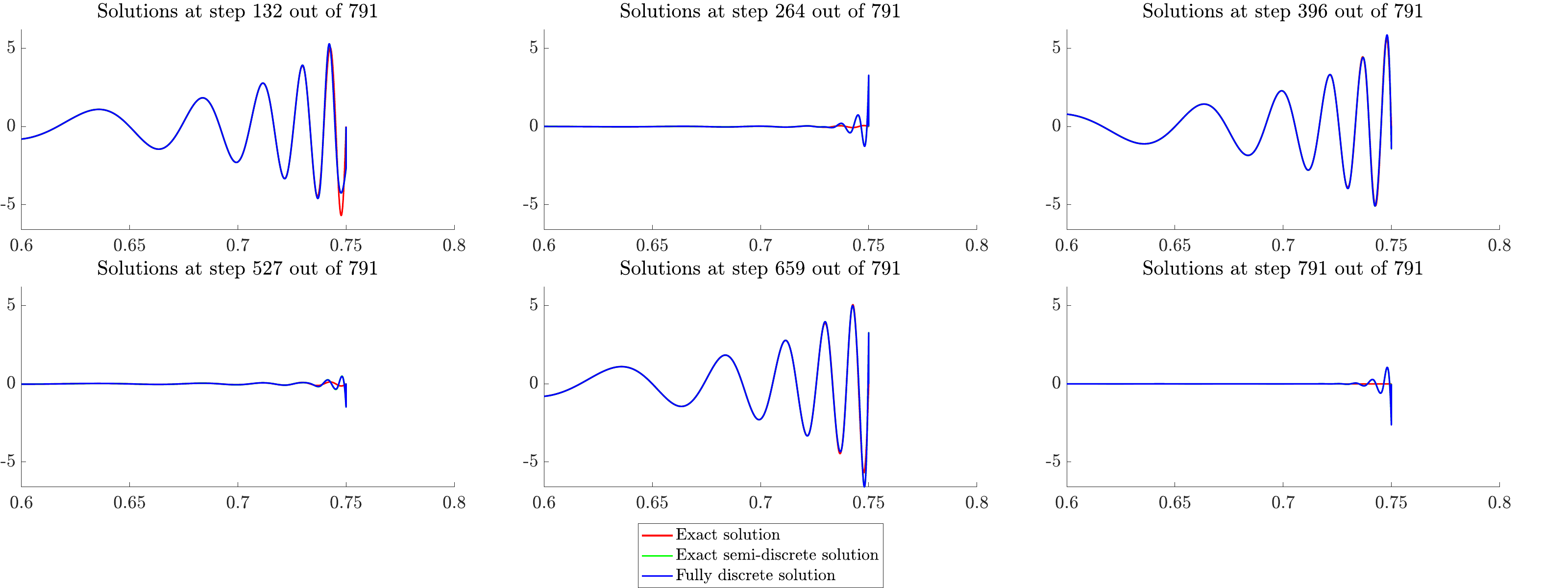}
    \caption{Exact solution, exact semi-discrete solution and fully discrete solution for the lumped mass in \Cref{ex: counter_example_1D} for cubic $C^2$ B-splines}
    \label{fig: 1D_Laplace_trimming_dynamics_solution_snapshots_p3_eps_1e-6_exact}
\end{figure}

From now on, we will focus on the spatial discretization error. Our error analysis will (partially) explain the differences observed in the examples of \Cref{se: motivation} between the consistent and lumped mass approximations. While the material in the first part of the analysis is rather standard and closely follows the presentation in \cite{hughes2014finite}, we later derive refined bounds on the modal error, drastically simplifying the results from \cite{hughes2014finite}, and highlighting major differences between the consistent and lumped mass approximations. For this purpose, let us return to the weak form \eqref{eq: continuous_weak_form} (with $h=0$ for simplicity): find $u \in V_T$ such that for almost every time $t \in (0,T)$
\begin{align}
 \langle\partial_{tt} u(t), v\rangle + (\nabla u(t), \nabla v)_{L^2(\Omega)^d} &= (f(t),v)_{L^2(\Omega)} + (h(t),v)_{L^2(\partial \Omega_N)} & \forall v \in V, \label{eq: variational_form} \\
 (u(0),v)_{L^2(\Omega)} &= (u_0,v)_{L^2(\Omega)} & \forall v \in V, \nonumber \\
 \langle \partial_t u(0),v \rangle &= \langle v_0,v \rangle & \forall v \in V. \nonumber
\end{align}
We may construct a basis for $V$ by solving the associated generalized eigenvalue problem
\begin{equation}
\label{eq: gen_eig_pb}
    (\nabla u, \nabla v)_{L^2(\Omega)^d}=\lambda (u,v)_{L^2(\Omega)} \quad \forall v \in V 
\end{equation}
Let $\{(\lambda_i, u_i)\}_{i=1}^{\infty}$ denote the eigenvalue/eigenfunction pairs solution of \eqref{eq: gen_eig_pb}. Since the bilinear forms involved are positive definite, all eigenvalues are real positive and we may assume they are ordered in increasing algebraic order such that $0<\lambda_1 \leq \lambda_2 \leq \dots$. Moreover, the (normalized) eigenfunctions form an $L^2$-orthonormal basis for $V$ such that the following relations hold:
\begin{equation*}
    (\nabla u_i, \nabla u_j)_{L^2(\Omega)^d}=\lambda_i \delta_{ij}, \qquad (u_i,u_j)_{L^2(\Omega)}=\delta_{ij}.
\end{equation*}
By expanding the solution $u$ in the eigenfunction basis $u(\mathbf{x},t)=\sum_{j=1}^{\infty} d_j(t) u_j(\mathbf{x})$, substituting the expression back in \cref{eq: variational_form}, testing for $v=u_i$ and using the orthogonality relations, we obtain the uncoupled system of equations
\begin{equation}
\label{eq: uncoupled_system}
\begin{cases}
    \ddot{d}_i(t) + \lambda d_i(t) = f_i(t) \qquad \text{for } t \in [0,T],\\
    d_i(0) = u_{0,i}, \\
    \dot{d}_i(0) = v_{0,i},
\end{cases}
\end{equation}
where we have denoted $f_i=(f,u_i)_{L^2}$, $u_{i,0}=(u_0, u_i)_{L^2}$ and $v_{i,0}=\langle v_0,u_i \rangle$ the coefficients for the $L^2$ projection of the right-hand side and initial conditions in the eigenbasis. The solution of \cref{eq: uncoupled_system}, according to \Cref{lem: scalar_ode}, is given by
\begin{equation}
\label{eq: di}
    d_i(t)=u_{i,0}\cos(\omega_i t)+\frac{v_{i,0}}{\omega_i}\sin(\omega_i t)+\frac{1}{\omega_i}\int_0^t \sin(\omega_i(t-\tau))f_i(\tau) \mathrm{d} \tau \quad i=1,2,\dots, \infty,
\end{equation}
where $\omega_i=\sqrt{\lambda_i}$ denotes the $i$th exact eigenfrequency. The previous steps still hold if one seeks an approximate solution $u^h$ in a finite dimensional subspace $V^h \subset V$ with dimension $n=\dim(V^h)$. The coefficients of the expansion $u^h(\mathbf{x},t)=\sum_{j=1}^n d^h_j(t) u^h_j(\mathbf{x})$ are then similarly given by
\begin{equation}
\label{eq: di^h}
    d^h_i(t)=u^h_{i,0}\cos(\omega^h_i t)+\frac{v^h_{i,0}}{\omega^h_i}\sin(\omega^h_i t)+\frac{1}{\omega^h_i}\int_0^t \sin(\omega^h_i(t-\tau))f^h_i(\tau) \mathrm{d} \tau \quad i=1,2,\dots, n,
\end{equation}
where quantities with a superscript $h$ are defined analogously as before. Note that \cref{eq: di^h} is equivalent to the matrix expression \eqref{eq: eigenbasis_comp} derived previously. The approximation error is then given by
\begin{equation*}
    u^h(\mathbf{x},t)-u(\mathbf{x},t)=\sum_{j=1}^{n} d_j^h(t)u^h_j(\mathbf{x})-\sum_{j=1}^{\infty} d_j(t)u_j(\mathbf{x})= e(\mathbf{x},t) - \eta(\mathbf{x},t).
\end{equation*}
where 
\begin{equation*}
    e(\mathbf{x},t)= \sum_{j=1}^n e_j(\mathbf{x},t) = \sum_{j=1}^n d^h_j(t)u^h_j(\mathbf{x})-d_j(t)u_j(\mathbf{x}) \quad \text{and} \quad \eta(\mathbf{x},t)=\sum_{j=n+1}^{\infty}d_j(t)u_j(\mathbf{x})
\end{equation*}
are the resolved and unresolved parts of the approximation error, respectively. Clearly, some of the approximated modes must contribute to the dynamics for the approximation to make any sense. For the isogeometric Galerkin method, either $h$, $p$, $k$-refinement or any combination thereof will eventually lead to an unresolved component whose norm is smaller than a prescribed tolerance. The analysis conducted in \cite{hughes2014finite} then focuses on bounding the modal error $e_j(\mathbf{x},t)=d^h_j(t)u^h_j(\mathbf{x})-d_j(t)u_j(\mathbf{x})$ as
\begin{equation}
\label{eq: modal_errors}
    \|e\| \leq \sum_{j=1}^n \|e_j\|
\end{equation}
holds in any norm. However, the simple triangle inequality in \eqref{eq: modal_errors} does not account for potential (favorable) cancellation of modal errors. Moreover, the pairing of the $j$th exact modal component with the $j$th approximate one also implicitly assumes a suitable labeling of the approximate eigenpairs. Namely, the $j$th approximate eigenpair is indeed an approximation of the $j$th exact one. While this usually holds true for the consistent mass, it typically does not for the lumped mass, as shown in \Cref{ex: counter_example_1D}. The following theorem, which holds only for the consistent mass, provides a refined bound on the modal error, improving the results in \cite{hughes2014finite}. Its proof is delayed to \Cref{se: error_analysis}, where the elliptic case is also treated for completeness.

\begin{theorem}
\label{th: modal_error_hyperbolic}
For a consistent mass approximation, the $L^2$ norm of the modal error is bounded by
\begin{flalign}
\|e_j\|_{L^2} &\leq \|u_0\|_{L^2} \left(2\|u^h_j-u_j\|_{L^2}+|\cos(\omega^h_j t)-\cos(\omega_j t)|\right) & \label{eq: error_displacement} \\
 &+\frac{\|v_0\|_{L^2}}{\omega_j}\left(\frac{\omega^h_j-\omega_j}{\omega_j}+2\|u^h_j-u_j\|_{L^2}+|\sin(\omega^h_j t)-\sin(\omega_j t)|\right) & \label{eq: error_velocity} \\
 &+\frac{1}{\omega_j} \int_{0}^t \|f(\tau)\|_{L^2} \mathrm{d} \tau \left( \frac{\omega^h_j-\omega_j}{\omega_j}+2\|u^h_j-u_j\|_{L^2} + \max_{\tau \in [0,t]} |\sin(\omega^h_j \tau)-\sin(\omega_j \tau)| \right). & \label{eq: error_forcing}
\end{flalign}
\end{theorem}
In \Cref{th: modal_error_hyperbolic}, \cref{eq: error_displacement,eq: error_velocity} are the error contributions from the initial conditions, while \cref{eq: error_forcing} is the error contribution from the forcing term $f$. The conclusions of \Cref{th: modal_error_hyperbolic} are the same as in \cite{hughes2014finite} and reveal the danger looming over Galerkin approximations of \cref{eq: wave_equation} (and hyperbolic PDEs more generally): contrary to \cref{eq: error_velocity,eq: error_forcing}, where errors in the high frequencies and modes are dampened away thanks to the $1/\omega_j$ factor, if $u_0 \neq 0$, those same errors may now degrade the solution. This finding is a sharp contrast to the elliptic and parabolic cases and was another argument in \cite{hughes2014finite} for adopting IGA, thanks to its enhanced approximation of higher frequencies and modes. Nevertheless, this issue practically only occurs if $u_0$ has a significant component along one of the inaccurate modes, which is not captured by \Cref{th: modal_error_hyperbolic}. Therefore, while this theorem certainly highlights the difficulties of the hyperbolic case, its practical value is limited. \Cref{eq: di^h} instead provides a better, albeit more intuitive, understanding. Most importantly, \cref{eq: di^h} explicitly features $u^h_{i,0}, v^h_{i,0}$ and $f^h_i$, the coefficients for the $L^2$ projection of the initial conditions and forcing term along the $i$th approximate eigenmode, which were lost in the proof of \Cref{th: modal_error_hyperbolic}.

In all examples of \Cref{se: motivation}, $u_0=0$ and the first term in \eqref{eq: di^h} vanishes. However, $v_0$, $f$ and $h$ (which prescribes Neumann boundary conditions) are all nonzero and the latter is incorporated in the right-hand side of the weak form. In all examples, the second initial condition is $v_0(\mathbf{x})=C q(\mathbf{x})$, where $q$ is the spatial part of the manufactured solution and $C$ is a constant. As shown in \Cref{fig: 1D_Laplace_trimming_projection_eigKM_p3_eps_1e-6}, the projection of $q(\mathbf{x})$ relative to the eigenbasis for the consistent mass activates a broad range of frequencies, including high frequencies associated to spurious modes. However, their contribution is essentially wiped out thanks to $\omega_i$ appearing in the denominator in \cref{eq: di^h}. The same holds true for the contribution from the forcing term. Thus, the consistent mass provides an exceedingly good approximation, despite spurious modes being activated. Unfortunately, the same cannot be said for the lumped mass approximation, which is obtained by substituting the bilinear form $b(u^h,v^h)$ with the approximate one $\tilde{b} \colon V^h \times V^h \to \mathbb{R}$, defined as
\begin{equation*}
    \tilde{b}(u^h,v^h) = \mathbf{u}^T \mathcal{L}(M) \mathbf{v},
\end{equation*}
where $\mathbf{u}$ and $\mathbf{v}$ are the coefficient vectors of $u_h$ and $v_h$, respectively, in the B-spline basis. It is important to note that the lumped mass is defined algebraically from the consistent mass and explicitly depends on the B-spline basis. Therefore, $\tilde{b}(u^h,v^h)$ is only defined for functions in $V^h$ and generally not in the larger Sobolev space $V$. Nevertheless, thanks to the nonnegativity of the B-spline basis, $\tilde{b}(u^h,v^h)$ is positive definite and induces a discrete inner product. Therefore, there exist approximate eigenvalue/eigenfunction pairs $\{(\tilde{\lambda}^h_i, \tilde{u}^h_i)\}_{i=1}^n$ such that
\begin{equation*}
    a(\tilde{u}^h_i, \tilde{u}^h_j)=\tilde{\lambda}^h_i \delta_{ij}, \qquad \tilde{b}(\tilde{u}^h_i,\tilde{u}^h_j)=\delta_{ij}.
\end{equation*}
holds. In particular, the eigenbases $\mathcal{U}=\{u^h_i\}_{i=1}^n$ and $\tilde{\mathcal{U}}=\{\tilde{u}^h_i\}_{i=1}^n$ are both orthonormal bases for $V^h$, albeit in different inner products. The derivation then essentially follows the same lines by expanding the approximate solution $\tilde{u}^h(\mathbf{x},t)$ in the approximate eigenbasis $\tilde{\mathcal{U}}$ and testing against $v \in \tilde{\mathcal{U}}$. Eventually, \cref{eq: di^h} still holds after crowning all quantities with tildes. However, note that $\tilde{u}^h_{i,0}$ and $\tilde{v}^h_{i,0}$ are now the coefficients for the $L^2$ projection of the initial conditions relative to the eigenbasis $\tilde{\mathcal{U}}$.

Denoting $\Pi \colon V \to V^h$ the standard $L^2$ projection, \Cref{fig: 1D_Laplace_trimming_projection_p3_eps_1e-6} compares the coefficients of $\Pi(q)$ relative to the eigenbases $\mathcal{U}$ and $\tilde{\mathcal{U}}$ for \Cref{ex: counter_example_1D}. While the coefficients of $\Pi(q)$ relative to $\mathcal{U}$ are uniformly distributed over the entire frequency spectrum, the second coefficient of $\Pi(q)$ relative to $\tilde{\mathcal{U}}$ dominates all others. Yet, according to \Cref{fig: 1D_Laplace_trimming_dynamics_p3_eps_1e-6_eigenfunctions}, the corresponding eigenmode is completely spurious and contrary to the consistent mass, it is associated to a small frequency. Thus, its contribution will now unfortunately transpire in the solution.

\begin{figure}[H]
     \centering
     \begin{subfigure}[t]{0.48\textwidth}
    \centering
    \includegraphics[width=\textwidth]{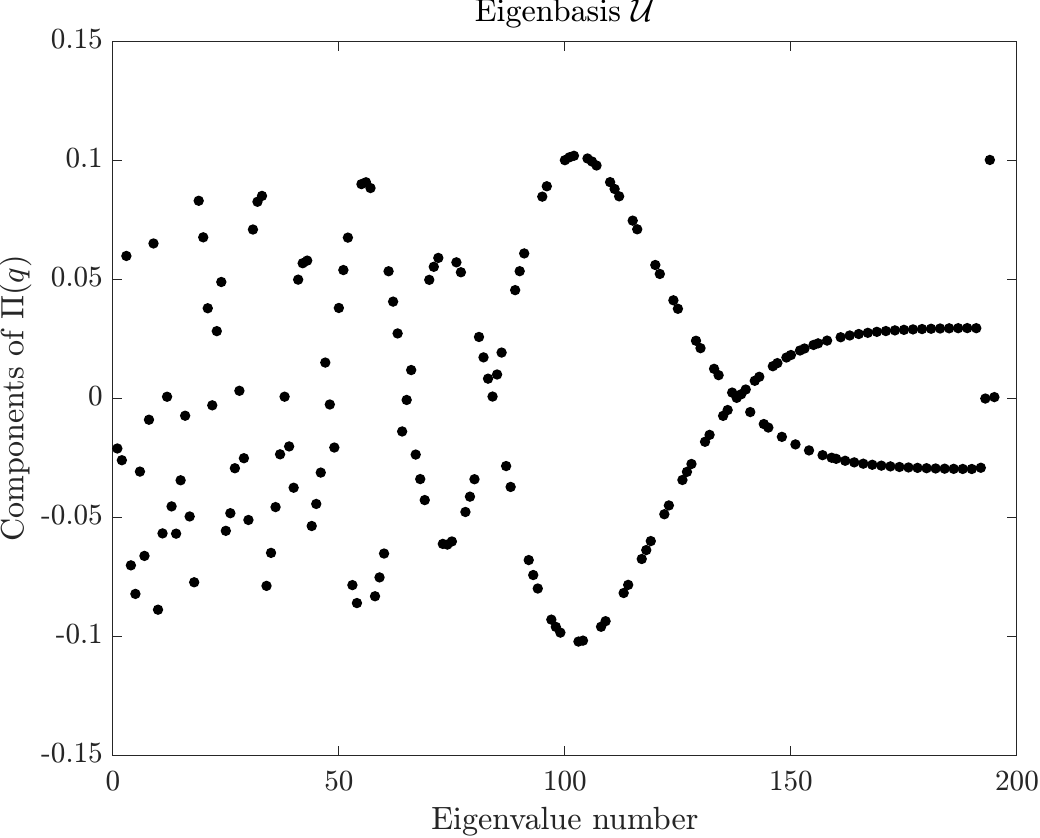}
    \caption{In $\mathcal{U}$}
    \label{fig: 1D_Laplace_trimming_projection_eigKM_p3_eps_1e-6}
     \end{subfigure}
     \hfill
     \begin{subfigure}[t]{0.48\textwidth}
    \centering
    \includegraphics[width=\textwidth]{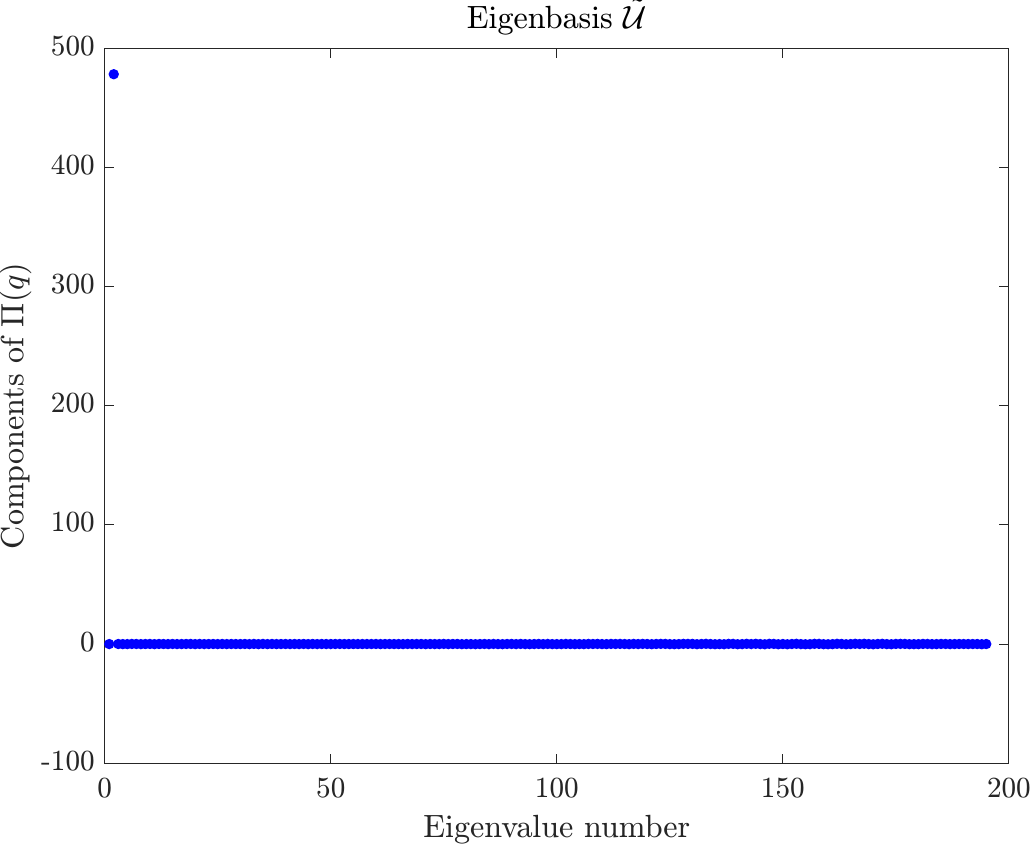}
    \caption{In $\tilde{\mathcal{U}}$}
    \label{fig: 1D_Laplace_trimming_projection_eigKP_p3_eps_1e-6}
     \end{subfigure}
     \hfill
    \caption{Coefficients of $\Pi(q)$ for \Cref{ex: counter_example_1D}}
    \label{fig: 1D_Laplace_trimming_projection_p3_eps_1e-6}
\end{figure}

The same issue arises for the contribution from the right-hand side but its magnitude is much smaller. Thus, in \Cref{ex: counter_example_1D}, the discrete solution for the lumped mass essentially behaves as the spurious mode $u_2^h$ with a sinusoidal amplitude (see \cref{eq: di^h}):
\begin{equation*}
    \tilde{u}^h(\mathbf{x},t) \approx \frac{\tilde{v}^h_{2,0}}{\tilde{\omega}^h_2}\sin(\tilde{\omega}^h_2 t) \tilde{u}_2^h(\mathbf{x}).
\end{equation*}
Both the frequency and amplitude of $a(t)=\frac{\tilde{v}^h_{2,0}}{\tilde{\omega}^h_2}\sin(\tilde{\omega}^h_2 t)$ depend on the spurious frequency $\tilde{\omega}^h_2$. Yet, the analysis conducted in \cite{bioli2024theoretical} recently indicated that the smallest discrete eigenvalue for a lumped mass approximation behaves as $\mathcal{O}(\epsilon^{p-2})$ if $p>2$. This finding explains the pattern observed in \Cref{fig: 1D_Laplace_trimming_dynamics_L2_error_LM_p3_explicit}. Indeed, for $p=3$, $\tilde{\omega}^h_2 = \sqrt{\tilde{\lambda}^h_2} = \mathcal{O}(\sqrt{\epsilon})$, which means that decreasing $\epsilon$ by a factor $100$ roughly decreases the frequency and increases the amplitude of $a(t)$ be a factor $10$ while $\tilde{u}_2^h$ barely changes at all. Thus $\tilde{u}^h(\mathbf{x},t)$ becomes increasingly inaccurate as $\epsilon \to 0$ or $p$ increases and it directly transpires in the $L^2$ error shown in \Cref{fig: 1D_Laplace_trimming_dynamics_L2_error_LM_p3_explicit}. This qualitative argument also holds in higher dimensions for the other examples.

In conclusion, the staggering loss of accuracy for the lumped mass formulation results from spurious ``low-energy'' modes; i.e. spurious modes associated to small eigenvalues. Although spurious modes also appear for a consistent mass formulation, they are associated to large outlier eigenvalues, which effectively filter out the corresponding modal contributions (except for the first initial condition). On the contrary, for the lumped mass formulation, nothing prevents spurious modes from entering the solution if they are initially activated. As we have seen in \Cref{se: motivation}, spurious modes may be activated in unfortunate circumstances. In the next section, we propose a stabilization technique to prevent spurious low-energy modes from surfacing at all.

\section{Stabilization technique}
\label{se: stabilization}
Since diverging eigenvalues for the consistent mass and decaying eigenvalues for the lumped mass have exactly the same origin \cite{bioli2024theoretical}, the most convenient way of preventing both effects is by modifying the discretization prior to mass lumping. Polynomial extension techniques are a known remedy against small trimmed elements \cite{hollig2001weighted,hollig2003finite,marussig2017stable,buffa2020minimal,burman2023extension}. Roughly speaking, the original idea from \cite{buffa2020minimal} is to identify badly cut elements and extend the B-splines from good neighboring elements into the bad elements, thereby replacing the troublesome basis functions $B$ whose active support $\supp(B) \cap \Omega$ is very small. The idea, applicable to smooth spline spaces as well as standard $C^0$ finite elements, is formulated more precisely in this section and largely follows the presentation in \cite{buffa2020minimal,burman2023extension}.

We consider a tensor product spline space $\widehat{\mathcal{S}}_{h,p,k}$ defined over a uniform B-spline tensor product mesh $\widehat{\mathcal{T}}_{h}$ discretizing the fictitious domain $\widehat{\Omega}$ and parameterized by a mesh size $h>0$, spline order $p$ and continuity $ 0 \leq k \leq p-1$. Let $\widehat{\mathcal{B}}=\{ B_i \}_{i=1}^{\widehat{n}}$ denote the B-spline basis constructed over the background mesh $\widehat{\mathcal{T}}_{h}$ and $\widehat{n} = \dim(\widehat{\mathcal{S}}_{h,p,k})$. Trimming curves (or surfaces) cut through the fictitious domain and define the physical domain, as illustrated in \Cref{fig: trimming_notation_1}. Those curves and surfaces only depend on the geometry description, not its discretization. Consequently, trimming may create arbitrarily small elements and basis functions whose support only contains such elements typically cause great stability and conditioning issues and must be carefully treated \cite{marussig2017stable,de2017condition,de2019preconditioning,de2023stability,buffa2020minimal}. Firstly, we define the so-called active mesh
\begin{equation*}
    \mathcal{T}_h = \{ T \in \widehat{\mathcal{T}}_h \colon T \cap \Omega \neq \emptyset \},
\end{equation*}
that defines the computational domain $\Omega_h = \bigcup_{T \in \mathcal{T}_h} T$ (see \Cref{fig: trimming_notation_1}). Secondly, we refer to the active basis
\begin{equation*}
    \mathcal{B}= \{ B \in \widehat{\mathcal{B}} \colon \supp(B) \cap \Omega \neq \emptyset \}
\end{equation*}
as the set of all basis functions whose support intersects (nontrivially) the physical domain. Let $I$ denote its index set and $\mathcal{S}_{h,p,k} = \Span(\mathcal{B})$ the associated spline space. The support of some active basis functions may only intersect the physical domain over small trimmed elements. The idea for circumventing this issue, originality proposed in \cite{buffa2020minimal} and inspired from \cite{haslinger2009new}, is to partition the active mesh in ``good'' (or large) and ``bad'' (or small) elements. Given a threshold tolerance $\gamma \in [0,1]$, an element $T \in \mathcal{T}_h$ is labeled as ``good'' (or large) if
\begin{equation}
\label{eq: good_element}
    |T \cap \Omega | \geq \gamma |T|
\end{equation}
and is ``bad'' (or small) otherwise. We denote $\mathcal{T}_h^L \subseteq \mathcal{T}_h$ the set of large elements and $\mathcal{T}_h^S = \mathcal{T}_h \setminus \mathcal{T}_h^L$ the set of small elements. \Cref{fig: trimming_notation_2} illustrates the partitioning of the mesh into good and bad elements for some small value of $\gamma$.

\begin{figure}[H]
     \centering
     \begin{subfigure}[t]{0.48\textwidth}
    \centering
    \includegraphics[width=\textwidth]{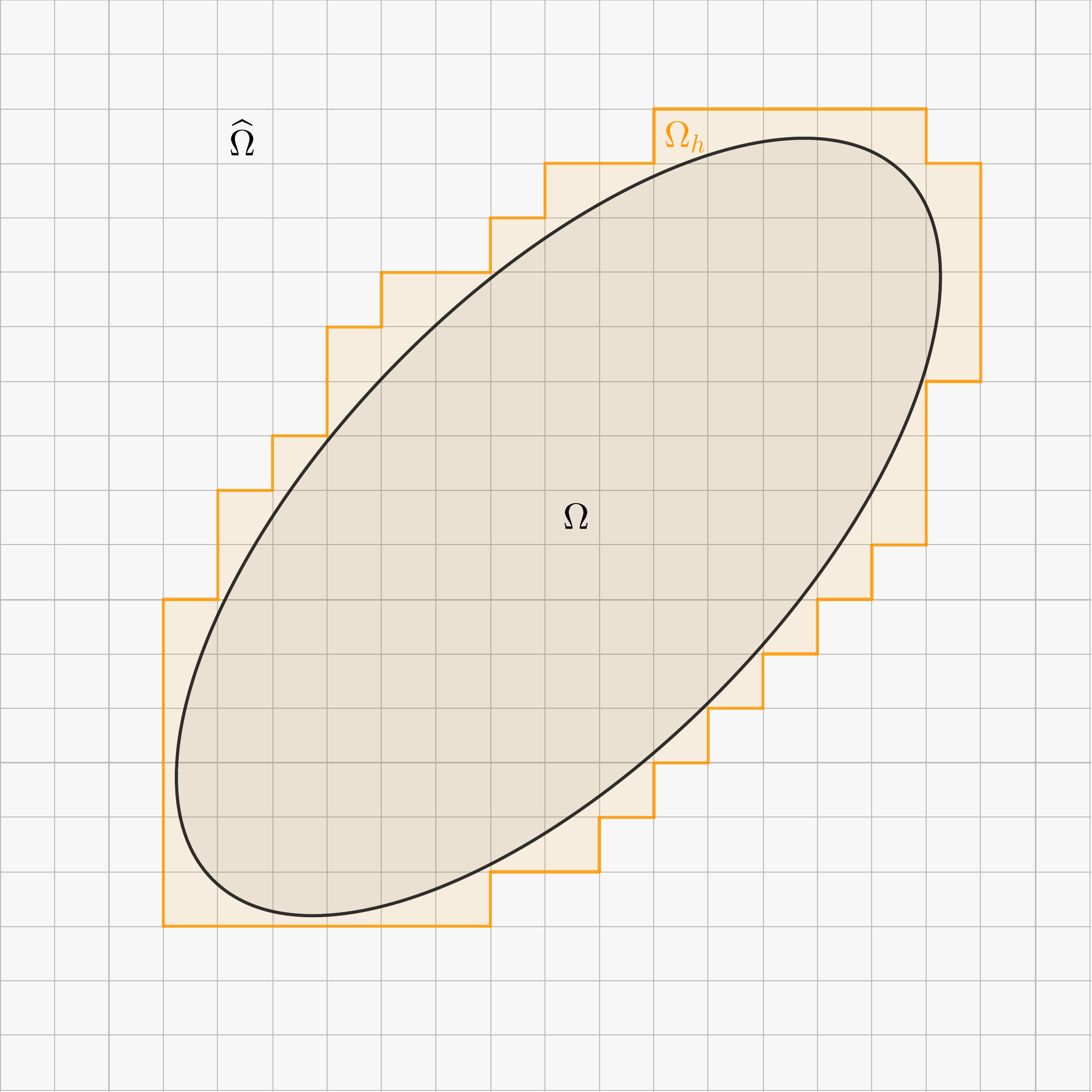}
    \caption{Fictitious domain $\widehat{\Omega}$ (light gray), physical domain $\Omega$ (dark gray) and computational domain $\Omega_h$ (orange)}
    \label{fig: trimming_notation_1}
     \end{subfigure}
     \hfill
     \begin{subfigure}[t]{0.48\textwidth}
    \centering
    \includegraphics[width=\textwidth]{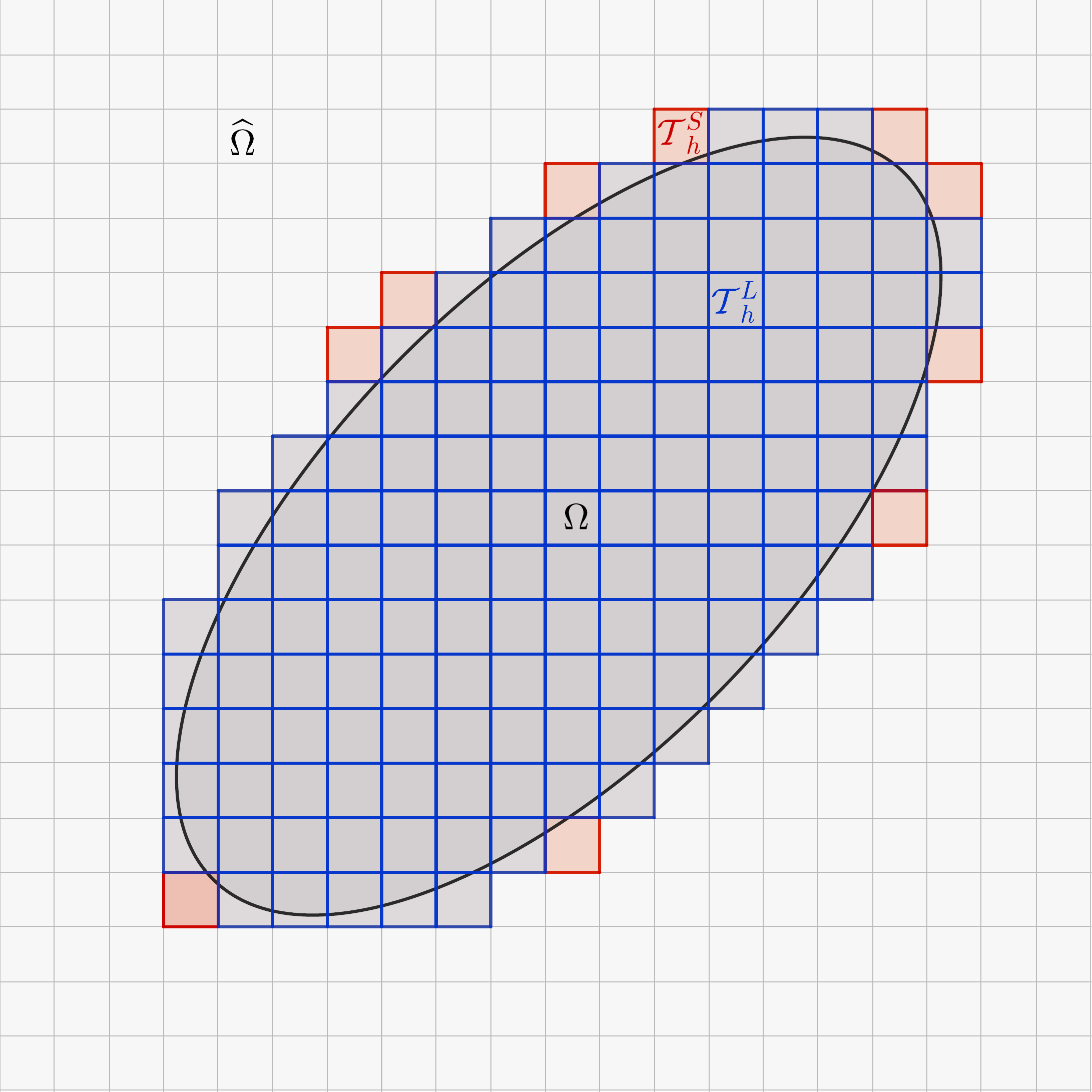}
    \caption{Partitioning of the active mesh $\mathcal{T}_h$ in good elements $\mathcal{T}_h^L$ (blue) and bad elements $\mathcal{T}_h^S$ (red) for a small value of $\gamma$}
    \label{fig: trimming_notation_2}
     \end{subfigure}
     \hfill
    \caption{Trimmed domain}
    \label{fig: trimming_notation_1_2}
\end{figure}

The next step is to locally substitute the B-splines on bad elements by their polynomial extension from a good neighboring element. More specifically, if the boundary is sufficiently regular and the mesh sufficiently fine, there exists a map $S_h \colon \mathcal{T}_h^S \to \mathcal{T}_h^L$, which associates to a bad element $T \in \mathcal{T}_h^S$ a good neighboring element $T' \in \mathcal{T}_h^L$. This good neighboring element may for instance either maximize $|T' \cap \Omega|$ among all neighbors \cite{buffa2020minimal} or minimize the distance between the centers of mass of $|T \cap \Omega|$ and $|T' \cap \Omega|$ \cite{burman2023extension}. Regardless of the strategy chosen, the assembly of system matrices is subsequently locally modified depending on whether the element is good or bad. The assembly of system matrices in IGA still mostly follows standard finite element practice whereby contributions of individual elements are collected and subsequently assembled in global arrays. We recall that there are only $(p+1)^d$ non-vanishing splines on element $T$, whose index set (or connectivity) is defined as
\begin{equation*}
    I_T = \{i \in I \colon T \subseteq \supp(B_i) \}
\end{equation*}
and whose contributions are obtained by simply restricting the integration to the element $T$. For example, for the stiffness and mass matrices defined via the bilinear forms $a,b$ in \eqref{eq: global_contribution}, we define their contribution $a_T, b_T \colon V^h \times V^h \to \mathbb{R}$ to the element $T$ as
\begin{equation}
\label{eq: local_contribution}
    a_T(u^h,v^h) = (\nabla u^h, \nabla v^h)_{L^2(T)^d}, \quad \text{and} \quad b_T(u^h,v^h) = (u^h,v^h)_{L^2(T)}.
\end{equation}
While the assembly on large elements $T \in \mathcal{T}_h^L$ remains unchanged, the assembly on small elements is modified by locally substituting the B-splines living on the small element $T$ by the extension of those living on its good neighbor $T'=S_h(T)$. The (canonical) extension
\begin{equation*}
    \mathcal{E}_{T'} \colon \mathbb{Q}_p(T') \to \mathbb{Q}_p(\Omega)
\end{equation*}
simply extends the polynomial segments on $T'$ to global polynomials on $\Omega$. By construction, this local stabilization technique circumvents the issues tied to basis functions whose support only intersects the physical domain over small trimmed elements. The assembly procedure is exemplarily presented for the stiffness matrix in \Cref{algo: assembly_stiffness}. Adapting it to the mass matrix and right-hand side is straightforward. In the sequel, the stabilized stiffness and mass matrices are denoted $\tilde{K}$ and $\tilde{M}$, respectively, and the latter is typically lumped with any suitable lumping scheme. Note that the block lumping strategies proposed in \cite{voet2024mass} assume a specific block structure encoded in the sparsity pattern of the matrix, whose definition is recalled below.

\begin{definition}[Sparsity pattern]
\label{def: sparsity pattern}
The sparsity pattern of a matrix $A \in \mathbb{R}^{n \times n}$ is the set
\begin{equation*}
    \sparsity(A)=\{(i,j) \colon a_{ij} \neq 0, 1 \leq i,j \leq n\}
\end{equation*}
\end{definition}

Clearly, the assembly procedure in \Cref{algo: assembly_stiffness} alters the sparsity pattern of the system matrices but one easily sees that the number of nonzero entries never increases. 

\begin{lemma}
For the stabilized stiffness and mass matrices assembled in \Cref{algo: assembly_stiffness},
\begin{equation*}
    \sparsity(\tilde{K}) \subseteq \sparsity(K), \quad \text{and} \quad \sparsity(\tilde{M}) \subseteq \sparsity(M).
\end{equation*}
\end{lemma}
\begin{proof}
By inspecting \Cref{algo: assembly_stiffness}, 
\begin{equation*}
    \sparsity(\tilde{K}) = \bigcup_{T \in \mathcal{T}_h^L} I_T \times I_T \subseteq \bigcup_{T \in \mathcal{T}_h} I_T \times I_T = \sparsity(K).
\end{equation*}
The first and last equalities follow from the fact that the element matrices are dense and element contributions do not exactly cancel out. The same holds for the mass matrix.
\end{proof}

Thus, we may employ the sparsity pattern of the non-stabilized mass matrix $M$ for lumping its stabilized counterpart, following the technique described in \citep[][Example 5.8]{voet2024mass}. Anyway, regardless of whether the mass matrix is lumped, \Cref{algo: assembly_stiffness} only accounts for so-called large degrees of freedom
\begin{equation*}
    I^L = \bigcup_{T \in \mathcal{T}_h^L} I_T,
\end{equation*}
while small degrees of freedom $I^S = I \setminus I^L$ are left out, which effectively produces zero rows and columns at indices $i,j \in I^S$. Those are subsequently removed in a post-processing step.

\begin{algorithm}[H]
\begin{algorithmic}[1]
\caption{Assembly of the stabilized stiffness matrix}
\label{algo: assembly_stiffness}
\Statex \textbf{Input}: Parameter $\gamma$
\State $[\mathcal{T}_h^L, \mathcal{T}_h^S] = \texttt{partition}(\mathcal{T}_h, \gamma)$ \Comment{Partition the mesh in large and small elements}
\State Initialize $\tilde{K}$
\For{$T \in \mathcal{T}_h$}
    \If{$T \in \mathcal{T}_h^L$}
    \State $\tilde{K}_{ij} \gets \tilde{K}_{ij} + a_T(B_i,B_j) \quad i,j \in I_T$ \Comment{Standard assembly}
    \Else
    \State $T'=S_h(T)$ \Comment{Get good neighboring element}
    \State $\tilde{K}_{ij} \gets \tilde{K}_{ij} + a_T(\mathcal{E}_{T'}(B_i),\mathcal{E}_{T'}(B_j)) \quad i,j \in I_{T'}$ \Comment{Modified assembly}
    \EndIf
\EndFor
\end{algorithmic}
\end{algorithm}

This technique boils down to a non-conforming Galerkin method on a modified space. The approximation space is the space of splines defined on $\Omega^L = \cup_{T \in \mathcal{T}_h^L} T$ which are then extended locally to the small elements in $\mathcal{T}_h^S$. In 1D, if $\gamma$ is small enough, this technique evidently reduces to extended B-splines (see e.g. \cite{hollig2003finite,marussig2017stable} and \Cref{fig: 1D_poly_ext}). However, the approximation space is generally non-conforming and in 2D a priori independent extensions create discontinuities between neighboring elements (see \Cref{fig: 2D_poly_ext}). Consequently, the resulting approximation space is generally not a subspace of $\mathcal{S}_{h,p,k}$, which precludes forming the mass and stiffness matrices by a simple projection. This is not a particular drawback, especially not for nonlinear problems, given that the stiffness matrix is rarely assembled in practice. Alternatively, one may combine the extension with a projection step to restore conformity (see e.g. \cite{burman2023extension}). In either case, the dimension of the modified space is generally smaller than the original one owing to the substitution of basis functions only supported on small trimmed elements. Nevertheless, the resulting space still has good approximation properties, even without the projection \cite{burman2023extension}.

\begin{figure}[H]
     \centering
     \begin{subfigure}[t]{0.48\textwidth}
    \centering
    \includegraphics[width=\textwidth]{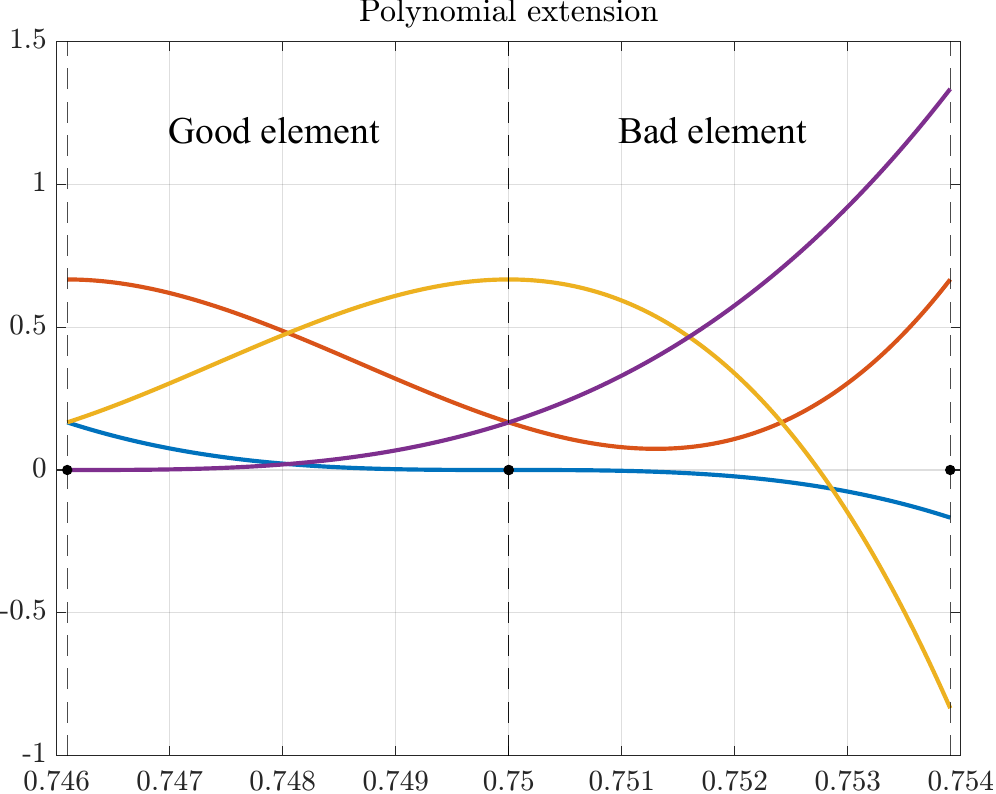}
    \caption{1D}
    \label{fig: 1D_poly_ext}
     \end{subfigure}
     \hfill
     \begin{subfigure}[t]{0.48\textwidth}
    \centering
    \includegraphics[width=\textwidth]{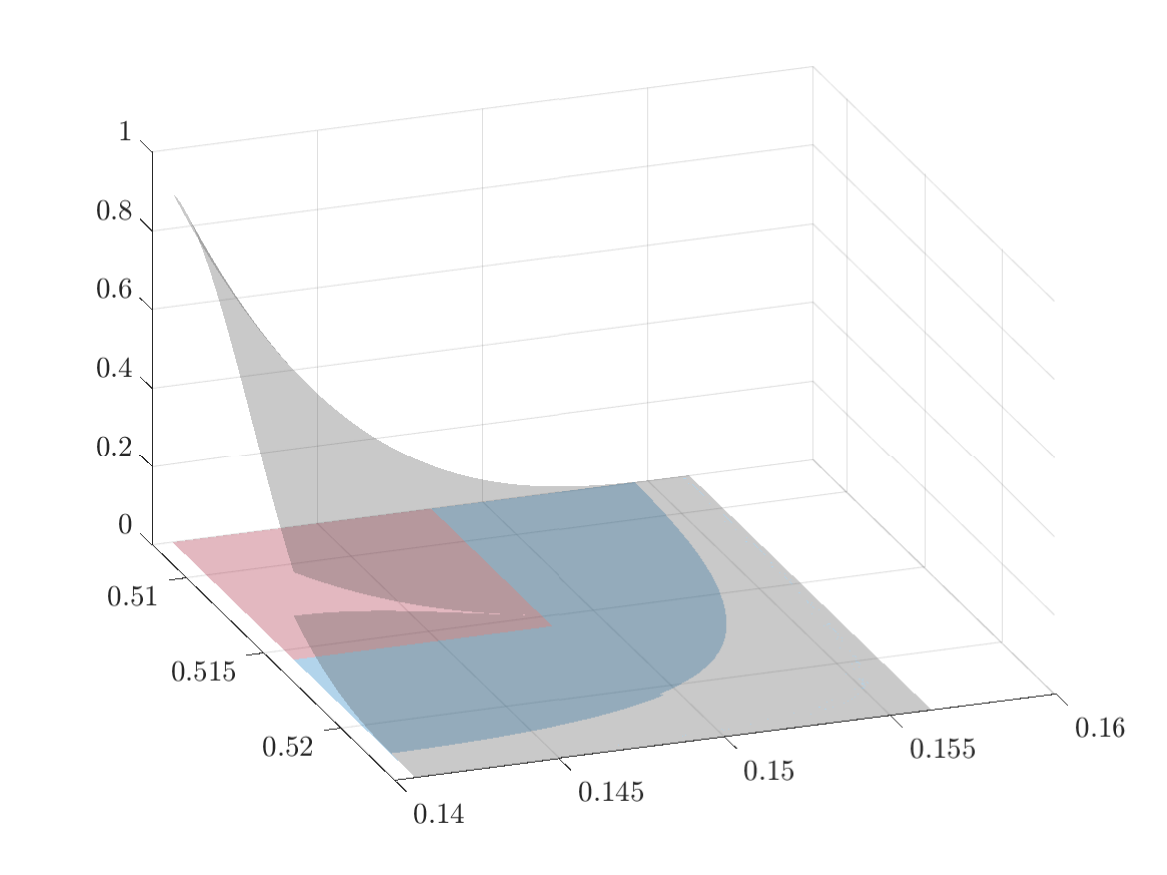}
    \caption{2D}
    \label{fig: 2D_poly_ext}
     \end{subfigure}
     \hfill
    \caption{Polynomial extension of cubic $C^2$ B-splines}
    \label{fig: poly_ext}
\end{figure}

\begin{remark}
Contrary to the consistent mass for the standard B-spline basis, its stabilized counterpart may feature negative entries and choosing an appropriate lumping strategy is less obvious. In \cite{voet2023mathematical}, the authors proved that the absolute row-sum guaranteed an improvement of the CFL, but its effect on the accuracy is unclear. Fortunately, since the extension only takes place on small cut elements from neighboring elements, negative entries are often very small in magnitude and both the standard and absolute row-sum techniques practically produced identical results. Thus, we consistently used the standard row-sum technique \eqref{eq: row_sum} in all our experiments. However, this is generally not a safe practice when stabilizing relatively large trimmed elements (i.e. choosing a large value of $\gamma$).
\end{remark}

\section{Numerical examples}
\label{se: experiments}
Combining mass lumping with stabilization may vastly improve its accuracy without compromising its CFL. As a matter of fact, stabilization restores a level of accuracy comparable to boundary-fitted discretizations. We provide strong numerical evidence for this claim by successfully treating the failed examples of \Cref{se: motivation}. In all our examples, we form the stabilized stiffness and mass matrices by employing \Cref{algo: assembly_stiffness} with $\gamma=0.1$, ensuring large elements are at least $10\%$ within the physical domain. The mass matrix is then lumped with the standard row-sum technique \eqref{eq: row_sum} or any of its generalizations \cite{voet2023mathematical,voet2024mass} if need be.

\begin{example}
\label{ex: counter_example_1D_poly_ext}
Let us come back to \Cref{ex: counter_example_1D}, which we now solve with the stabilization technique. The problem data and discretization parameters remain unchanged. \Cref{fig: 1D_Laplace_trimming_normalized_spectrum_eps_1e-6_poly_ext_buffa} shows the normalized spectrum for the stabilized consistent and lumped mass approximations. The stabilization technique combined with the lumped mass successfully removes spurious eigenvalues from the low-frequency spectrum without increasing the outlier eigenvalues. As a matter of fact, the figure is nearly identical to \Cref{fig: 1D_Laplace_trimming_normalized_spectrum_eps_1e-6_relabeled} (except for $p=1$) and is reminiscent of boundary-fitted discretizations (see e.g. \cite{cottrell2006isogeometric}). Additionally, the stabilization preserves the accuracy of the consistent mass while it removes the dependency of the largest eigenvalues on the size of the trimmed element. The latter now actually behave similarly to standard outlier eigenvalues for boundary-fitted discretizations. This observation is another clear indication that small spurious eigenvalues for the lumped mass have exactly the same origin as large spurious eigenvalues for the consistent mass. The stabilization is a solution to both issues and in particular prevents spurious high-frequency modes from moving towards the low-frequency range when the mass is lumped, as shown in \Cref{fig: 1D_Laplace_trimming_dynamics_p3_eps_1e-6_eigenfunctions_poly_ext_buffa}.

\begin{figure}[H]
    \centering
    \includegraphics[width=0.5\linewidth]{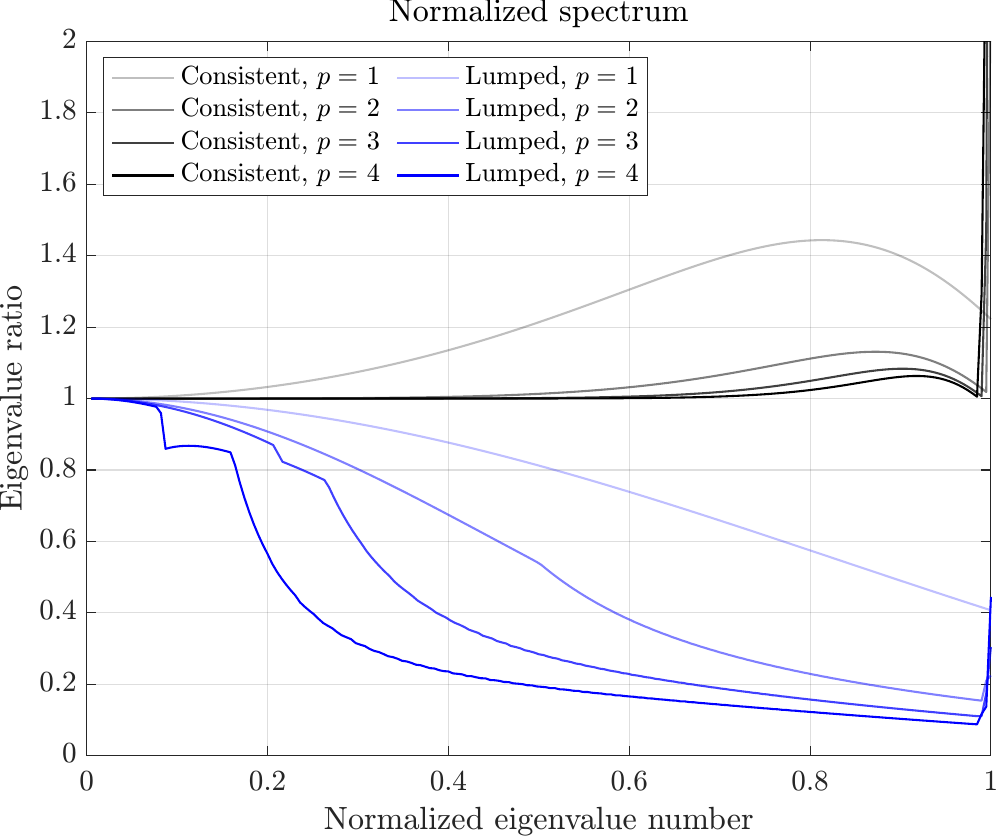}
    \caption{Ratio of approximate over exact eigenvalues for stabilized consistent and lumped mass approximations}
    \label{fig: 1D_Laplace_trimming_normalized_spectrum_eps_1e-6_poly_ext_buffa}
\end{figure}

\begin{figure}[H]
    \centering
    \includegraphics[width=\linewidth]{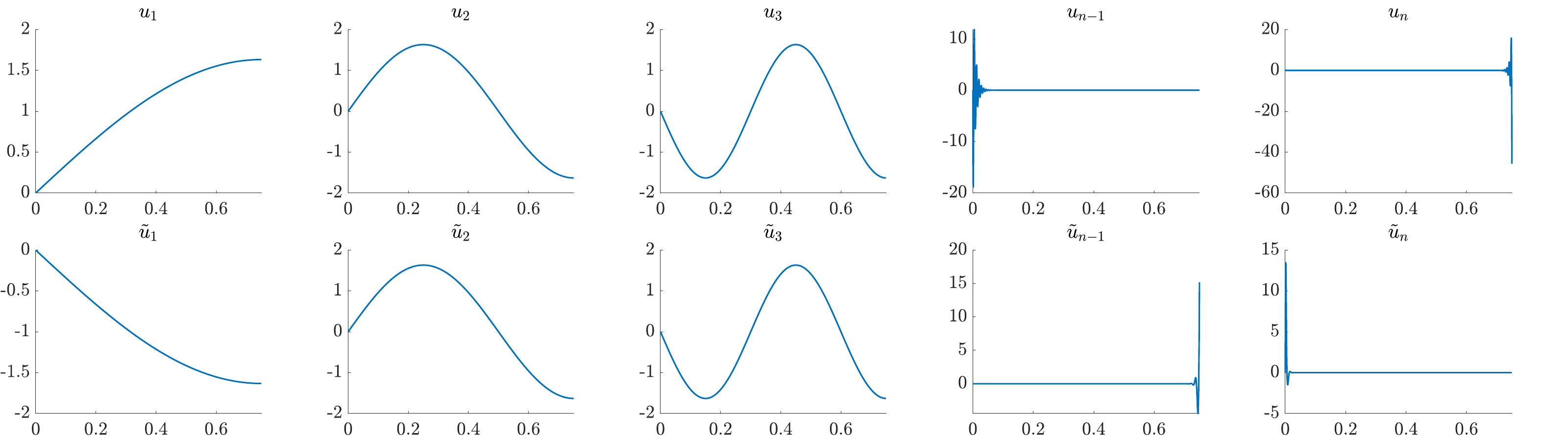}
    \caption{Three first and two last eigenfunctions of $(\tilde{K},\tilde{M})$ (top row) and $(\tilde{K},\mathcal{L}(\tilde{M}))$ (bottom row) for $p=3$. The labeling of eigenfunctions matches the labeling of eigenvalues.}
    \label{fig: 1D_Laplace_trimming_dynamics_p3_eps_1e-6_eigenfunctions_poly_ext_buffa}
\end{figure}

We now solve the same initial boundary-value problem as in \Cref{ex: counter_example_1D} with the same manufactured solution. As shown in \Cref{fig: 1D_Laplace_trimming_dynamics_solution_snapshots_p3_eps_1e-6_Cp_1_explicit_poly_ext_buffa,fig: 1D_Laplace_trimming_dynamics_solution_snapshots_p3_eps_1e-6_C0_explicit_poly_ext_buffa}, the stabilization cures both the $C^2$ and $C^0$ discretizations by completely removing the spurious oscillations from \Cref{fig: 1D_Laplace_trimming_dynamics_solution_snapshots_p3_eps_1e-6_Cp_1_explicit,fig: 1D_Laplace_trimming_dynamics_solution_snapshots_p3_eps_1e-6_C0_explicit}. The solution for the lumped mass now behaves similarly to the consistent mass and the $L^2$ error in \Cref{fig: 1D_Laplace_trimming_dynamics_L2_error_eps_1e-6_explicit_poly_ext_buffa} nearly overlaps.

\begin{figure}[H]
    \centering
    \includegraphics[width=\linewidth]{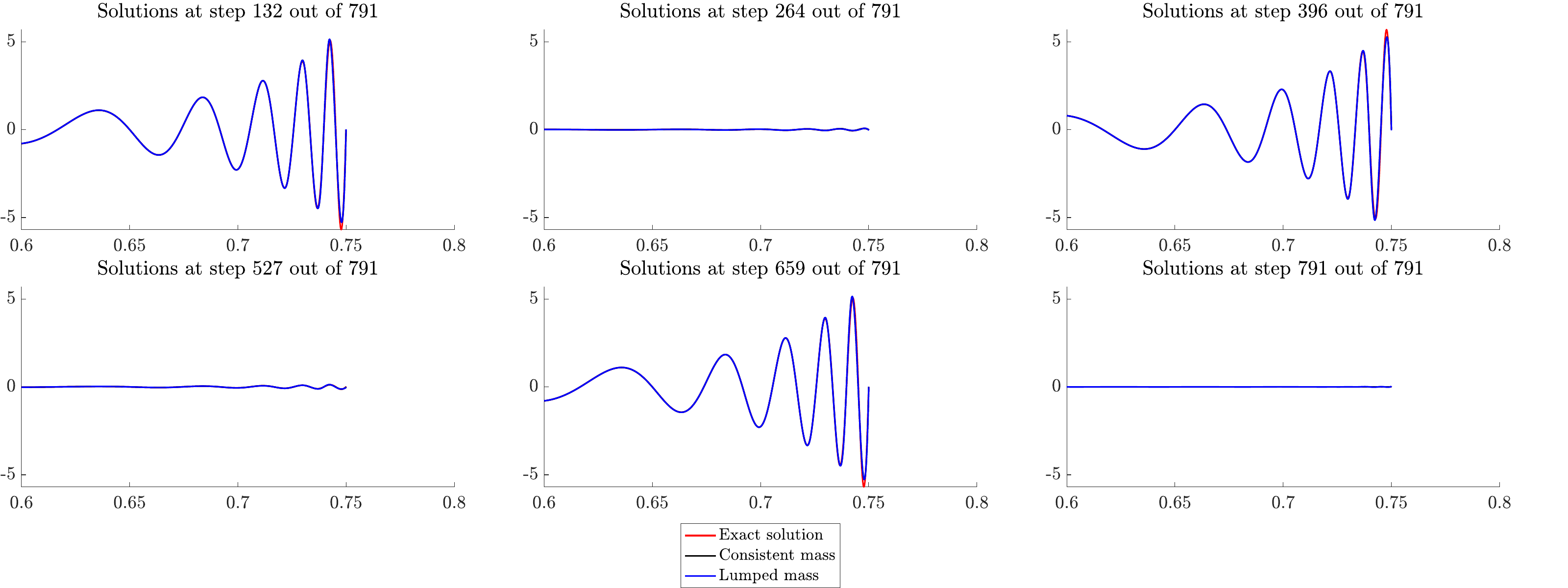}
    \caption{Exact solution and stabilized discrete solutions for the consistent and lumped mass for cubic $C^2$ B-splines}
    \label{fig: 1D_Laplace_trimming_dynamics_solution_snapshots_p3_eps_1e-6_Cp_1_explicit_poly_ext_buffa}
\end{figure}

\begin{figure}[H]
    \centering
    \includegraphics[width=\linewidth]{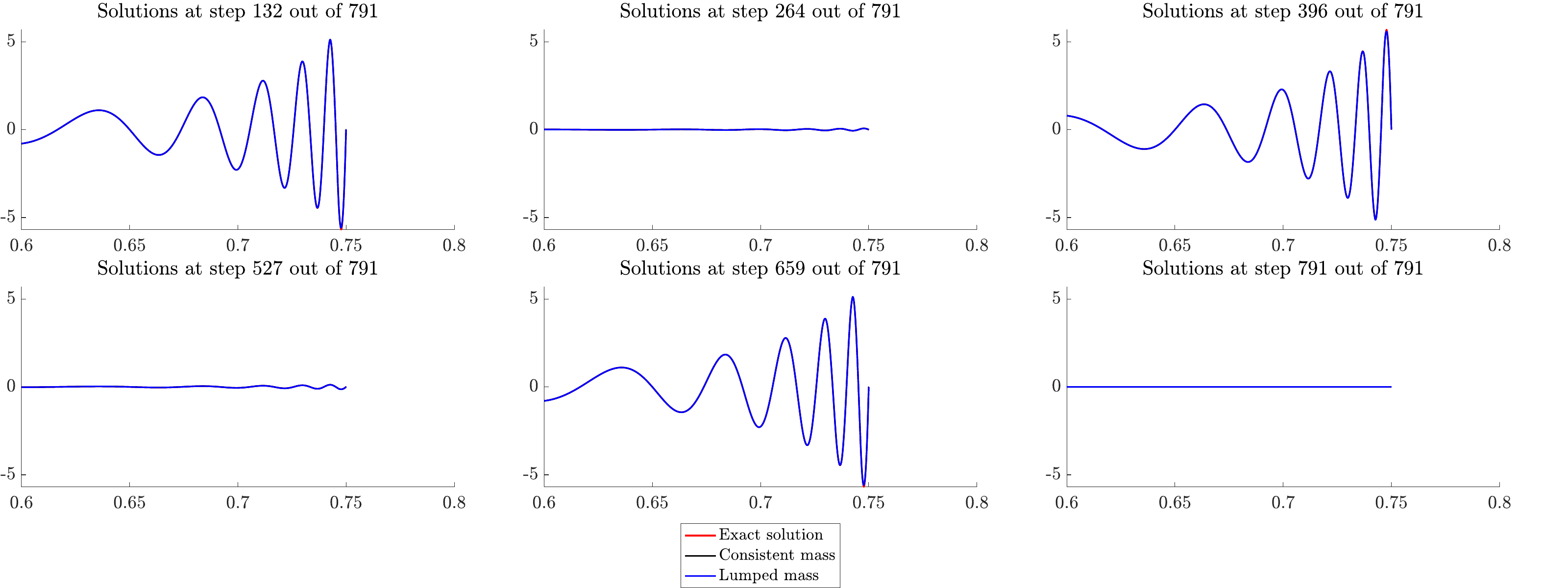}
    \caption{Exact solution and stabilized discrete solutions for the consistent and lumped mass for cubic $C^0$ B-splines}
    \label{fig: 1D_Laplace_trimming_dynamics_solution_snapshots_p3_eps_1e-6_C0_explicit_poly_ext_buffa}
\end{figure}

\begin{figure}[H]
    \centering
    \includegraphics[width=0.5\linewidth]{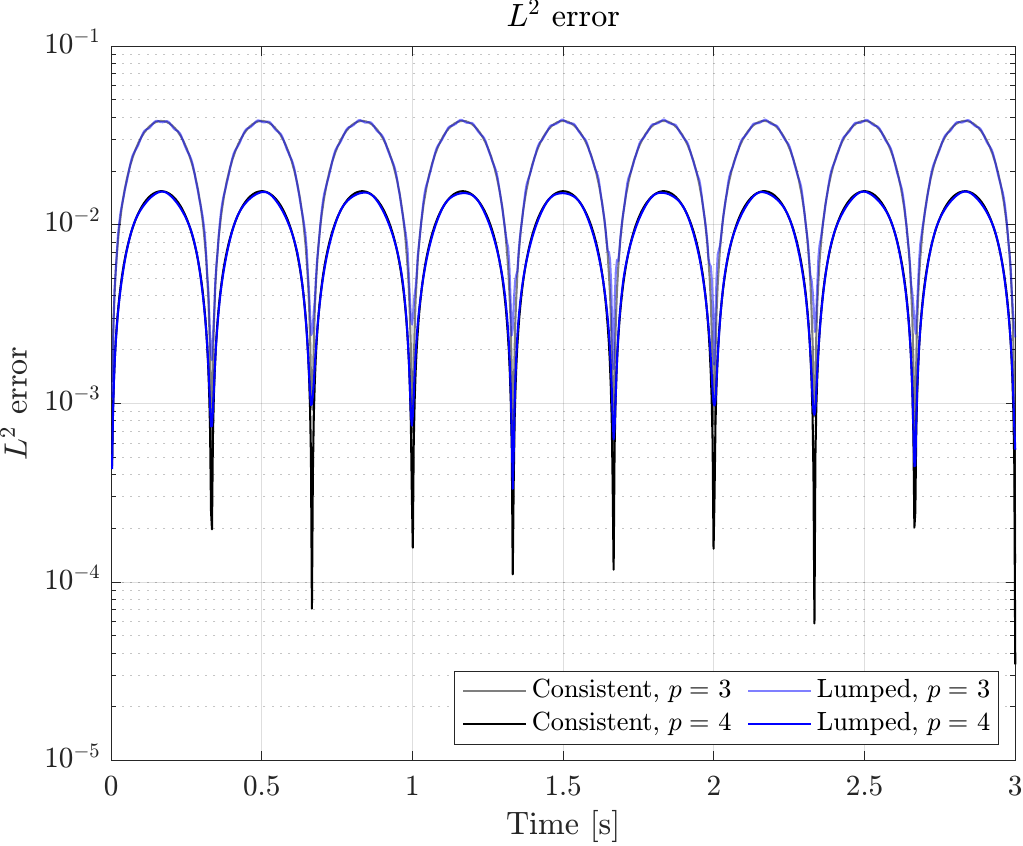}
    \caption{Evolution of the $L^2$ error over time for the stabilized consistent and lumped mass solutions for $C^{p-1}$ discretizations of degree $p=3,4$}
    \label{fig: 1D_Laplace_trimming_dynamics_L2_error_eps_1e-6_explicit_poly_ext_buffa}
\end{figure}

\end{example}

\begin{example}
\label{ex: counter_example_2D_poly_ext}
We now return to \Cref{ex: counter_example_rotated_square} for the rotated and shifted square. There again, the solution for the lumped mass behaved very poorly, especially in the corners. We now employ the stabilization technique of \Cref{se: stabilization} on this same example. Snapshots of the stabilized solutions are shown in \Cref{fig: 2D_Laplace_trimmed_rotated_square_dynamics_solution_snapshots_p3_rot_0_85_poly_ext_buffa_gamma_0_1}. Once again, spurious oscillations for the solution with the lumped mass have disappeared. Although minor discrepancies are still visible in \Cref{fig: 2D_Laplace_trimmed_rotated_square_dynamics_solution_snapshots_p3_rot_0_85_s135_poly_ext_buffa_gamma_0_1}, the magnitude of the solution is also rather small at this time and the error committed in this region of the domain barely contributes to the overall error. Indeed, the $L^2$ error for the lumped mass only slightly departs from the consistent mass in \Cref{fig: 2D_Laplace_trimmed_rotated_square_L2_error_explicit_p3_rot_0_85_poly_ext_buffa_gamma_0_1} and these slight inaccuracies are easily resolved with advanced mass lumping techniques, as discussed in \cite{voet2024mass}.

\begin{figure}[H]
     \centering
     \begin{subfigure}[t]{1.0\textwidth}
    \centering
    \includegraphics[width=\textwidth]{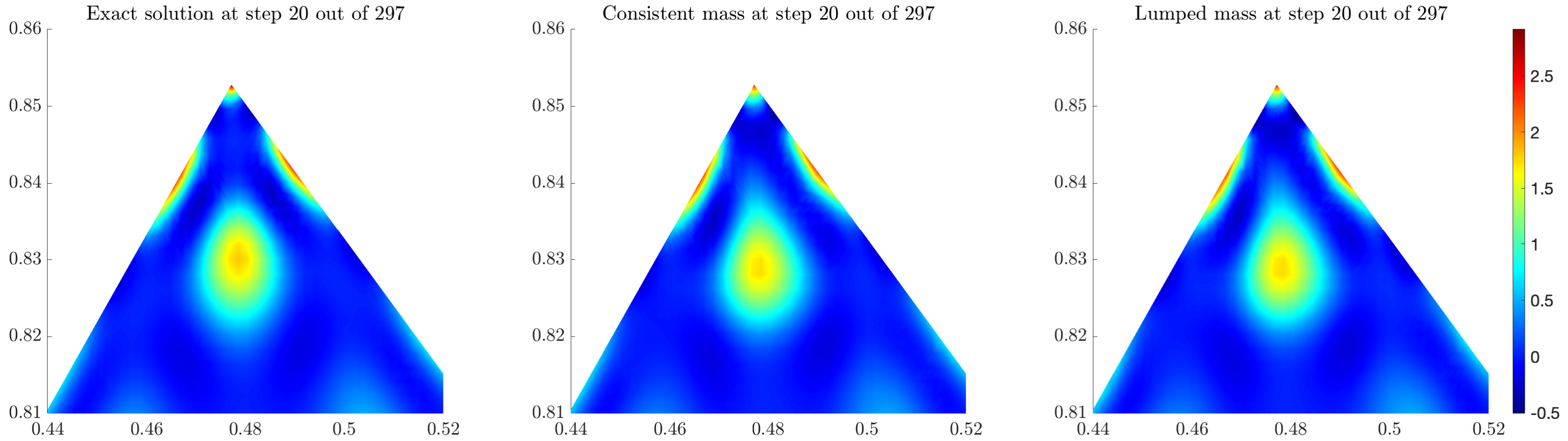}
    \caption{Solution snapshots at time $t=0.1926$}
    \label{fig: 2D_Laplace_trimmed_rotated_square_dynamics_solution_snapshots_p3_rot_0_85_s20_poly_ext_buffa_gamma_0_1}
     \end{subfigure}
     \hfill
     \begin{subfigure}[t]{1.0\textwidth}
    \centering
    \includegraphics[width=\textwidth]{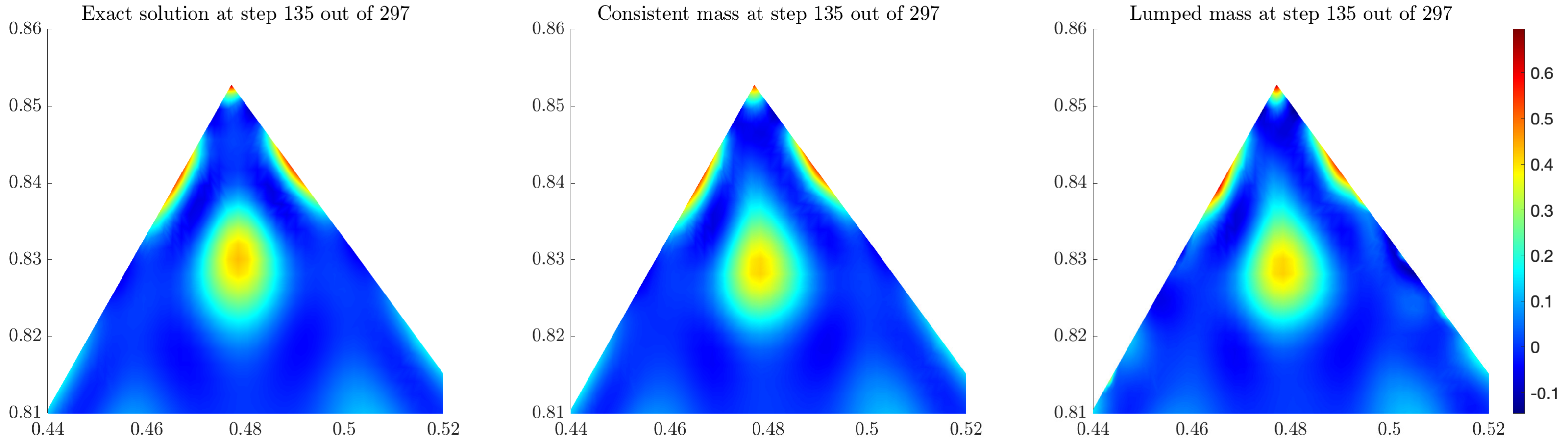}
    \caption{Solution snapshots at time $t=1.3581$}
    \label{fig: 2D_Laplace_trimmed_rotated_square_dynamics_solution_snapshots_p3_rot_0_85_s135_poly_ext_buffa_gamma_0_1}
     \end{subfigure}
     \hfill
    \caption{Exact solution and stabilized discrete solutions for the consistent and lumped mass for $p=3$}
    \label{fig: 2D_Laplace_trimmed_rotated_square_dynamics_solution_snapshots_p3_rot_0_85_poly_ext_buffa_gamma_0_1}
\end{figure}

\begin{figure}[H]
    \centering
    \includegraphics[width=0.5\linewidth]{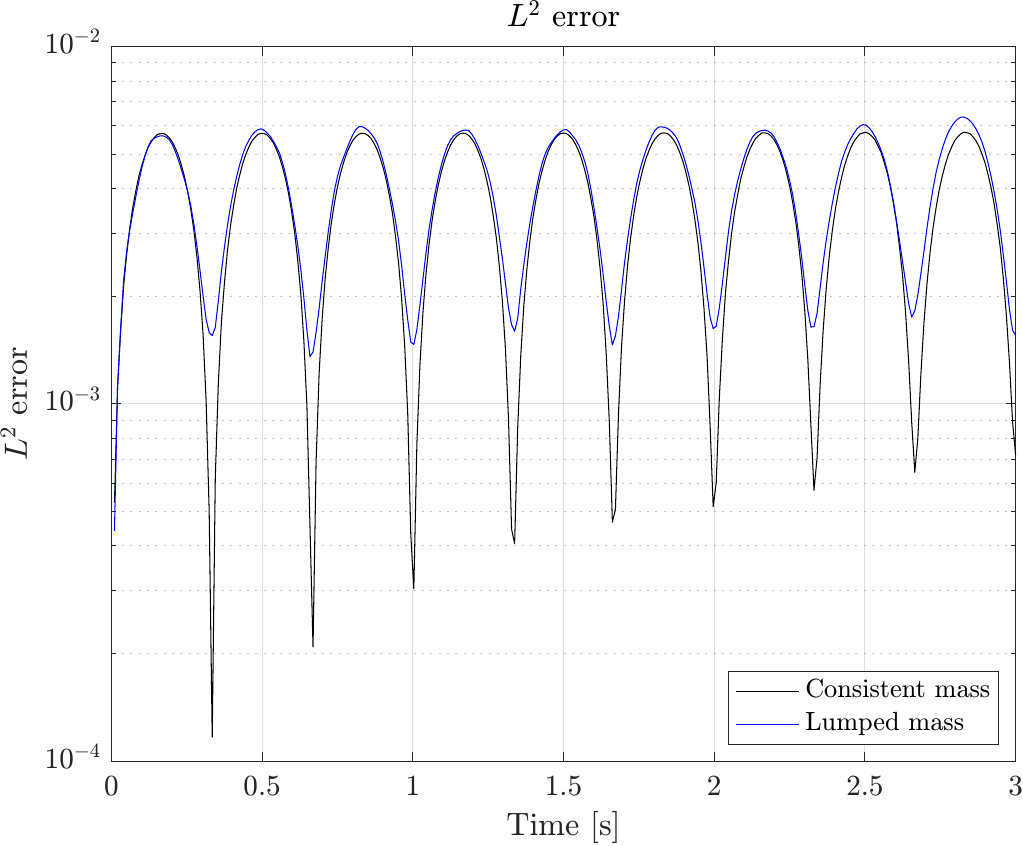}
    \caption{Evolution of the $L^2$ error over time for the stabilized consistent and lumped mass solutions}
    \label{fig: 2D_Laplace_trimmed_rotated_square_L2_error_explicit_p3_rot_0_85_poly_ext_buffa_gamma_0_1}
\end{figure}

\end{example}

\begin{example}
Let us now come back to \Cref{ex: counter_example_plate_with_extrusion}, which we solve with the stabilization technique presented in \Cref{se: stabilization}. All other parameters remain unchanged. Snapshots of the stabilized solution are shown in \Cref{fig: 2D_Laplace_trimmed_plate_with_extrusion_dynamics_solution_snapshots_p2_eps_1e-7_poly_ext_buffa_gamma_0_1} and reveal that the solution for the lumped mass now behaves similarly to the consistent mass. As testified in \Cref{fig: 2D_Laplace_trimmed_plate_with_extrusion_L2_error_explicit_p2_eps_1e-7_poly_ext_buffa_gamma_0_1}, high-frequency oscillations have also been eliminated from \Cref{fig: 2D_Laplace_trimmed_plate_with_extrusion_L2_error_explicit_p2_eps_1e-7}.

\begin{figure}[H]
     \centering
     \begin{subfigure}[t]{1.0\textwidth}
    \centering
    \includegraphics[width=\textwidth]{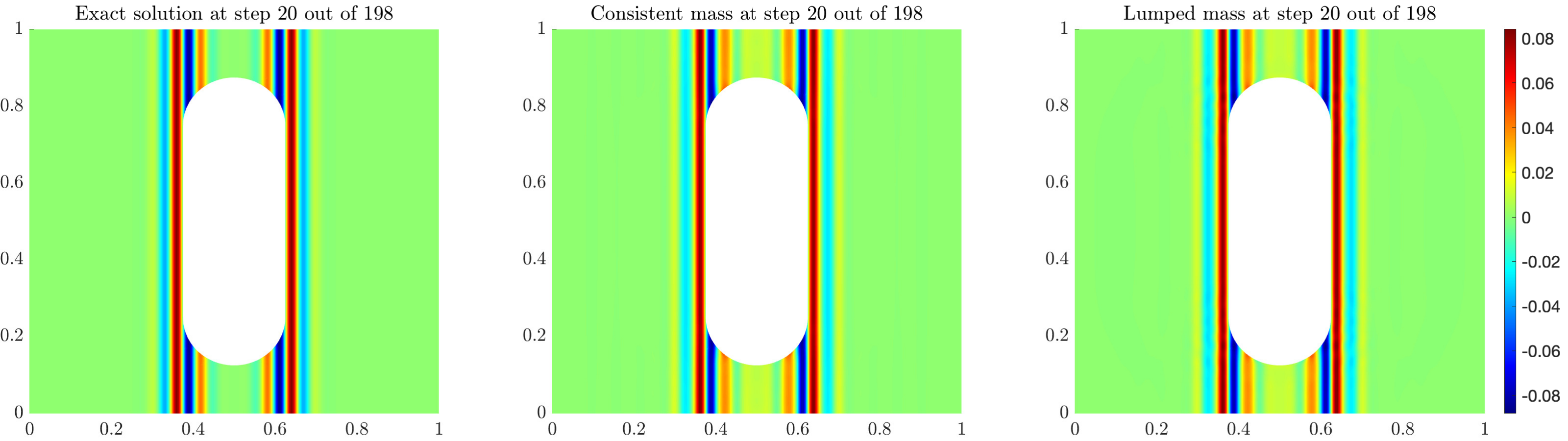}
    \caption{Solution snapshots at time $t=0.1926$}
    \label{fig: 2D_Laplace_trimmed_plate_with_extrusion_dynamics_solution_snapshots_p2_eps_1e-7_s20_poly_ext_buffa_gamma_0_1}
     \end{subfigure}
     \hfill
     \begin{subfigure}[t]{1.0\textwidth}
    \centering
    \includegraphics[width=\textwidth]{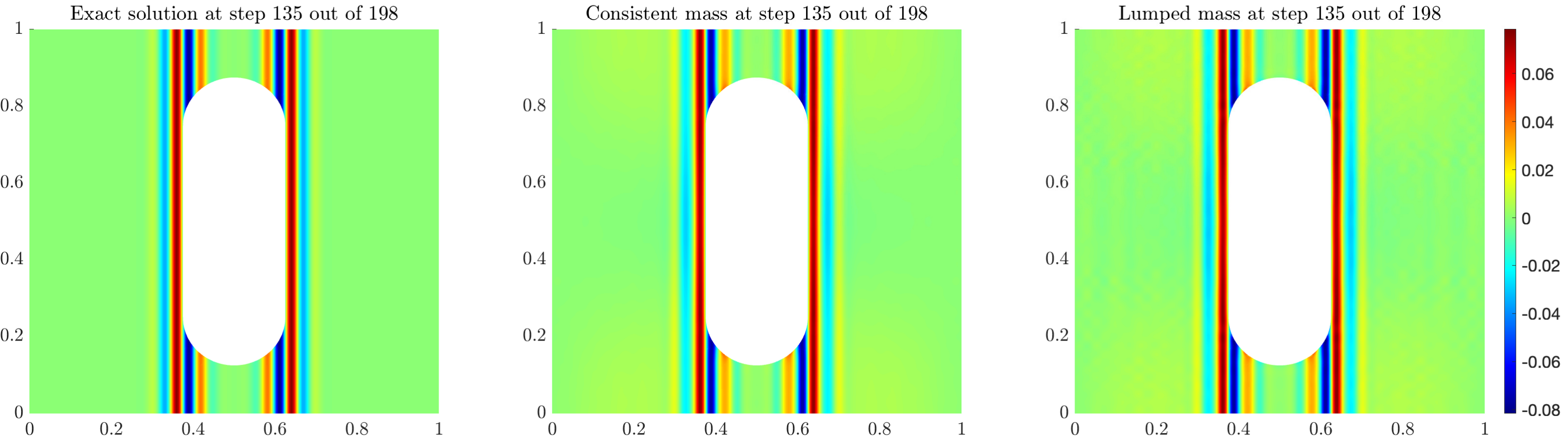}
    \caption{Solution snapshots at time $t=1.3581$}
    \label{fig: 2D_Laplace_trimmed_plate_with_extrusion_dynamics_solution_snapshots_p2_eps_1e-7_s135_poly_ext_buffa_gamma_0_1}
     \end{subfigure}
     \hfill
    \caption{Exact solution and stabilized discrete solutions for the consistent and lumped mass for quadratic $C^1$ B-splines}
    \label{fig: 2D_Laplace_trimmed_plate_with_extrusion_dynamics_solution_snapshots_p2_eps_1e-7_poly_ext_buffa_gamma_0_1}
\end{figure}

\begin{figure}[H]
    \centering
    \includegraphics[width=0.5\linewidth]{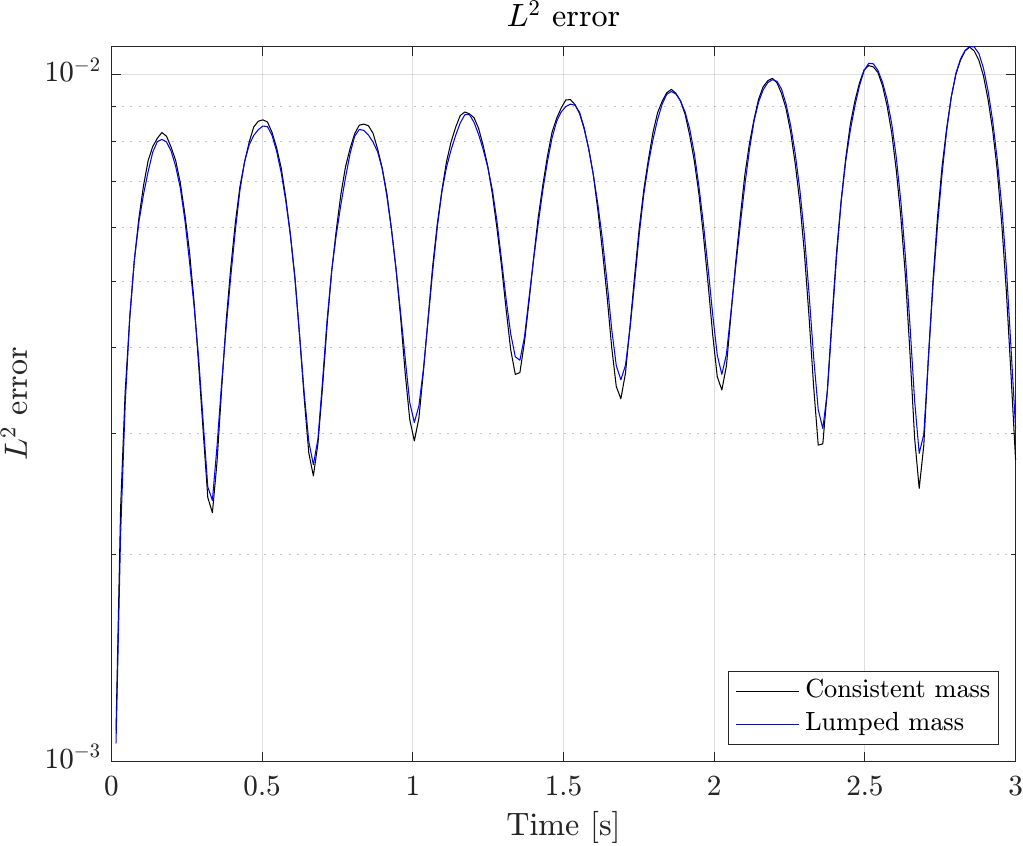}
    \caption{Evolution of the $L^2$ error over time for the stabilized consistent and lumped mass solutions}
    \label{fig: 2D_Laplace_trimmed_plate_with_extrusion_L2_error_explicit_p2_eps_1e-7_poly_ext_buffa_gamma_0_1}
\end{figure}

\end{example}

\begin{example}
Finally, we solve \Cref{ex: counter_example_plate_with_hole} with the same stabilization technique. Contrary to the previous examples, the row-sum technique is still rather inaccurate in this case, even with the stabilization. This is not surprising given that the row-sum is only second order accurate with respect to the mesh size and numerical evidence indicates that the constant in the convergence rate scales poorly with the spline order \cite{cottrell2006isogeometric}. The family of (block) lumped mass matrices devised in \cite{voet2023mathematical,voet2024mass} reduces this constant by increasing the (block) bandwidth of the matrix. Although it does not improve the convergence rate, it is the only improved lumping strategy we are aware of for immersogeometric analysis. According to \Cref{fig: 2D_Laplace_trimmed_plate_with_hole_dynamics_solution_snapshots_p3_eps_1e-6_poly_ext_buffa_gamma_0_1}, the block-diagonal approximation from \cite{voet2024mass} (therein referred to as $\mathcal{P}_1$) produces satisfactory results, as \Cref{fig: 2D_Laplace_trimmed_plate_with_hole_L2_error_explicit_p3_eps_1e-6_poly_ext_buffa_gamma_0_1} confirms. However, since the improved lumping strategies better approximate the stabilized consistent mass, they also tighten the constraint on the CFL condition: $1244$ time-steps were necessary for the block-diagonal approximation instead of $304$ for the row-sum.

\begin{figure}[H]
     \centering
     \begin{subfigure}[t]{1.0\textwidth}
    \centering
    \includegraphics[width=\textwidth]{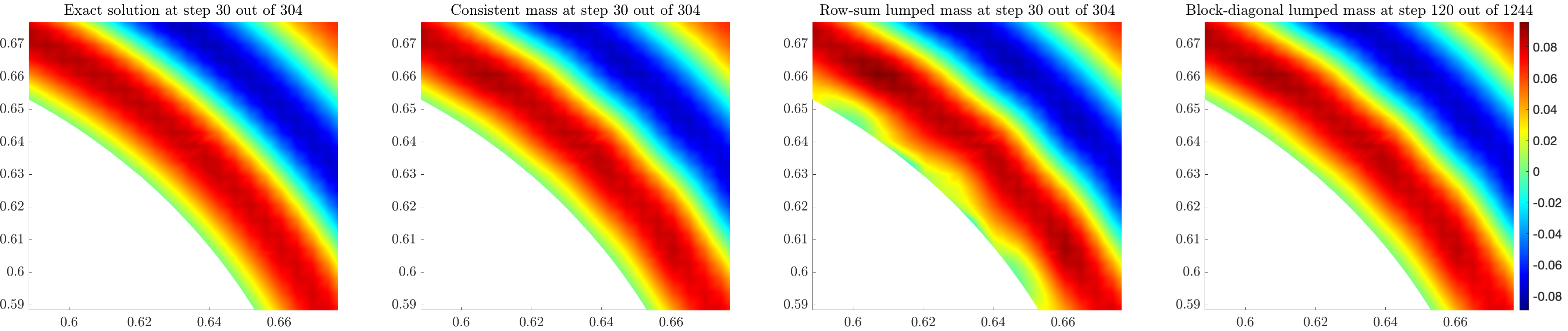}
    \caption{Solution snapshots at time $t=0.2871$}
    \label{fig: 2D_Laplace_trimmed_plate_with_hole_dynamics_solution_snapshots_p3_eps_1e-6_s30_poly_ext_buffa_gamma_0_1}
     \end{subfigure}
     \hfill
     \begin{subfigure}[t]{1.0\textwidth}
    \centering
    \includegraphics[width=\textwidth]{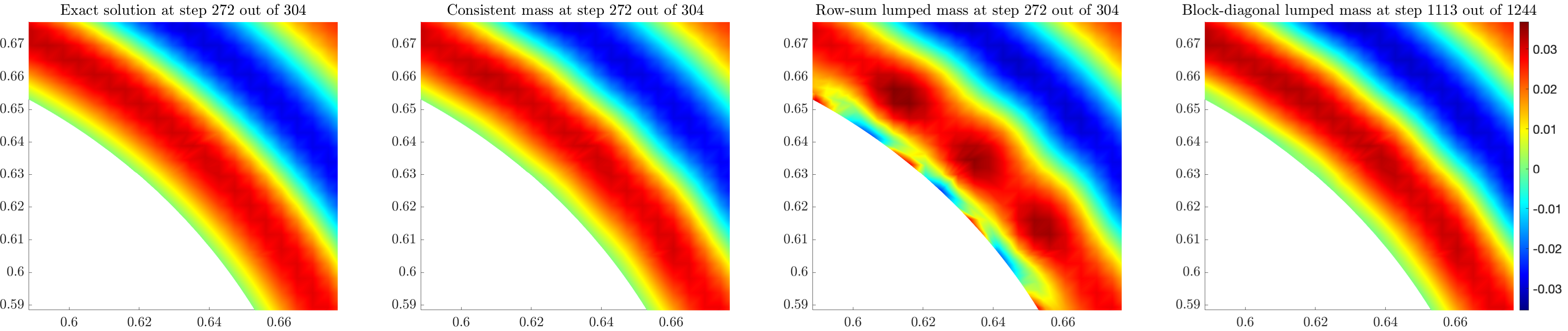}
    \caption{Solution snapshots at time $t=2.6832$}
    \label{fig: 2D_Laplace_trimmed_plate_with_hole_dynamics_solution_snapshots_p3_eps_1e-6_s272_poly_ext_buffa_gamma_0_1}
     \end{subfigure}
     \hfill
    \caption{Exact solution and stabilized discrete solutions for the consistent and lumped mass for cubic $C^2$ B-splines}
    \label{fig: 2D_Laplace_trimmed_plate_with_hole_dynamics_solution_snapshots_p3_eps_1e-6_poly_ext_buffa_gamma_0_1}
\end{figure}

\begin{figure}[H]
    \centering
    \includegraphics[width=0.5\linewidth]{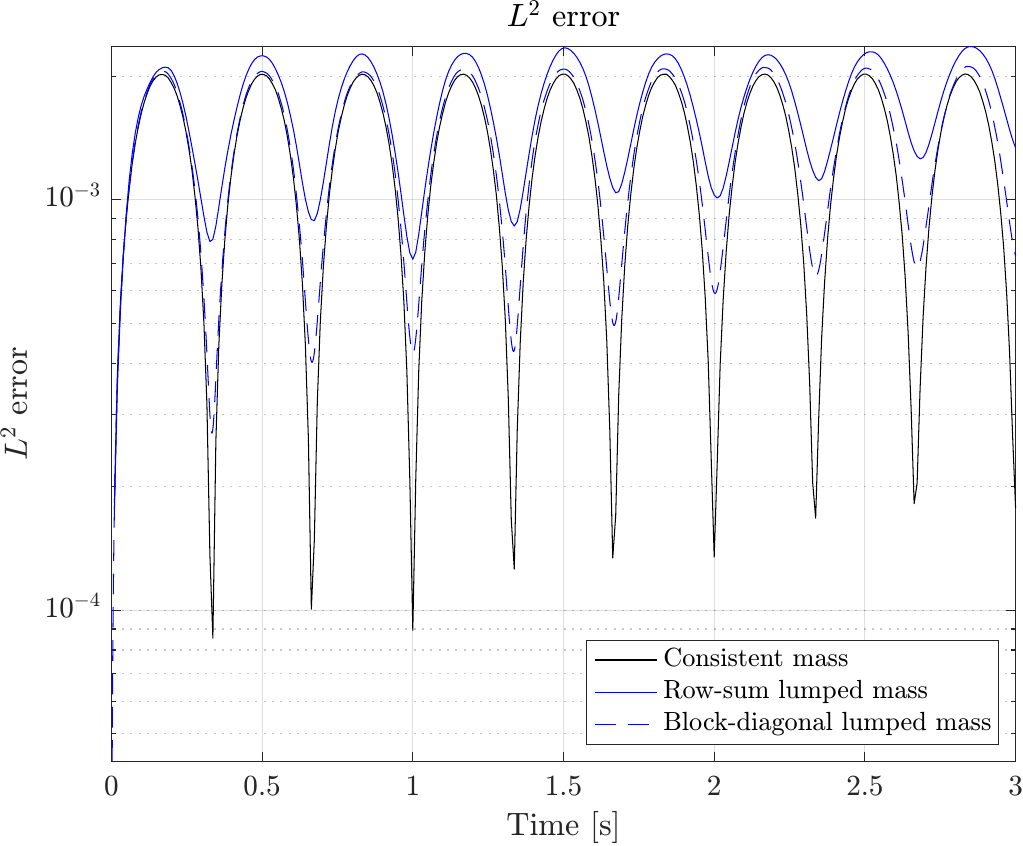}
    \caption{Evolution of the $L^2$ error over time for the stabilized consistent and lumped mass solutions}
    \label{fig: 2D_Laplace_trimmed_plate_with_hole_L2_error_explicit_p3_eps_1e-6_poly_ext_buffa_gamma_0_1}
\end{figure}    
\end{example}

These experiments confirm that the stabilization removes the adverse effect of trimming on the eigenvalues. However, as shown in the previous example, spurious eigenmodes may appear in the low-frequency spectrum of the lumped mass for a very different reason. Most importantly, for the row-sum technique, increasing the spline order also brings in spurious eigenmodes and it appears badly trimmed elements only accelerate this effect. While the stabilization takes care of the trimming, it cannot resolve the issue for high order splines on its own. Although (block) generalizations of the row-sum technique \cite{voet2023mathematical,voet2024mass} might alleviate the issue, they are only second order accurate and remain ill-suited for high order splines. If sufficiently general high order lumping techniques become available for trimmed geometries, we could foresee combining them with the stabilization technique presented herein. Alternatively, combining the stabilization technique with outlier-free spaces on the background mesh \cite{sande2019sharp,manni2022application,hiemstra2021removal} might also mitigate the issue.

\section{Conclusion}
\label{se: conclusion}
Mass lumping has long been praised for removing the dependency of the critical time step on the size of trimmed elements for smooth isogeometric discretizations. This fact, sometimes considered a blessing of smoothness, has unfortunately overshadowed a much more subtle effect: although the largest eigenvalues remain bounded, the smallest ones may become arbitrarily small and bring in spurious modes in the low-frequency spectrum. In this article, we have shown how those spurious modes may trigger oscillations in the solution when they are activated. We have then shown how to prevent such oscillations by stabilizing the discrete formulation prior to mass lumping. This stabilization, both applicable to smooth isogeometric discretizations and standard $C^0$ finite elements, restores a level of accuracy and CFL condition comparable to a boundary-fitted discretization. More generally, since spurious high frequencies for the consistent mass and spurious low frequencies for the lumped mass have exactly the same origin, any solution for the former is also a solution for the latter. Thus, we expect alternative techniques based on extended B-splines to yield similar results.

\section*{Declaration of competing interest}
The authors declare that they have no known competing financial interests or personal relationships that could have appeared to influence the work reported in this paper.

\section*{Acknowledgments}
The second author kindly acknowledges the support of the SNSF through the project ``Smoothness in Higher Order Numerical Methods’’ n. TMPFP2\_209868 (SNSF Swiss Postdoctoral Fellowship 2021). 
\newline The third author kindly acknowledges the support of SNSF through the project ``PDE tools for analysis-aware geometry processing in simulation science’’ n. 200021\_215099.

\appendix
\section{Error analysis}
\label{se: error_analysis}
This section refines and simplifies the error analysis first carried out in \cite{hughes2014finite} for boundary-value and initial-value problems. Although the derivation essentially follows the same arguments, repeated triangle inequalities lead to simpler bounds, much easier to interpret. For completeness, we will not only cover the hyperbolic case, but also the elliptic one.

\subsection{Elliptic boundary-value problems}
The weak form of an elliptic boundary-value problem takes a generic form: find $u \colon \Omega \to \mathbb{R}$ such that
\begin{equation}
    a(u,v)=(f,v)_{L^2} \quad \forall v \in V \label{eq: weak_form_elliptic_bv}
\end{equation}
for a symmetric continuous and coercive bilinear form $a \colon V \times V \to \mathbb{R}$ and $f \in L^2(\Omega)$ with $\Omega$ an open connected domain with Lipschitz boundary and $V$ a suitable Sobolev space that depends on the order of the differential equation. We may construct a basis for $V$ by solving the associated generalized eigenvalue problem
\begin{equation}
    a(u,v)=\lambda b(u,v) \quad \forall v \in V \label{eq: gen_eig_problem}
\end{equation}
with $b(u,v)=(u,v)_{L^2}$. Let $\{(\lambda_i, u_i)\}_{i=1}^{\infty}$ denote the eigenvalue/eigenfunction pairs solution of \eqref{eq: gen_eig_problem}. Under the aforementioned properties, all eigenvalues are positive and we may assume they are ordered in increasing algebraic order such that $0<\lambda_1 \leq \lambda_2 \leq \dots$. Moreover, the (normalized) eigenfunctions form an $L^2$-orthonormal basis for $V$ such that
\begin{equation*}
    a(u_i, u_j)=\lambda_i \delta_{ij}, \qquad b(u_i,u_j)=\delta_{ij}.
\end{equation*}
Expanding the solution $u$ in the eigenfunction basis $u(\mathbf{x})=\sum_{j=1}^{\infty} d_j u_j(\mathbf{x})$, substituting it in \eqref{eq: weak_form_elliptic_bv}, testing for $v=u_i$ and using the orthogonality relations, we obtain 
\begin{equation*}
    \lambda_i d_i = (f,u_i) \implies d_i =\frac{f_i}{\lambda_i}
\end{equation*}
where we defined $f_i=(f,u_i)_{L^2}$, the $i$th coefficient for the $L^2$ projection of $f$ in the eigenbasis. Finally, we obtain
\begin{equation*}
    u(\mathbf{x})=\sum_{j=1}^{\infty} \frac{f_j}{\lambda_j} u_j(\mathbf{x}).
\end{equation*}

For a conforming Galerkin discretization, one seeks an approximate solution $u^h$ in a finite dimensional subspace $V^h \subset V$. The procedure outlined above still holds in a finite dimensional setting and we obtain
\begin{equation*}
    u^h(\mathbf{x})=\sum_{j=1}^{n} \frac{f^h_j}{\lambda^h_j} u^h_j(\mathbf{x}).
\end{equation*}
where $n = \dim(V^h)$ and all discrete quantities bearing an $h$ superscript are defined analogously to their continuous counterpart. Now our goal is to quantify the error
\begin{equation*}
    u^h(\mathbf{x})-u(\mathbf{x})=\sum_{j=1}^{n} d^h_j u^h_j(\mathbf{x})-\sum_{j=1}^{\infty} d_j u_j(\mathbf{x})=e(\mathbf{x}) - \eta(\mathbf{x})
\end{equation*}
where 
\begin{equation*}
    e(\mathbf{x})= \sum_{j=1}^n e_j(\mathbf{x}) \quad \text{and} \quad \eta(\mathbf{x})=\sum_{j=n+1}^{\infty}d_ju_j(\mathbf{x})
\end{equation*}
and $e_j(\mathbf{x})=d^h_j u^h_j(\mathbf{x})- d_j u_j(\mathbf{x})$ denotes the $j$th modal error. The Galerkin approximation leaves an unresolved error component $\eta(\mathbf{x})$ and nothing can be done about it. For the resolved component $e(\mathbf{x})$, the triangle inequality yields
\begin{equation*}
    \|e\|_{L^2} \leq \sum_{j=1}^n\|e_j\|_{L^2}
\end{equation*}
and $\|e_j\|_{L^2}$ is bounded by the eigenvalue and eigenfunction error, as shown in the next theorem.

\begin{theorem}
\label{th: modal_err_elliptic}
The $L^2$ norm of the modal error is bounded by
\begin{equation}
\label{eq: modal_err_elliptic}
    \|e_j\|_{L^2} \leq \frac{\|f\|_{L^2}}{\lambda_j}\left(\frac{\lambda^h_j-\lambda_j}{\lambda_j}+2\|u^h_j-u_j\|_{L^2}\right).
\end{equation}
\end{theorem}
\begin{proof}
After expressing the modal error as
\begin{equation*}
    e_j=\frac{f^h_j}{\lambda^h_j} u^h_j-\frac{f_j}{\lambda_j} u_j=\left(\frac{1}{\lambda^h_j}-\frac{1}{\lambda_j}\right)f^h_ju^h_j +(f^h_j-f_j)\frac{u^h_j}{\lambda_j}+(u^h_j-u_j)\frac{f_j}{\lambda_j}
\end{equation*}
and recalling that $\lambda_i \leq \lambda^h_i$ holds for a conforming approximation (see for instance \cite{ern2004theory}), \cref{eq: modal_err_elliptic} immediately follows from the triangle and Cauchy-Schwarz inequalities, combined with the fact that
\begin{align*}
    \left|\frac{1}{\lambda^h_j}-\frac{1}{\lambda_j}\right|&=\frac{\lambda^h_j-\lambda_j}{\lambda^h_j \lambda_j} \leq \frac{1}{\lambda_j}\frac{\lambda^h_j-\lambda_j}{\lambda_j}, &
    \left|f_j^h-f_j\right|&=\left|(f, u^h_j-u_j)_{L^2}\right| \leq \|f\|_{L^2}\|u^h_j-u_j\|_{L^2},
\end{align*}
and $|f_i| \leq \|f\|_{L^2}$ and $|f^h_i| \leq \|f\|_{L^2}$ hold thanks to the normalization of the eigenfunctions.
\end{proof}

According to \Cref{th: modal_err_elliptic}, the norm of the modal error is bounded by the relative eigenvalue and eigenfunction errors. This upper bound is much simpler than the one derived in \cite{hughes2014finite} and reveals that the approximation errors committed in the higher modes do not contribute much to the overall error thanks to the eigenvalue $\lambda_j$ in the denominator. The authors in \cite{hughes2014finite} reached the same conclusion, although with greater effort. Anyway, neither bounds account for beneficial cancellation of modal errors, which is a clear limitation. It is also possible to derive error bounds in the ``energy norm'' (i.e. the norm induced by $a(u,v)$) but we will not detail them here.

\subsection{Hyperbolic initial boundary-value problems}
Let us now move to the hyperbolic case. From the discussion in \Cref{se: analysis}, we bound the modal error $e_j(\mathbf{x},t)=d^h_j(t)u^h_j(\mathbf{x})-d_j(t)u_j(\mathbf{x})$, where $d_i(t)$ and $d^h_i(t)$ are given by
\begin{equation*}
    d_i(t)=u_{i,0}\cos(\omega_i t)+\frac{v_{i,0}}{\omega_i}\sin(\omega_i t)+\frac{1}{\omega_i}\int_0^t \sin(\omega_i(t-\tau))f_i(\tau) \mathrm{d} \tau \quad i=1,2,\dots, \infty,
\end{equation*}
and
\begin{equation*}
    d^h_i(t)=u^h_{i,0}\cos(\omega^h_i t)+\frac{v^h_{i,0}}{\omega^h_i}\sin(\omega^h_i t)+\frac{1}{\omega^h_i}\int_0^t \sin(\omega^h_i(t-\tau))f^h_i(\tau) \mathrm{d} \tau \quad i=1,2,\dots, n,
\end{equation*}
respectively. We will now prove \Cref{th: modal_error_hyperbolic} recalled below.

\begin{theorem}
\label{th: modal_error_hyperbolic_appendix}
The $L^2$ norm of the modal error is bounded by
\begin{flalign}
\|e_j\|_{L^2} &\leq \|u_0\|_{L^2} \left(2\|u^h_j-u_j\|_{L^2}+|\cos(\omega^h_j t)-\cos(\omega_j t)|\right) & \label{eq: initial_cond1_hyperbolic} \\
 &+\frac{\|v_0\|_{L^2}}{\omega_j}\left(\frac{\omega^h_j-\omega_j}{\omega_j}+2\|u^h_j-u_j\|_{L^2}+|\sin(\omega^h_j t)-\sin(\omega_j t)|\right) & \label{eq: initial_cond2_hyperbolic} \\
 &+\frac{1}{\omega_j} \int_{0}^t \|f(\tau)\|_{L^2} \mathrm{d} \tau \left( \frac{\omega^h_j-\omega_j}{\omega_j}+2\|u^h_j-u_j\|_{L^2} + \max_{\tau \in [0,t]} |\sin(\omega^h_j \tau)-\sin(\omega_j \tau)| \right). & \label{eq: forcing_term_hyperbolic}
\end{flalign}
\end{theorem}
\begin{proof}
For bounding the norm of the modal error $e_j(\mathbf{x},t)=d^h_j(t)u^h_j(\mathbf{x})-d_j(t)u_j(\mathbf{x})$, we first note that
\begin{equation*}
    e_j=(d^h_j-d_j)u^h_j+(u^h_j-u_j)d_j.
\end{equation*}
Thus,
\begin{equation*}
    \|e_j\|_{L^2} \leq |d^h_j-d_j|+|d_j|\|u^h_j-u_j\|_{L^2}.
\end{equation*}
Then, using the fact that $|u_{i,0}| \leq \|u_0\|_{L^2}$ and $|v_{i,0}| \leq \|v_0\|_{L^2}$, we obtain straightforwardly
\begin{equation}
\label{eq: bound_dj_hyperbolic}
    |d_j| \leq \|u_0\|_{L^2} +\frac{\|v_0\|_{L^2}}{\omega_j}+\frac{1}{\omega_j} \int_{0}^t \|f(\tau)\|_{L^2} \mathrm{d} \tau.
\end{equation}
Moreover,
\begin{equation*}
    |d^h_j-d_j| \leq \underbrace{|u^h_{j,0}\cos(\omega^h_j t)-u_{j,0}\cos(\omega_j t)|}_{(\mathrm{I})}+\underbrace{\left|\frac{v^h_{j,0}}{\omega^h_j}\sin(\omega^h_j t)-\frac{v_{j,0}}{\omega_j}\sin(\omega_j t)\right|}_{(\mathrm{II})} + \underbrace{\left|\int_0^t \sin(\omega^h_j(t-\tau))\frac{f^h_j(\tau)}{\omega^h_j} -  \sin(\omega_j(t-\tau))\frac{f_j(\tau)}{\omega_j} \mathrm{d} \tau \right|}_{(\mathrm{III})}.
\end{equation*}
As for the elliptic case, the strategy for bounding each term relies on triangle inequalities. For $(\mathrm{I})$:
\begin{align*}
    |u^h_{j,0}\cos(\omega^h_j t)-u_{j,0}\cos(\omega_j t)| &=|u^h_{j,0}(\cos(\omega^h_j t)-\cos(\omega_j t))+\cos(\omega_j t)(u^h_{j,0}-u_{j,0})| \\
    &\leq \|u_0\|_{L^2}(|\cos(\omega^h_j t)-\cos(\omega_j t)|+\|u^h_j-u_j\|_{L^2}).
\end{align*}
For $(\mathrm{II})$:
\begin{align*}
    \left|\frac{v^h_{j,0}}{\omega^h_j}\sin(\omega^h_j t)-\frac{v_{j,0}}{\omega_j}\sin(\omega_j t)\right| &= \left|v^h_{j,0}\sin(\omega^h_j t)\left(\frac{1}{\omega^h_j}-\frac{1}{\omega_j}\right) + \frac{v^h_{j,0}}{\omega_j}(\sin(\omega^h_j t)-\sin(\omega_j t))+\frac{\sin(\omega_j t)}{\omega_j}(v^h_{j,0}-v_{j,0})\right| \\
    &\leq \frac{\|v_0\|_{L^2}}{\omega_j}\left(\frac{\omega^h_j-\omega_j}{\omega_j}+|\sin(\omega^h_j t)-\sin(\omega_j t)|+\|u^h_j-u_j\|_{L^2}\right).
\end{align*}
Finally, for $(\mathrm{III})$:
\begin{flalign*}
    &\left|\int_0^t \sin(\omega^h_j(t-\tau))\frac{f^h_j(\tau)}{\omega^h_j} -  \sin(\omega_j(t-\tau))\frac{f_j(\tau)}{\omega_j} \mathrm{d} \tau \right| & \\
    &=\left|\int_0^t \sin(\omega^h_j(t-\tau))f^h_j(\tau) \left(\frac{1}{\omega^h_j}-\frac{1}{\omega_j}\right) + \frac{f^h_j(\tau)}{\omega_j}(\sin(\omega^h_j(t-\tau))-\sin(\omega_j(t-\tau)))+\frac{\sin(\omega_j(t-\tau))}{\omega_j}(f^h_j(\tau)-f_j(\tau)) \mathrm{d} \tau \right| \\
    &\leq \frac{1}{\omega_j} \int_{0}^t \|f(\tau)\|_{L^2} \mathrm{d} \tau \left( \frac{\omega^h_j-\omega_j}{\omega_j}+\|u^h_j-u_j\|_{L^2}\right)+ \frac{1}{\omega_j}\int_0^t \|f(\tau)\|_{L^2} |\sin(\omega^h_j(t-\tau))-\sin(\omega_j(t-\tau))| \mathrm{d} \tau \\
    &\leq \frac{1}{\omega_j} \int_{0}^t \|f(\tau)\|_{L^2} \mathrm{d} \tau \left( \frac{\omega^h_j-\omega_j}{\omega_j}+\|u^h_j-u_j\|_{L^2} + \max_{\tau \in [0,t]} |\sin(\omega^h_j \tau)-\sin(\omega_j \tau)| \right).
\end{flalign*}
The statement of \Cref{th: modal_error_hyperbolic_appendix} now follows after combining the bounds for $(\mathrm{I})$, $(\mathrm{II})$, $(\mathrm{III})$ and \eqref{eq: bound_dj_hyperbolic}.
\end{proof}

The three contributions to the modal error are coming from the first \eqref{eq: initial_cond1_hyperbolic} and second \eqref{eq: initial_cond2_hyperbolic} initial conditions and the forcing term \eqref{eq: forcing_term_hyperbolic}. Contrary to the elliptic case, approximation errors in the higher frequencies are not dampened away in \eqref{eq: initial_cond1_hyperbolic} and instead oscillate indefinitely. Once again, the bounds are clearer than those reported in \cite{hughes2014finite}, where the contribution from the right-hand side is not considered.

\end{document}